\documentclass[a4paper]{article}
\usepackage[T1]{fontenc}
\usepackage{times}
\usepackage{amsmath,amsfonts,amssymb,amsthm,epsfig,euscript,pifont,psfrag, color,wrapfig,float,dsfont}
\usepackage{mathrsfs}
\usepackage[english]{babel} 
\usepackage{hyperref}
\usepackage[table]{xcolor}
\usepackage{euscript,graphicx}
\usepackage{graphics}
\usepackage{makeidx}
\usepackage{soul}
\usepackage[normalem]{ulem}
\usepackage{tabularx}
\usepackage{subcaption}
\usepackage{fancyhdr}
\usepackage{xcolor}
\usepackage{float}
\usepackage{graphicx} 

\newcommand{\bRb}{\mathbb{R}}

\usepackage{titlesec}

\numberwithin{equation}{section}

\titleformat{\section}{\normalfont\Large\bfseries}{\thesection.}{1em}{}

\newtheorem{thm}{Theorem}[section]
\newtheorem{prop}[thm]{Proposition}

\newtheorem{lem}[thm]{Lemma}
\newtheorem{rem}[thm]{Remark}

\newcommand{\cCc}{\mathcal{C}}

\newcommand{\cGc}{\mathcal{G}}

\newcommand{\bal}{\begin{align}}
\newcommand{\eal}{\end{align}}
\newcommand{\beq}{\begin{equation}}
\newcommand{\eeq}{\end{equation}}

\newcommand{\ba}{\begin{align*}}
\newcommand{\be}{\begin{equation*}}
\newcommand{\ee}{\end{equation*}}

\newcommand{\EE}{\mathbb{E}}
\newcommand{\PP}{\mathbb{P}}

\newcommand{\dd}{{\rm d}}
\newcommand{\ddt}{{\rm d}t}
\newcommand{\dds}{{\rm d}s}
\newcommand{\ddu}{{\rm d}u}
\newcommand{\ddz}{{\rm d}z}
\newcommand{\ddx}{{\rm d}x}
\newcommand{\ddy}{{\rm d}y}

\newcommand{\U}{\mathbb{U}}
\newcommand{\F}{\mathbb{F}}
\newcommand{\W}{\mathbb{W}}
\newcommand{\m}{m}
\newcommand{\dist}{\mathrm{dist}}
\newcommand{\argmin}{\mathrm{arg min}}
\newcommand{\cl}{c_\ell} 
\newcommand{\cd}{c_d} 
\newcommand{\Rd}{R_d} 
\newcommand{\K}{K} 
\newcommand{\KD}{K_{\Gamma}} 
\newcommand{\h}{{h}}

\begin{document}
\title{Freidlin-Wentzell collision-laws between self-stabilizing diffusions\footnote{The present paper features a completely revised and expanded version (in its framework and refined proof arguments) of our unpublished work: \href{arXiv:2206.04542}{arXiv:2206.04542}, 2022.}}
\author{Jean-Fran\c{c}ois Jabir\footnote{HSE University, Department of Statistics and Data Analysis and Laboratory of Stochastic
		Analysis and its Applications, Pokrovsky Blvd 11, Moscow, Russia Federation. jjabir@hse.ru}
	\and Julian Tugaut\footnote{Universit\'{e} Jean Monnet, Institut Camille Jordan, 23, rue du
		docteur Paul Michelon, CS 82301, 42023 Saint-\'{E}tienne Cedex 2, France . julian.tugaut@univ-st-etienne.fr}}
\date{\today}

\maketitle
\begin{abstract} The present work investigates the asymptotic behaviors at the zero-noise limit of the first near collision-time and first near collision-location between a pair of independent $d\ge 1$-dimensional Brownian-driven self-stabilizing (McKean-Vlasov type) diffusions. These asymptotics are considered in a peculiar setting where the diffusions evolve in a bistable landscape and collisions are only triggered by the combined action of the Brownian noises.
As the Brownian perturbations fade away, we show that the first near collision-time increases at an explicit exponential rate and that related collision-locations persist in specific regions of the space. These results mainly rely on a non-trivial extension of classical Freidlin-Wentzell's exit-time and exit-location estimates into collision estimates, and on expanding the coupling techniques of \cite{Alea,Tugaut2021}. Similar asymptotic behaviors are established for related mean-field interacting particle systems, and  the specific one-dimensional case -- where true collisions can be examined -- is further discussed.
\end{abstract}

\tableofcontents
\section{Introduction}
In this paper, we are interested in estimating the zero-noise limit of mere-approximations (see \eqref{near-collision-time-SS} below for the precise definition) of the first collision-time and first collision-location generated by two general non-reversible self-stabilizing dynamics, $X=(X_t)_{t\ge 0}$
and~$Y=(Y_t)_{t\ge 0}$, of the form:
\begin{subequations}\label{MV}
	\begin{equation}
		\label{MV1}
		X_t=x_1 -\int_0^t\Big(\mathbb{U}\left(X_s\right)+\int \F\left(X_s-x\right)\,\mu^X_s(\ddx)\Big)\dds+\sigma \int_0^t \Gamma(X_s)\,\dd B_s\,,\ 
		\mu^X_t=\text{Law}(X_t)\,,\,t\ge 0\,,
	\end{equation}
	and
	\begin{equation}
		\label{MV2}
		Y_t=x_2 -\int_0^t\Big(\U\left(Y_s\right)+\int
		\F\left(Y_s-y\right)\,\mu^Y_s(\ddy)\Big)\dds+\sigma  \int_0^t \Gamma(Y_s)\,\dd\widetilde B_s\,,\ \mu^Y_t=\text{Law}(Y_t)\,,\,t\ge 0\,.
	\end{equation}
\end{subequations}
The driving noises of the systems are defined by a pair of independent $d$-dimensional Brownian motions $B$ and $\widetilde{B}$, a smooth non-degenerate diffusion matrix $\Gamma$, and a positive factor $\sigma$ representing the noise intensity of the system; the aforementioned zero-noise limit corresponds to the limit $\sigma\downarrow 0$. The vector fields $-\U:\mathbb R^d\rightarrow \mathbb R^d$ and $-\F:\mathbb R^d\rightarrow \mathbb R^d$, common to both dynamics and characterizing their self-stabilizing traits, model respectively a non-convex confining force generating a bistable landscape and a smooth concentrating interaction force  regulating the dynamics at large time (we refer to~\cite{Dawson-83}, \cite{BRTV}, \cite{BRV}, \cite{HIP} for this defining feature of self-stabilizing diffusions).  

\medskip
 
In addition to the specific conditions on $\F$ and the bistable force $\U$ (precised in Assumptions ${\bf (A)}$-$(i)$ to ${\bf (A)}$-$(iv)$ below), the initial conditions $x_1$ and $x_2$ will be assumed to be located in distinctive basins of attractions of $\U$ (${\bf (A)}$-$(v)$). These conditions specifically give rise to a setting of {\it noise-induced collisions} between \eqref{MV1} and \eqref{MV2} in which, at large-time and at the limit $\sigma\downarrow 0$, the force field $-\U$ coerces $X$ and $Y$ to each wanders around one of the attractors (i.e. asymptotically stable points) of $-\U$. The interest for this framework is discussed shortly after, and the prototypical (reversible) \textit{gradient form} $\U=\nabla u$, the quadratic interaction force $\F=\alpha\nabla \|x\|^2/2$ and constant diffusion $\Gamma=I_d$ will often serve to illustrate our results below.  

\medskip

As McKean-Vlasov type dynamics, the systems \eqref{MV1} and \eqref{MV2} further relate to the mean-field limit of the $N$-interacting particle systems:
\begin{subequations}\label{particles}
	\begin{equation}
		\label{MFSP1}
		X_t^{i,N}=x_1 -\int_0^t\Big(\mathbb{U}(X_s^{i,N})+\frac{1}{N}\sum_{j=1}^N\F\big(X_s^{i,N}-X_s^{j,N}\big)\Big)\dds+\sigma\int_0^t
		\Gamma(X_s^{i,N})\dd B_s^i\,,\ \ t\ge 0\,,1\le i\le N\,,
	\end{equation}
	\begin{equation}
		\label{MFSP2}
		Y_t^{i,N}=x_2 -\int_0^t\Big(\mathbb{U}(Y_s^{i,N})+\frac{1}{N}\sum_{j=1}^N\F\big(Y_s^{i,N}-Y_s^{j,N})\Big)\dds+\sigma\int_0^t
		\Gamma(Y_s^{i,N})\,\dd\widetilde{B}_s^i\,,\ \
		t\ge 0\,,1\le i\le N\,,
	\end{equation}
\end{subequations}
for $(B^1,\cdots,B^N)$ and $(\widetilde{B}^1,\cdots,\widetilde{B}^N)$ denoting mutually independent copies
of $B$ and $\widetilde{B}$.  The zero-noise limits of the collision-time and collision-location for an isolated pair $(X^{i,N},Y^{i,N})$ (see \eqref{near-collision-time-Particle}), for fixed large $N$, will be also addressed below.

\noindent
  \paragraph{Near collision-time, near collision-location and a short overview of the paper.} Outside the one-dimensional case (which will be re-discussed in Appendix \ref{sec:1DCase}), tracking the first time $X$ and $Y$ may collide is typically an ill-posed question: 
  in the case $d\ge 2$, as $B$ and
  $\widetilde B$ almost surely  do not hit each others, provided
  that $\U$ and $\F$ allow a Girsanov's change of probability measure, the same can be stated for $X$ and $Y$. For $d$ general, considering the continuously diffusive nature of $X$ and $Y$ for arbitrary $\sigma>0$, the first
  collision-time and the first collision-location between \eqref{MV1} and \eqref{MV2} are more naturally described by the family of {\it near collision-times} and {\it near collision-locations}:
  \begin{equation}
  	\label{near-collision-time-SS}
  	C_\varepsilon(\sigma):=\inf\left\{t\ge 0\,:\,\|X_t-Y_t\|\le 2\varepsilon\right\}\,,\qquad L_\varepsilon(\sigma):=\big(X_{C_\varepsilon(\sigma)},Y_{C_\varepsilon(\sigma)}\big)\,, 
  \end{equation}
  with $\varepsilon>0$ defining a radius of collision aimed to be very small. With the above, collisions between the point-materials $(X,Y)$ are widened into collisions between a pair of moving permeable balls $\mathbb B(X_t;\varepsilon)$ and $\mathbb B(Y_t;\varepsilon)$, with centers of mass located in $X_t$ and $Y_t$ at each time $t$. {\color{black}While our first main result, Theorem \ref{thm:main1} below, focuses on the zero-noise asymptotic of $C_\varepsilon(\sigma)$ and $L_{\varepsilon}(\sigma)$ along the radius limit $\varepsilon\downarrow 0$, intermediate asymptotics for small enough $\varepsilon>0$ are also established in Proposition \ref{label_bizarre8}, Section \ref{sec:SelfStabilizingCase}.
  	
  	The analogues for the particle systems \eqref{MFSP1}-\eqref{MFSP2} are considered with the family of \textit{individual near collision-times} and \textit{individual near collision-locations}:
  	\begin{equation}
  		\label{near-collision-time-Particle}
  		C^i_{\varepsilon,N}(\sigma):=\inf\left\{t\ge 0\,:\,\|X_t^{i,N}-Y_t^{i,N}\|\le 2\varepsilon\right\}\,,\quad L^i_{\varepsilon,N}(\sigma):=\big(X^{i,N}_{C^i_{\varepsilon,N}(\sigma)},Y^{i,N}_{C^i_{\varepsilon,N}(\sigma)}\big),\quad \varepsilon>0\,,\,1\leq i\le N\,.
  	\end{equation}   
  	For $N$ large enough, similar zero-noise asymptotic (with $\varepsilon\downarrow 0$ or intermediate $\varepsilon$ small) are established in our second main result Theorem \ref{thm:main2} and Proposition \ref{ned}.

 Our particular interest when investigating $C_\varepsilon(\sigma)$ and $L_\varepsilon(\sigma)$ (or more generally the collision between \eqref{MV1} and \eqref{MV2}) in a noise-induced collision setting lies with the fact that, in the absence of noise, i.e. $\sigma=0$, the governing force field $-\U$ and initial points $x_1,x_2$ draw an energy landscape entrapping each dynamic in collision-free slopes and (near) collisions are only triggered by the Brownian perturbations. More specifically, for $\sigma=0$ and assuming $\F(0)=0$, all sources of randomness of \eqref{MV1} and \eqref{MV2} disappear and the paths reduces to two orbits, $\phi(x_1)=\{\phi_t(x_1)\}_{t\ge 0}$ and~$\phi(x_2)=\{\phi_t(x_2)\}_{t\ge 0}$, satisfying
  	\[
  	\dot{\phi}_t(x_k)=-\mathbb{U}(\phi_t(x_k))\,,\quad \phi_0(x_k)=x_k\,,\qquad t\ge 0\,,\,k=1,2\,.
  	\]
  	Considering $\U$ admits two attractors, $\lambda_1$ and $\lambda_2$, and considering the initial states $x_1$ and $x_2$ are each in a distinct basin of attraction of $\U$ with non-overlapping of the curves $\phi(x_1)$ and $\phi(x_2)$, naturally no collision occur. In the same way, for small enough $\varepsilon$, the balls $\mathbb B(\phi_t(x_1);\varepsilon)$ and $\mathbb B(\phi_t(x_2);\varepsilon)$ are collision-free at any time $t$.  
  	
  	Finite time near-collisions (i.e. $C_\varepsilon(\sigma)<\infty$ a.s.) so occur as long as the diffusive effects of $B$ and $\widetilde B$ are effective. As these stochastic forces are vanishing (with an evanescent intensity $\sigma$), the external force $-\U$ prevails and a (near) collision takes a longer time to occur, the driving Brownian motions having to increasingly force each diffusion $X$ and $Y$ to overcome the potential barrier draw by $-\U$. With this view, necessarily $C_\varepsilon(\sigma)$ steadily increases at a certain rate meanwhile the tracked  collision-location $L_\varepsilon(\sigma)$ may \textit{persists} in an energy-balanced region between the attractors $\lambda_1$ and $\lambda_2$. Both phenomena are intuitively determined by the parameter $\sigma$, the radius $\varepsilon$ and the "energy space" between the attractors $\lambda_1$ and $\lambda_2$.  
  
  Noise-driven collision mechanisms naturally play a central role in a wide range of physical and biological systems, including molecular encounters, Brownian coagulation and collective motions in active matter (see e.g. \cite{Norris06}, \cite{MaFiHePaYi-16}). Our own interest for this framework was initially motivated by measuring how small noise-induced encounters may be measured, and still be affecting, in population dynamics particularly in swarming 
   multi-agent systems (\cite{NaldiPareTosc-10,PareTosc-14}). The systems \eqref{MV1} and \eqref{MV2} (or \eqref{MFSP1} and \eqref{MFSP2} when considering models of finite populations) can so be viewed as two instances of autonomous dynamical agents each governed by a different aggregative self-propelled mechanisms and environmental noise cause populations to deviate significantly from their deterministic paths and collide with each others. The models \eqref{MV1}-\eqref{MV2} and  \eqref{MFSP1}-\eqref{MFSP2} feature prototypical versions of the kinetic Langevin-type dynamics, usually at play in such agents systems, where also no post-collision effects nor cross-systems interactions are considered here. These extensions significantly out-range the scope of this paper and will be considered in future works. 
   
   Let us further point out that noise-induced collisions and the zero-noise asymptotic of $C_\varepsilon(\sigma)$ (and of $C^{i}_{\varepsilon,N}(\sigma)$) are also inherently connected to chemical kinetics, particularly to the reaction rates of bi-molecular systems under thermal activation. Such rates are generally defined by Arrhenius \cite{Arrhenius-1889} law (see also \cite{Laider-84}), which quantifies the average frequency of a chemical reaction between two reactants as, up to an additive factor, of order $\exp\{-E_a/RT\}$, for $R$ the ideal gas constant, $T$ the temperature of matter and $E_a$ a minimal activation energy for a reaction to occur.
  

\medskip  
  
  Mathematically the questions of estimating how fast
  the near collision-times between the pair $(X,Y)$ and between each pairs $(X^{i,N},Y^{i,N})$ grow, and whether~-~and also where - the near collision-locations
  $L_\varepsilon(\sigma)=(X_{C_\varepsilon(\sigma)},Y_{C_\varepsilon(\sigma)})$ and
  $L^{i,N}_\varepsilon(\sigma)=(X^{i,N}_{C^i_{\varepsilon,N}(\sigma)},Y^{i,N}_{C^i_{\varepsilon,N}(\sigma)})$ may persist at the zero-noise limit, are intrinsically linked to Freidlin-Wentzell \cite{FreWen-12}'s LDP theory and exit-time estimates for small-noise perturbed dynamics. Essential results on this theory, which will be used throughout the paper, are briefly recollected in Section \ref{ssec:FW}. The link between exit-time and near collision-time may be drawn by observing that $C_\varepsilon(\sigma)$ (respectively $C^{i,N}_\varepsilon(\sigma)$) can be equivalently viewed as the first time when the diffusion $(X,Y)$ (resp. $(X^{i,N},Y^{i,N})$) enters the domain \begin{equation}\label{NaiveCollisionDomain}
  	\triangle_\varepsilon:=\{(x,y)\in\mathbb
  	R^d\,:\,\|x-y\|\le 2\varepsilon\}\,,
  \end{equation} 
  or,  equally, as the first time when $(X,Y)$ leaves~$\mathbb R^{2d}\setminus
  \triangle_\varepsilon$. With this view, our main results, Theorems \ref{thm:main1} and \ref{thm:main2}, establish a form of \textit{Freidlin-Wentzell collision laws}, where, at the limit of negligible near collision-radius ($\varepsilon\downarrow 0$) and of a vanishing noise ($\sigma\downarrow 0$), $C_\varepsilon(\sigma)$ grows at an exponential rate and $L_\varepsilon(\sigma)$ concentrates on a compact product set $\mathcal M_0\times\mathcal M_0$, with  
  \begin{equation}\label{MainAsymptoticSS}
  	\frac{\sigma^2}{2}\log\big(C_\varepsilon(\sigma)\big)\overset{\mathbb P}{\underset{\sigma\downarrow 0}{\asymp}}  \underline H_0\,,\qquad \dist\bigg(L_\varepsilon(\sigma),\mathcal M_0\times \mathcal M_0\bigg)\overset{\mathbb P}{\underset{\sigma\downarrow 0}{\asymp}}  0\,.
  \end{equation}
  The symbol $\overset{\mathbb P}{\underset{\sigma\downarrow 0}{\asymp}}$ denotes, here and below, the asymptotic equivalence in probability at $\sigma\downarrow 0$. For the particle systems \eqref{MFSP1}-\eqref{MFSP2}, with $N$ large enough, the asymptotics remain the same, with, as $\varepsilon\downarrow 0$,
  \begin{equation}\label{MainAsymptoticParticle}
  	\frac{\sigma^2}{2}\log\big(C^{i}_{\varepsilon,N}(\sigma)\big)\overset{\mathbb P}{\underset{\sigma\downarrow 0}{\asymp}} \underline H_0\,,
  	\qquad \dist\bigg(L^i_{\varepsilon,N}(\sigma),\mathcal M_0\times \mathcal M_0\bigg)\overset{\mathbb P}{\underset{\sigma\downarrow 0}{\asymp}} 0\,.
  \end{equation}  
  The (positive) factor $\underline H_0$ in \eqref{MainAsymptoticSS} and  \eqref{MainAsymptoticParticle} describes a {\it collision-cost} at the vanishing noise limit characterizing a minimal energy required for a collision to occur. The domain $\mathcal M_0$ defines the common asymptotic concentration region where each pair $X_{C_\varepsilon(\sigma)}$-$Y_{C_\varepsilon(\sigma)}$, and $X^{i,N}_{C^i_{\varepsilon,N}(\sigma)}$-$Y^{i,N}_{C^i_{\varepsilon,N}(\sigma)}$, will remain. The exact form of the collision cost is given by $\underline H_0=\inf_{\lambda}H_0(\lambda)$ (see $H_0$ defined in \eqref{collisioncost} in Theorem \ref{thm:main1}) and the persistence domain by the minimizers set $\mathcal M_0=\text{argmin} H_0$.  
  \medskip
%
  
    To (briefly) illustrate the domain $\mathcal M_0$: In the case $\Gamma=I_d$, $\U=\nabla u$ and $\F(x)=\alpha\nabla \|x\|^2/2$, the domain $\mathcal M_0$ reduces to a single collision-location $\lambda_0$ which lies in an exact energy equilibrium balance between the two wells $\lambda_1$ and $\lambda_2$ of $\U$. In a further ``simple'' one-dimensional picture, considering $\Gamma=1$, the quadratic interaction force and  $\U(x)=\nabla u(x)$, with the two instances $u(x)=x^4/4-x^2/2$ (symmetric  double-wells) and $u(x)=x^4/4+x^3/3-x^2/2$ (asymmetric  double-wells), the collision-location $\lambda_0$ can be exemplified with the Figure \ref{Fig1}:
  \begin{figure}[h!]
  	\centering
  	\begin{subfigure}{0.4\textwidth}
  		\centering
  		\includegraphics[width=\linewidth]{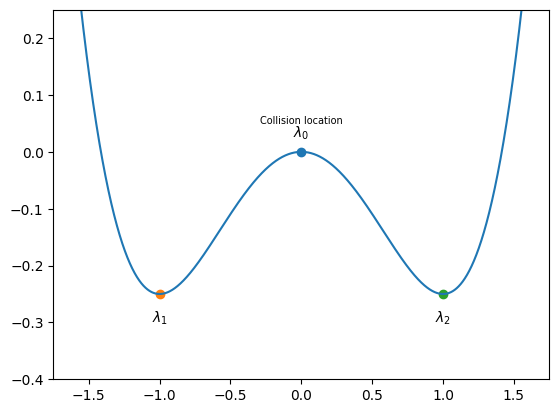}
  	\end{subfigure}
  	\vspace{0.5cm} 
  	\begin{subfigure}{0.4\textwidth}
  		\centering
  		\includegraphics[width=\linewidth]{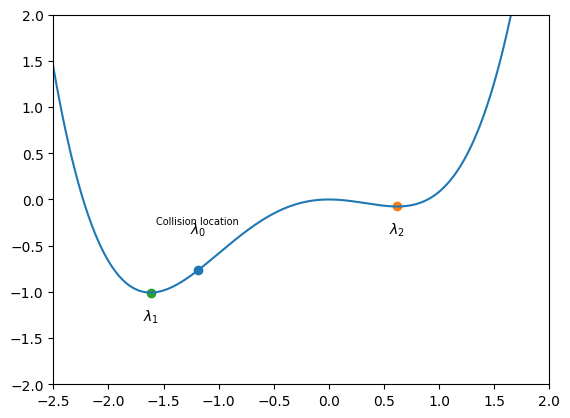}
  	\end{subfigure}
  	\caption{Collision-location $\lambda_0$ for symmetric (left) and asymmetric (right) bistable landscape}
  	\label{Fig1}
  \end{figure}
  
  \noindent
  (We refer to Remark \ref{rem:examples} for detailed comments on these particular cases and the figure.)
  
  \medskip 
  
  As it will be discussed in Section \ref{sec:mainresults}, while the connection with Freidlin-Wentzell's exit-time estimate (with the exit-domain $\mathbb R^{2d}\setminus 	\triangle_\varepsilon$) is \emph{a priori} immediate -- and one would be able to infer the validity of  \eqref{MainAsymptoticSS} from the exit-laws for self-stabilizing diffusions established in \cite{HIP} --, the rigorous demonstration of the {\it Freidlin-Wentzell collision-laws} \eqref{MainAsymptoticSS} and  \eqref{MainAsymptoticParticle} is not straightforward and requires an intricate tailoring of exit-domains to appropriately apply Freidlin-Wentzell's estimates (see in particular Section \ref{sec:LinearCase}). 
  
  Additionally, it may also note that the estimate \eqref{MainAsymptoticSS} for $C_\varepsilon(\sigma)$, equivalent to 
  $$
  \frac{1}{C_\varepsilon(\sigma)}\overset{\mathbb P}{\underset{\sigma\downarrow 0}{\asymp}}  \exp\Big[-\frac2{\sigma^2}\underline H_0\Big]\,,
  $$
  provides a form of Arrhenius' law (as mentioned above), considering the frequency of a collision as the quotient of the collision time. While the connection between Freidlin-Wentzell theory and Arrhenius' law, or more generally Arrhenius-Eyring-Kramers law (\cite{Eyring-35}, \cite{Kramers-40}), has been extensively studied (see e.g. \cite{Berglund-13}), the present connection through the asymptotics \eqref{MainAsymptoticSS} of $C_\varepsilon(\sigma)$ and the derivation of a persisting collision domain, is, to the best of our knowledge, new and features one of the specific contribution of the paper. 
    
\paragraph{Notation.} Throughout the paper, the symbols  $\cdot$ and $\|.\|$ will denote the scalar product and the Euclidean norm on $\mathbb R^d$.
For any point $x\in\mathbb R^d$, $\mathbb B(x;\varepsilon)$ will denote the open ball centered in $x$ with radius $\varepsilon$ and, for any set $\mathcal O\subset \mathbb R^d$, $\dist(x,\mathcal O)$ will denote the minimal distance $\inf_{z\in \mathcal O}\|x-z\|$ between $x$ and $\mathcal O$. For all $x,z\in\mathbb R^d$ and $A(x)$ an invertible square matrix depending on $x$, $\|z\|^2_{A^{-1}(x)}$ will also denote the quadratic form $z \cdot A^{-1}(x) z$. Finally, given the vector fields $\U$ and $\F$ in \eqref{MV1}-\eqref{MV2}, and for $\mu$ a probability measure on $\mathbb R^d$, we will make a systematic use of the effective force  $\mathbb W_\mu$ common to \eqref{MV1}-\eqref{MV2} and given by
$$
\mathbb W_\mu(x):=\U(x)+\int \F(x-z)\mu(\ddz)=\U(x)+F\star\mu(x),\qquad x\in\mathbb R^d\,,
$$
for $\star$ the convolution product. We will also make use of the special convention: for $\lambda$ a point in $\mathbb R^d$, $\mathbb W_\lambda$ corresponds to $\mathbb W$ evaluated along the Dirac measure in $\lambda$, that is  
 $$
 \mathbb W_{\lambda}(x)=\U(x)+\F(x-\lambda),\qquad x\in\mathbb R^d\,.
 $$

\subsection{Model assumptions}
From here on, the pairs \eqref{MV1}-\eqref{MV2} and \eqref{MFSP1}-\eqref{MFSP2} will be considered under the following set of
assumptions $\mathbf{(A)}$.

\medskip

\noindent
\textit{$\mathbf{(A)}$-$(i)$ The vector field $\U:\mathbb{R}^d\rightarrow \mathbb R^d$ is of class $\mathcal C^1$, grows at most at a $2m$-polynomial rate at infinity and its Jacobian $\nabla\U$ is uniformly elliptic far from the origin. Namely, $\displaystyle{\lim_{\|x\|\rightarrow \infty}\frac{\|\U(x)\|}{1+\|x\|^{2m}}<\infty}$ and, for $\Rd'>0$ large enough, there exists $C(\Rd')>0$ such that for all $\|x\|>\Rd'$ and all $\xi\in\mathbb R^d$,
\begin{equation}\label{Hyp-U3}
	\xi\cdot\nabla \U(x)\xi\ge C(\Rd')\|\xi\|^2\,.
\end{equation}}
\noindent
\textit{$(\mathbf{A})$-$(ii)$ {\bf (Bistable landscape)} The flows generated by $-\U$ admit  \emph{exactly} two distinct asymptotically stable attracting points, $\lambda_1$ and~$\lambda_2$, in the Lyapunov sense. That is: $\lambda_1$ and $\lambda_2$ are two stable fixed points for the zero-noise flows}
\begin{equation}\label{eq:LimitDynamic}
\phi_t(z)=z-\int_0^t \U(\phi_s(z))\,\dds\,,\qquad t\ge 0\,,\,z\in\mathbb R^d\,,
\end{equation}
\textit{and, for $i=1,2$ and any $\epsilon>0$, there exits a non-empty neighborhood $O_\epsilon$ of $\lambda_i$ such that, for all $t\ge 0$, $\phi_t(O_\epsilon)\subset \mathbb B(\lambda_i;\epsilon)$ and $\lim_{t\rightarrow \infty} \phi_t(z)=\lambda_i$ for all $z$ in $O_\epsilon$.\footnote{We stress that the assumption concerning the existence of the neighborhood $\mathcal O_\epsilon$ corresponds to the right setting when the drift $-\U$ is \emph{not} of gradient form.}}  \\

\noindent
\textit{$(\mathbf{A})$-$(iii)$ The interaction kernel $\F$ is of the form}
\begin{equation}\label{RotInv}
\F(x)=\frac x{\|x\|}G(\|x\|)\,,\qquad x\in\mathbb R^d\,,
\end{equation}
\textit{for $G:[0,\infty)\rightarrow [0,\infty)$ an increasing polynomial function of order $2n$, for $n\ge 1$, such that $G(0)=0$.\\
\noindent
$(\mathbf{A})$-$(iv)$ {\bf (Synchronization condition)}  For all probability measure 
$\mu\in \mathcal P(\mathbb R^d)$, the effective force field
$$
\W_\mu(x)=\U(x)+\int \F(x-y)\,\mu(\ddy)\,,\qquad x\in\mathbb R^d\,,
$$
is (globally) one-sided contractive Lipschitz continuous: for some $\theta>0$, we have}
\begin{equation}\label{Hyp:Synchro}
	\Big(\W_\mu(x)-\W_\mu(y)\Big)\cdot \Big(x-y\Big)\ge \theta\|x-y\|^2\,,\qquad x,y\in \mathbb R^d\,,\mu\in\mathcal P(\mathbb R^d)\,.
\end{equation}

\noindent
\textit{$\mathbf{(A)}$-$(v)$ Each initial condition $x_1$ and $x_2$ lies in a distinctive basin of attraction of $-\U$, namely, for $i=1$ or $2$, $x_i$ belongs to the basin of attraction $\mathcal{G}(\lambda_i)$ of $\lambda_i$ given, with $(\phi_t(z))_{t\ge
	0}$ as in \eqref{eq:LimitDynamic}, by}
\[
\mathcal{G}(\lambda_i):=\left\{z\in\mathbb R^d\,:\,\,\lambda_i=\lim_{t\rightarrow \infty}\phi_t(z)\right\}\,.
\]

\noindent
\textit{$\mathbf{(A)}$-$(vi)$ The diffusion matrix $\Gamma$ is  bounded, Lipschitz continuous and} 
	$$
	x\mapsto A(x):=\Gamma (x)\Gamma^*(x)\,,
	$$ 
\textit{is uniformly elliptic with: for some $\Lambda_+, \Lambda_->0$,}
\[
\max_{x\in \mathbb R^d} \max_{i,j}|\Gamma_{i,j}(x)|\le \Lambda_+<\infty\,,\quad \inf_{x\in \mathbb R^d}\xi\cdot A(x)
\xi \ge \Lambda_-\|\xi\|^2\,,\qquad \forall \xi\in\mathbb R^d\,.
\]

\noindent
\textbf{Remarks}
\noindent
$\bullet$ Assumptions $\mathbf{(A)}$-$(i)$, $\mathbf{(A)}$-$(iii)$ and$\mathbf{(A)}$-$(vi)$ are enough to guarantee strong wellposedness and moment controls (provided $\sigma$ is sufficiently small) for Equations \eqref{MV1}-\eqref{MV2} (see Proposition \ref{prop:WellposedMV} below) and also for Equations \eqref{MFSP1}-\eqref{MFSP2}. In the case $\Gamma=I_d$, similar results under an analog setup were previously established in \cite{HIP} (see e.g. \cite[Theorem
2.13]{HIP}), and the extension of the proof arguments therein to inhomogeneous diffusion $\Gamma$ are addressed in the Appendix \ref{Appendix-A}. From \cite{HIP}, let us already recollect some key properties of $\U$ and $\F$ which will be used later.\\ 
\noindent
First, Assumption $\mathbf{(A)}$-$(i)$ ensures (see \cite[Lemma~2.2]{HIP}) that $-\mathbb{U}$ is a one-sided contractive $\lambda$-Lipschitz ($\lambda>0$) function for large distanced points, namely, for $R''>0$ large enough
\begin{equation}\label{Hyp-V1}
-\left( \mathbb{U}(x)-\mathbb{U}(y)\right)\cdot (x-y)\leq -\lambda \|x-y\|^2, \qquad x,y\in\mathbb R^d\,\mbox{such
	that}\,\|x-y\|\ge R''\,.
\end{equation}
With the $\mathcal C^1$-regularity of $\U$, $-\U$ is also globally one-sided Lipschitz: for $\gamma>0$,
\begin{equation}\label{Hyp-V2}
-\left( \mathbb{U}(x)-\mathbb{U}(y)\right)\cdot (x-y)\leq \gamma \|x-y\|^2, \qquad x,y\in\mathbb R^d\,.
\end{equation}
For notation convenience, we will denote the local Lipschitz
constant related to $\U$ by
\begin{equation}\label{LocalLip-U}
\forall R>0\,,\:\:\cl(R):=\sup_{\max(\|x\|,\|y\|)<R}\frac{\|\mathbb{U}(x)-\mathbb{U}(y)\|}{\|x-y\|}<\infty\,.
\end{equation}
Moreover $-\U$ is strongly confining far from the origin: there exists a radius $\Rd'>0$ and $\cd=\cd(\Rd')>0$ such that
\begin{equation}\label{Hyp-U3-bis}
	-\U(x)\cdot x\le -\cd(\Rd')\|x\|^2\,,\qquad \,\|x\|>\Rd'\,.
\end{equation}
Second, Assumption $(\mathbf{A})$-$(iii)$ implies that $\F$ rewrites as the gradient field 
\begin{equation}\label{GradientInteractions}
\F(x)=\nabla\left(g(\|x\|)\right),\ \ g'=G,\qquad x\in\mathbb R^d\,,
\end{equation}
and is locally-Lipschitz continuous with at most polynomial growth satisfying, for some $K>0$: 
\begin{equation}\label{Hyp-F1}
\|\F(x)-\F(y)\|\le \K\|x-y\|\Big(1+\|x\|^{2n-1}+\|y\|^{2n-1}\Big), \qquad  x,y\in\mathbb R^d\,.
\end{equation}
Furthermore (\cite[Lemma 2.3-(d)]{HIP}), $\F$ satisfies the monotone condition: for any $p\ge 1$, 
\begin{equation}\label{Hyp-F2}
\Big(x\|x\|^p-y\|y\|^p\Big)\cdot \F(x-y)\ge 0\,.
\end{equation}
As $\F$ is anti-symmetric, the above is further equivalent to the extended monotone property: for any probability measure $\mu$ on $\mathbb R^d$,
\begin{equation}\label{Hyp-F2_bis}
2\int \|x\|^p x\cdot \F(x-y)\mu(\ddx)\mu(\ddy)=\iint \Big(x\|x\|^p-y\|y\|^p\Big)\cdot \F(x-y)\mu(\ddx)\mu(\ddy)\ge 0\,.
\end{equation}

\medskip 

\noindent
$\bullet$ Assumption $\mathbf{(A)}$-$(iv)$ sets the systems \eqref{MV1}-\eqref{MV2} and \eqref{MFSP1}-\eqref{MFSP2} in a synchronized regime (see e.g. \cite{BerFerGen-07a}, \cite{BerFerGen-07b}) where self-stabilizing force
compensates the lack of global attractivity of $-\U$ toward the attractors $\lambda_1$ and $\lambda_2$. Illustratively, in the prototypical gradient case $\U=\nabla u$ and $\F=\nabla f$ with $f(x):=\frac \alpha{2}\|x\|^2$, $\mathbf{(A)}$-$(i)$ only ensures strong convexity of $u$ for $\|x\|$ far from~$0$, and the synchronization condition \eqref{Hyp:Synchro} imposes, provided that $\alpha$ is large enough, that the \emph{effective potential} of the system,  $w_\mu(x):= u(x)+\int f(x-z)\mu(\ddz) =u(x)+\frac{\alpha}{2}\int \|x-y\|^2\mu(\ddy)$,
is uniformly convex on $\mathbb R^d$, independently of the measure argument $\mu$. The convexity threshold is provided by $\theta:=\alpha+\inf_{x,\xi\in\mathbb R^d} \xi \cdot \nabla^2 u(x) \xi/\|\xi\|^2>0$. The synchronization naturally vanishes with $\sigma$ (and since $\F(0)=0$) leaving only $\U$ as the driving force for the zero-noise limits of \eqref{MV1}-\eqref{MV2}, but remains active in the estimate of $C_\epsilon(\sigma)$ and $L_\epsilon(\sigma)$.

\noindent
This {\it synchronization} assumption is primarily used to derive key coupling arguments between \eqref{MV1}-\eqref{MV2}, \eqref{MFSP1}-\eqref{MFSP2} and intermediate diffusions concentrating in $\lambda_1$ and $\lambda_2$ (see Propositions \ref{label_bizarre2} and \ref{label_bizarre9} below). Following \cite{Alea} (for $d=1$) and \cite{Tugaut2021}, in the  gradient case with $\Gamma=I_d$ and $f(x)=\alpha\|x\|^2/2$, this condition can potentially be weakened into the form of Lipschitz contractiveness of $\W_{\lambda_i}=\nabla w_{\lambda_i}$ around $\lambda_i$: 
\begin{equation}\label{future}
-\left( x-\lambda_i\right)\cdot\nabla w_{\lambda_i}(x)\Big(=-\left( x-\lambda_i\right)\cdot\left(\nabla u(x)+\alpha(x-\lambda_i)\right)\Big)\le-\rho_i\|x-\lambda_i\|^2,\qquad i=1,2,\,x\in\mathbb R^{d}\,.
\end{equation}
More specifically, as pointed out in \cite[Corollary D]{Alea} and \cite[Theorem 3.4]{Tugaut2021}, this weaker condition allows to control the proximity between
the law of the self-stabilizing diffusions and their assigned attractors over an extensive time-period. In our setting, such condition is enough to ensure some particular controls (e.g. Lemma \ref{chante}) holds but not enough to ensure the coupling estimates in Propositions \ref{label_bizarre2} and \ref{label_bizarre9} yielding to Theorems~\ref{thm:main1} and~\ref{thm:main2}. The extension of these results under this weaker condition \eqref{future} will so necessitate a complete new proof strategy which should be addressed in future works.
\medskip 

\noindent
$\bullet$ Assumptions $\mathbf{(A)}$-$(ii)$ and $(\mathbf{A})$-$(v)$ generate the specific noise-induced collision regime aimed by the paper. 
Specifically, the combination of these assumptions ensures a bistable regime
where all diffusions evolve in a landscape where $\lambda_1$ and $\lambda_2$ define two separate attractors, one for $X$ and $X^{i,N}$, and the
other one for $Y$ and $Y^{i,N}$, and the zero-noise regimes $\phi(x_1)$ and
$\phi(x_2)$ are collision-free. In order to avoid any trivial situation with near collision, the collision-radius $\varepsilon$ is taken from now on smaller than the smallest distance between the
zero-noise limit $(\phi(x_1),\phi(x_2))$ of $(X,Y)$:
\begin{equation}\label{Threshold}
	\varepsilon_0:=
	2^{-1}\inf_{t\ge 0}\{\|\phi_t(x_1)-\phi_t(x_2) \|\}\,,
\end{equation} 
 $\mathbf{(A)}$-$(ii)$ and $(\mathbf{A})$-$(v)$ ensuring 
 $\varepsilon_0>0$. 
 \medskip 

\subsection{Well-posedness and moment estimates for \texorpdfstring{\eqref{MV1}}{}-\texorpdfstring{\eqref{MV2}}{} and  \texorpdfstring{\eqref{MFSP1}}{}-\texorpdfstring{\eqref{MFSP2}}{}}\label{ssec:Wellposed-Moment}
The smoothness of the diffusion $\Gamma$, the one-sided Lipschitz continuity of $\U$ and $\F$, namely \eqref{Hyp-V2} and \eqref{Hyp-F1}, along with the polynomial form of $G$ in \eqref{RotInv} and the confining condition \eqref{Hyp-U3-bis} ensure strong well-posedness of the particle system \eqref{particles} (see e.g. \cite[Chapter 5]{KarShr91}, \cite[Chapter 3]{Mao-07}). For the self-stabilizing diffusions \eqref{MV1} and \eqref{MV2},  we have the following wellposedness and uniform-in-time moments control, which is proved in Section~\ref{Appendix-A}.
\begin{prop}\label{prop:WellposedMV} If $\mathbf{(A)}$-$(i)$, $(\mathbf{A})$-$(iii)$ and $\mathbf{(A)}$-$(vi)$ hold, then, for any $x\in\mathbb R^d$, the SDE
	$$
	X_t=x-\int_0^t \Big(\U(X_s)+\int \F(X_s-x)\mu^X_s(\ddx)\Big)\,\dds+\sigma\int_0^t\Gamma(X_s)\,\dd B_s\,, \qquad t\ge 0\,,
	$$
	admits a unique strong solution such that, whenever $\sigma<\Big(\sqrt{d}\Lambda_+\Big)^{-1}$ (for $\Lambda_+$ as in $\mathbf{(A)}$-$(vi)$), there exists a finite constant $C>1$, independent of $t$, such that, for all $p\ge 1$,
	\begin{equation}\label{Moment_SS}
	\sup_{t\ge 0}\mathbb E\Big[\|X_t\|^{2p}\Big]\le C^p (p-1)!\,.
	\end{equation}
\end{prop} 

In connection with \eqref{Moment_SS}, let us highlight the following, rather straightforward, martingale concentration inequality. This lemma will be intensively used throughout to handle the lack of the non-homogeneity of the diffusion matrices $\Gamma$ in \eqref{MV} and \eqref{particles}.
\begin{lem}\label{lem:MartingaleConcentration} Let $\mathbf{B}$ be a $m$-dimensional Brownian motion and $(\beta_t)_{t\ge0 }$ be a $\mathbb R^m$-valued $\mathcal F_t$-adapted stochastic process (for some filtration $\left(\mathcal F_t\right)_{t\geq0}$) such that, there exists $C>0$ satisfying, for all $p\ge 1$, $\sup_{t\ge 0}\mathbb E[\|\beta_t\|^{2p}]\le C^p(p-1)!$. Then, for $M=\big(\int_0^t \beta_s\,\dd\mathbf B_s\big)_{t\ge 0}$ and $\langle M\rangle$ the related quadratic variation process, it holds that: for any $c>0$, 
	$$
	\mathbb P\Big(\exists t\ge 0\,:\,M_t>\frac 12\langle M\rangle_t+c\Big)\le \exp\big[-c\big]\,.
	$$
\end{lem}\begin{proof} Owing to the moment control $\sup_{t\ge 0}\mathbb E[\|\beta_t\|^{2p}]\le C^p(p-1)!$, the following local Novikov condition holds for any non-random increasing sequence of times $\{t_n\}_{n\ge 0}$ such that $t_0=0$ and $\sup_n \big(t_{n+1}-t_n\big)<\min\{1;2/C\}$:
	\begin{align*} 
		\mathbb E\bigg[\exp\Big[\frac1 2\int_{t_n}^{t_{n+1}}\|\beta_s\|^2\,\dds\Big]\bigg]&=\sum_{p\ge 0}\frac{1}{p!2^p}\mathbb E\bigg[\Big(\int_{t_n}^{t_{n+1}}\|\beta_s\|^2\,\dds\Big)^p\bigg]\nonumber\\
		&\le \sum_{p\ge 0}\frac{(t_{n+1}-t_n)^p}{p!2^p}\mathbb E\bigg[\int_{t_n}^{t_{n+1}}\|\beta_s\|^{2p}\,\dds\bigg]\nonumber\\
		&\le \sum_{p\ge 0}\frac{(p-1)!(t_{n+1}-t_n)^{p+1}C^p}{p!2^p}<\infty
	\end{align*} 
	(the first inequality being due to Jensen's inequality).
	The process $\big\{\exp\left[M_t-\frac 12\langle M\rangle_t\right]\big\}_{t\ge 0}$ is thus a martingale, with value $1$ at time $t=0$, and Doob's inequality immediately implies that
	\begin{align*}
		\mathbb P\Bigg(\exists t\ge 0\,:\,M_t\ge \frac 12\langle M\rangle_t+c\Bigg)
		&=\lim_{T\uparrow \infty}\mathbb P\Big(\exists t\in[0,T]\,:\,\exp\big[M_t-\frac12\langle M\rangle_t\big]\ge \exp\big[c\big]\Big)\\
		&\le\lim_{T\uparrow \infty} \mathbb E\left[\exp\big[M_{T}-\frac12\langle M\rangle_{T}\big]\right]\exp\big[-c\big]=\exp\big[-c\big]\,.
		\end{align*}
\end{proof}
Owing to Lemma \ref{lem:MartingaleConcentration}, we have the analog of Proposition \ref{prop:WellposedMV} for the particle system \eqref{MFSP1}-\eqref{MFSP2} (the proof is also postponed to Appendix \ref{Appendix-A}).
\begin{lem}\label{lem:Moment_particle} Under the same assumptions than Proposition \ref{prop:WellposedMV}, we have 
	\begin{equation}\label{Moment_particle}
		\sup_{t\ge 0,N\ge 1}\mathbb E\left[\|X_t^{i,N}\|^{2p}+\|Y_t^{i,N}\|^{2p}\right]
		<C^p (p-1)!\,.
	\end{equation}
\end{lem} 
Without loss of generality, we will assume from now on that the systems \eqref{MV1}, \eqref{MV2}, \eqref{MFSP1} and \eqref{MFSP2} are defined on a common filtered probability space $(\Omega,\mathcal F,\{\mathcal F_t\}_{t\ge 0},\mathbb P)$ - with $\{\mathcal F_t\}_{t\ge0}$ satisfying the usual conditions - and that $\sigma$ is lower than the threshold $(\sqrt{d}\Lambda_+)^{-1}$ set in Proposition \ref{prop:WellposedMV} and Lemma \ref{lem:Moment_particle}.

\subsection{Freidlin-Wentzell's exit-time and exit-location estimates.}\label{ssec:FW}
Consider the diffusion process:
\begin{equation}\label{GenSGF}
z_t^\sigma=z_0 +\int_0^t b\left(z_s^\sigma\right)\dds+\sigma \int_0^t \gamma(z^\sigma_s)\,\dd\mathcal W_s,\qquad t\ge 0\,,
\end{equation}
where $b:\bRb^m\rightarrow \bRb^m$ is a smooth vector field, $\gamma$ is a smooth non-degenerate diffusion matrix and $\mathcal W$ is a $\mathbb R^m$-Brownian motion (with $m\geq1$). As the noise intensity~$\sigma$ decreases,
the paths of $z^\sigma$ become closer to the dynamical system:
\begin{equation*}\label{GenSGFbis}
\Psi_t(z_0)=z_0+\int_0^t b\left(\Psi_s(z_0)\right)\dds,\qquad t\ge 0\,.
\end{equation*}
Over any finite time-interval, the  way $z^\sigma$ approaches the orbit $\Psi(z_0)=(\Psi_t(z_0))_{t\ge 0}$ is characterized by the following large deviations principle (see e.g. \cite[Theorem~2.3]{FreWen-12}, \cite[Section~5.6]{DemZei-10}): for $a^{-1}(x)=(\gamma\gamma^*)^{-1}(x)$ the inverse matrix of $a(x)=\gamma\gamma^*(x)$ and, $\|\cdot\|_{a^{-1}(x)}$ the weighted semi-norm $\|z\|_{a^{-1}(x)}=\sqrt{z\cdot a^{-1}(x)z}$, given any
arbitrary time horizon $T$ and for any $\delta>0$, 
\begin{equation}\label{LDP}
\log\PP\left(\sup_{t\in[0,T]}\| z_t^\sigma-\Psi_t(z_0)\|>\delta\right)
\underset{\sigma\downarrow 0}{\asymp} -\frac{1}{2\sigma^2}\inf_{\phi}\left\{\int_0^T\Big\| \dot{\phi}_t(z_0)-b(\phi_t(z_0))\Big\|_{a^{-1}(\Phi_t(z_0))}^2\,\ddt\right\}\,,
\end{equation}
 the infimum on the right-hand side being taken over the class of continuously differentiable functions  $\phi:[0,T]\rightarrow \mathbb R^m$ starting from
$z_0$ at $t=0$, and such that $\max_{0\le t\le T}\| \phi_t(z_0)-\Psi_t(z_0)\|>\delta$. 

In the case where the orbit $\Psi(z_0)$
has an asymptotically stable attracting point $z_{stable}$ (i.e., for any starting point $z_0$ close to $z_{stable}$,  $\lim_{t\rightarrow \infty}\Psi_t(z_0)=z_{stable}$),
the trajectories of $z^\sigma$ are naturally wandering around a neighborhood of $z_{stable}$, as $\sigma\downarrow 0$. While the LDP \eqref{LDP} only captures the zero-noise behavior of $z^\sigma$ over finite time-horizons, 
 M. I. Freidlin and A. D. Wentzell established such asymptotic through the exponential growth of the exit-time
of $z^\sigma$ from a stable domain $\cGc$ and the concentration point of $z^\sigma$ evaluated at this exit-time. More specifically, Freidlin-Wentzell results  (\cite[Theorem 5.7.11]{DemZei-10}; \cite[Chapter 4, Theorems~2.1 and~4.2, and Chapter 5]{FreWen-12}) formulate 
as follows. 
\begin{thm}\label{thm:KramersDZ} Assume that $b:\mathbb R^m\rightarrow \mathbb R^m$ and $\gamma$ are globally Lipschitz continuous with $x\mapsto a(x)=\big(\gamma\gamma^*\big)(x)$ uniformly elliptic. Let $\cGc$ be an open bounded set of $\mathbb R^m$, stable by $b$ (i.e., for all $t\ge 0$, $\Psi_t(\cGc)\subset \cGc$), and such that $\cGc$ contains exactly one asymptotically attracting stable point $z_{stable}$ for the flows $\Psi(z_0)$, $z_0\in\bar{\mathcal G}$, generated by $b$. Moreover, assume that, for all $z_0$ in the boundary $\partial\mathcal{G}$, $\Psi_t(z_0)$ converges to $z_{stable}$ as $t\uparrow \infty$. Then, for any $z_0$ in $\mathcal{G}$, given $z^\sigma$ the solution to \eqref{GenSGF}, the first exit-time from $\overline{\cGc}$ of $z^\sigma$:
 		$$
 		\tau_{\cGc}(\sigma):=\inf\{t>0\,:\,z_t^\sigma \notin \cGc\}\Big(=\inf\{t>0\,:\,z_t^\sigma \in\partial \cGc\}\Big)\,,
 		$$
 		satisfies: for all $\delta>0$,
 		\begin{equation}\label{GenericKramers}
 			\lim_{\sigma\to0}\PP\left(\exp\left[\frac{2}{\sigma^2}\left(\underline{h}-\delta\right)\right]<\tau_{\cGc}(\sigma)<
 			\exp\left[\frac{2}{\sigma^2}\left(\underline{h}+\delta\right)\right]\right)=1\,,
 		\end{equation}
 		for $\underline{h}:=\inf_{z\in \partial\mathcal G}V(z)$ the {\it exit-cost} of escaping $\mathcal G$ and for $V$ the {\it quasi-potential}:
 		\[ 
 		V(z):=\frac 1{4}\inf_{0\le T<\infty}\inf_\phi\left\{\int_{0}^{T} \big\|\dot{\phi}_t-b(\phi_t)\big\|^2_{a^{-1}(\phi_t)} 
 		\, \ddt\,:\, \phi\in\mathcal
 		A^{z_{stable},z}([0,T])
 		\right\}\,,
 		\]
 		(the control set $\mathcal A^{z_{stable},z}([0,T])$ corresponding to all continuous differentiable paths $\phi=(\phi_t)_{t\in [T_1,T_2]}$ starting from $z_{stable}$ at time $t=T_1$ and ending at the location $z$ at time $t=T_2$).
 		
 		\medskip
 		
 		\noindent
 		Additionally, for any closed subset $\mathcal N\subset \partial \mathcal G$, if $\inf_{z\in\mathcal N}\mathcal V(z)> \inf_{z\in\partial\mathcal G}\mathcal V(z)$, then
 		
 		\begin{equation}\label{GenericKramers-bis}
 		\lim_{\sigma\to0}\mathbb P\left(z_{\tau_{\cGc}(\sigma)}^\sigma\in \mathcal N\right)=0\,.
 		\end{equation}
 \end{thm}

The asymptotic  \eqref{GenericKramers} specifically states that $\tau_{\mathcal G}(\sigma)$ grows in probability at an exponential rate proportional to the cost $\underline{h}$ corresponding to the minimal energy for exiting $\mathcal G$ from $z_{stable}$. Equivalently, at a logarithmic scale, $\log(\tau_{\mathcal G}(\sigma))\overset{\mathbb P}{\underset{\sigma\downarrow 0}{\asymp}} \frac2{\sigma^2}\underline{h}$. In the case $b=-\nabla u$ and $\gamma$ is the identity matrix $I_m$ of $\mathbb R^m$, the exit-cost in \eqref{GenericKramers} reduces to the closed-form:
\[
\underline{h}=\inf_{z\in \partial G}\big(u(z)-u(z_{stable})\big).
\]
and the minimizing path for $\inf_{z\in \partial\mathcal G} V$ is achieved, up to a time-change, by the flow $
\dot{\phi}_t=\nabla u(\phi_t)$ connecting $z_{stable}$ and $\partial \mathcal{G}$. In the non-reversible situation (\cite[Chapter 5]{FreWen-12}), the exit-cost preserves the form $\underline{h}=\inf_{z\in \partial G}\big(V(z)-V(z_{stable})\big)$ provided that $V$ is solution to the Ricatti equation: $H(x,\nabla V(x))=0$, for the Hamiltonian
\[
H(x,p)=\sup_{\alpha \in\mathbb R^m}\left\{\alpha\cdot p-\|\alpha+b(x)\|^2_{a^{-1}(x)}\right\}\,.
\]
The corresponding minimizing curve is then $\dot{\phi}_t=\nabla V(\phi_t)$. 

 The asymptotic \eqref{GenericKramers-bis} states that the exit-location $z^\sigma_{\tau_{\mathcal G}(\sigma)}$ concentrates on the region where it is the least costly  to exit the domain starting from the attraction point $z_{stable}$. Considering again the case $b=-\nabla u$ and $\gamma=I_m$, if there exists a unique~$z^\star$ in $\partial\mathcal G$ such that
$u(z^\star)-u(z_{stable})=\inf_{z\in\partial\mathcal G}(u(z)-u(z_{stable}))$, then, for all $\delta>0$, $z_0\in\mathcal G$,
\begin{equation*}
	\lim_{\sigma\rightarrow 0}\mathbb P\left(\|z_{\tau_{\cGc}(\sigma)}^\sigma-z^\star\|<\delta\right)=1\,.
\end{equation*}

Let us also point out that the asymptotic \eqref{GenericKramers} can be somehow strengthened (\cite{DemZei-10}, \cite{FreWen-12}) with the average asymptotic 
\[
\log\mathbb E\big(\tau_{\cGc}(\sigma)\big)\underset{\sigma\downarrow 0}{\asymp} \frac{2} {\sigma^2}\underline{h},\qquad  z_0\in\mathcal G,
\]
yielding (as discussed in the previous section) to a form of Arrhenius \cite{Arrhenius-1889}' law (we again refer to \cite{Berglund-13}, and \cite{BouRey-16} for the refined Eyring-Kramers law).

Freidlin-Wentzell's LDP theory and small noise exit-estimates are at the foundations of the mathematics of metastability (\cite{OliVar-05},\cite{BovHol-15}) and have been found with powerful applications in e.g. statistical physics, genetics and stochastic algorithms (\cite{MorSaw-89}, \cite{Kushner-87}, \cite{FreKo-17}, \cite{AzIuMaMe-25}). Particularly, Freidlin-Wentzell's exit-time estimates for self-stabilizing diffusions have been originally addressed by S.~Herrmann, P.~Imkeller and D. Peithmann in their work~\cite{HIP}. Therein, for $-b$ a globally convex potential and $\Phi$ satisfying $\mathbf{(A)}$-$(iii)$, the authors extend the LDP \eqref{LDP} 
to the class of self-stabilizing diffusions:
\begin{equation*}
Z^\sigma_t=z+\int_0^t\Big(b\left(Z^\sigma_s\right)-\int \Phi\left(Z^\sigma_s-z\right)\,\mu^{Z^\sigma}_s(\ddz)\Big)\dds+\sigma \mathcal W_t,\quad 
\mu^{Z^\sigma}_t=\text{Law}(Z^\sigma_t),\qquad t\ge 0\,,
\end{equation*}
 and establish an analogue of
Theorem~\ref{thm:KramersDZ}, 
 there referred to as a Kramers' type law. Specifically, for any $\mathcal G$ stable set by $b$ and for $z_{stable}$ the unique attractor of $\Psi$ in $\overline{\mathcal G}$, the zero-noise asymptotic for the exit-time
$\tau_{\mathcal G}(\sigma):=\inf\{t\ge 0\,:\,Z_t\notin \mathcal G\}$ is given by
\begin{equation}\label{KramersSelfStab}
  \forall \delta>0,\,\lim_{\sigma\to0}\PP\left(\exp\left[\frac{2}{\sigma^2}\left(\underline{H}-\delta\right)\right]<\tau_{\cGc}(\sigma)<
	\exp\left[\frac{2}{\sigma^2}\left(\underline{H}+\delta\right)\right]\right)=1\quad\Leftrightarrow\quad	\frac{\sigma^2}2\log\big(\tau_{\mathcal G}(\sigma)\big)\overset{\mathbb P}{\underset{\sigma\downarrow 0}{\asymp}}\underline H,
\end{equation}
for 
\[
\underline{H}:=\inf_{z\in \partial\mathcal G}\mathcal V(z),\quad\mathcal V(z)=\frac 1{4}\inf_{0\le T<\infty} \inf_{\phi}\left\{\int_{0}^{T} \Big\|\dot{\phi}_t-b(\phi_t)+\Phi(\phi_t-z_{stable})\Big\|^2
\, \ddt\,:\,\phi\in\mathcal
A^{z_{stable},z}([0,T])\right\}.
\]
Although the $\sigma\downarrow 0$-limit of $Z^\sigma$ and $z^\sigma$ remain unchanged, the nonlinear component $\Phi$ yields to a damping of the exit-time compared to \eqref{GenericKramers}. In a series of works, the second author of this paper has extended these Kramers' type laws to the situation of a non-globally convex
reversible potential, under  a synchronized regime or weaker assumptions (see~\cite{Tugaut2018,Alea,Tugaut2021} and references therein) as well as
for stochastic particle systems (\cite{Tugaut2020}). The key strategy displayed in these papers consists in a coupling technique asserting that, after a certain (deterministic) time, the self-stabilizing diffusion can be found arbitrarily close to its "linear version" set at the stable force $b(z)-\Phi(z-z_{stable})$. Herein, this coupling technique will be extensively extended and adapted to our set-up. Let us point out that these strategies have been revisited and strongly extended in the recent paper \cite{AleVil-25}.
 
\subsection{Main results}\label{sec:mainresults}
For the self-stabilizing diffusions \eqref{MV1}-\eqref{MV2}, the zero-noise limit, for a negligible collision-radius $\varepsilon$, of the near collision-time $C_\varepsilon(\sigma)$ and near collision-location $L_\varepsilon(\sigma)$ in \eqref{near-collision-time-SS}  are given as follows.
\begin{thm}
\label{thm:main1} Let $\underline{H}_0$ be the minimum of 
\begin{equation}\label{collisioncost}
	H_0:\lambda\in \mathbb R^d\rightarrow H_0(\lambda)=\inf_{-\infty<T_1\leq T_2<\infty}\inf_{\phi}\left\{I(T_1,T_2,\phi)\,:\,\phi\in \mathcal A^{(\lambda_1,\lambda_2),(\lambda,\lambda)}([T_1,T_2])\right\}\,,	
\end{equation}
where $I$ is the action functional:
\begin{align}\label{Main_SS_AF}
	I(T_1,T_2,\phi)=\frac 1{4}\int_{T_1}^{T_2}\left(\Big\|\dot{\phi}^1_t+\W_{\lambda_1}(\phi^1_t)\Big\|^2_{A^{-1}(\phi^1_t)}+\Big\|\dot{\phi}^2_t+\W_{\lambda_2}(\phi^2_t)\Big\|^2_{A^{-1}(\phi^2_t)}\right)
	\,\ddt\,,	
\end{align}
and $\mathcal A^{(\lambda_1,\lambda_2),(\lambda,\lambda)}([T_1,T_2])$ is the set of all $\mathbb R^{2d}$-valued $\mathcal C^1$-paths $\phi=\{(\phi^1_t,\phi^2_t)\}_{t\in[0,T]}$ starting at $(\lambda_1,\lambda_2)$ at time $T_1$ and ending at the common point $(\lambda,\lambda)$ at time $T_2$.

\medskip 
\noindent
Then, for any $\delta>0$, it holds that
\begin{equation*}
\lim_{\varepsilon\to0}\lim_{\sigma\to0}\PP
\left(\exp\left[\frac{2}{\sigma^2}\left(\underline{H}_0-\delta\right)\right]
<{C}_\varepsilon(\sigma)
<\exp\left[\frac{2}{\sigma^2}\left(\underline{H}_0+\delta\right)\right]\right)=1\,.
\end{equation*}
Moreover, $L_\varepsilon(\sigma)=(X_{C_{\varepsilon}(\sigma)},Y_{C_\varepsilon(\sigma)})$ persists asymptotically in the vicinity of the minimizers set of $H_0$:
\begin{equation}\label{CollisionPersistence}
\mathcal M_0=\argmin H_0:=\{\lambda_0\in\mathbb R^d\,:\,H_0(\lambda_0)=\underline H_0\},
\end{equation}
with, for any $\delta>0$,  
\begin{align*}
&\lim_{\varepsilon\to0}\lim_{\sigma\to0}\PP\left(\dist\bigg( L_{{C}_\varepsilon(\sigma)},\mathcal M_0\times\mathcal M_0\bigg)\leq\delta\right)\\
&=\lim_{\varepsilon\to0}\lim_{\sigma\to0}\PP\left(\max\bigg(\dist( X_{{C}_\varepsilon(\sigma)},\mathcal M_0),\dist(
	Y_{{C}_\varepsilon(\sigma)},\mathcal M_0)\bigg)\leq\delta\right)=1.
\end{align*}
\end{thm}
Estimates for the individual near collision of the particle systems \eqref{particles} are analogue and given by the following.
\begin{thm}\label{thm:main2} For $\underline{H}_0$ and $\mathcal M_0$ as in Theorem \ref{thm:main1}, and for any $1\le i\le N$ with $N$ large enough, and any $\delta>0$, it holds that
	\begin{equation*}
		\lim_{\varepsilon\to0}\lim_{\sigma\to0}\PP\left(\exp\left[\frac{2}{\sigma^2}\left(\underline{H}_0-\delta\right)\right]<\mathcal{C}^i_{\varepsilon,N}(\sigma)<
		\exp\left[\frac{2}{\sigma^2}\left(\underline{H}_0+\delta\right)\right]\right)=1\,,
	\end{equation*}
	and
	\begin{align*}
			&\lim_{\varepsilon\to0}\lim_{\sigma\to0}\PP\left(\dist\bigg( L^{i }_{\varepsilon,N}(\sigma),\mathcal M_0\times\mathcal M_0\bigg)\leq\delta\right)\\
			&=\lim_{\varepsilon\to0}\lim_{\sigma\to0}\PP\left(\max\bigg(\dist\big(
		X_{\mathcal{C}^i_{\varepsilon,N}(\sigma)}^{i,N},\mathcal M_0\big),\dist\big(
		Y_{\mathcal{C}^i_{\varepsilon,N}(\sigma)}^{i,N},\mathcal M_0\big)\bigg)\leq\delta\right)=1\,.\nonumber
	\end{align*}
\end{thm}

\begin{rem}\label{rem:examples} In the gradient case and provided that $\Gamma={\rm I}_d$, the function $H_0$ in Theorems \ref{thm:main1} and \ref{thm:main2} adopts a convenient explicit and $\underline H_0$ is more directly interpretable in terms of $\lambda_1$ and $\lambda_2$: assuming $\U=\nabla u$, $\F=\nabla f$ (and so $\W_\mu(x)$ is the gradient $\nabla w_{\mu}(x)=\nabla u(x)+\int \nabla f(x-z)\mu(\ddz)$), $H_0$ writes as
\begin{align*}
H_0(\lambda)&=\frac 1{4}\inf_{\substack{-\infty<T_1\leq T_2<\infty\\ \phi}}\left\{\int_{T_1}^{T_2}\left(\Big\|\dot{\phi}^1_t+\nabla w_{\lambda_1}(\phi^1_t)\Big\|^2+\Big\|\dot{\phi}^2_t+\nabla w_{\lambda_2}(\phi^2_t)\Big\|^2\right)\ddt\,:\,\phi\in \mathcal A^{(\lambda_1,\lambda_2),(\lambda,\lambda)}([T_1,T_2])  \right\}\\
&=w_{\lambda_1}(\lambda)-w_{\lambda_2}(\lambda)\,,
\end{align*}
where $w_\lambda(x):=u(x)+f(x-\lambda)$. The related minima, that is the \textit{collision-cost} $\underline H_0$, is so
\begin{equation}\label{collisioncost_reduced}
\underline H_0=\inf_{\lambda\in\mathbb R^d}\Big(2u(\lambda)-u(\lambda_1)-u(\lambda_2)+f(\lambda-\lambda_1)+f(\lambda-\lambda_2)\Big)\,.
\end{equation}
Recalling that the synchronization assumption $({\bf A})$-$(iv)$ ensures the uniform
convexity of $w_\lambda$ for any $\lambda$ in $\mathbb R^d$, the minimizing set $\mathcal M_0$ reduces to a single element:
\begin{equation*}\label{collisionlocation}
	\lambda_0=
	\Big(\nabla w_{\lambda_1}+\nabla w_{\lambda_2}\Big)^{-1}(0)\,.
\end{equation*}
If further $f(x)=\alpha\|x\|^2/2$, $\nabla w_{\lambda_i}(x)=\big(\nabla u(x)+\alpha x\big)-\alpha \lambda_i$ and the above eventually gives
\begin{equation*}\label{collisionlocation_example}
\lambda_0=
\Big(\nabla u+\alpha {\rm Id}\Big)^{-1}\bigg(\frac{\alpha}{2}(\lambda_1+\lambda_2)\bigg),\qquad {\rm Id}\mbox{, the identity map on }\mathbb R^d.
\end{equation*}
From this expression, simple symmetry of the landscape enables to reduce the collision-location; e.g. for a perfectly symmetrical bistable landscape where $\lambda_1=-\lambda_2$, $\lambda_0$ becoming the zero of $x\mapsto \nabla u(x)+\alpha x$. This may be further exemplified in the case $d=1$

For the one-dimensional double-wells $u(x)=x^4/4-x^2/2$: $\lambda_1=-1$, $\lambda_2=1$, synchronization holds for $\alpha>1$ and the near first collision-location
persists at the unstable point $\lambda_0=0$ (independently of $\alpha$). The related collision-cost is given by $\underline H_0=H_0(\lambda_0)=\alpha+1/2$ so higher is the synchronization, higher will be the collision-cost. 

For the asymmetric double-wells potential $u(x)=x^4/4+x^3/3-x^2/2$: the wells are located at the points $\lambda_1=-1/2-\sqrt{5}/2$ and $\lambda_2=-1/2+\sqrt{5}/2$ -- the former defining the lowest energy well --, and, for a synchronization factor~$\alpha>4/3$, $\lambda_0=(u'+\alpha {\rm Id})^{-1}(-\alpha/2)$ corresponds to the root of
~$u'(x)+\alpha (x+1/2)=x^3+x^2+(\alpha-1)x+\alpha/2$ (such roof is notably distinct from the saddle point~$0$). For $\alpha$ close to $4/3$, $\lambda_0$ is found numerically close to $-1.19$ and for large value of $\alpha$, this root numerically concentrates around $-0,5$. (These one-dimensional cases apply to the {\it exact} first collision-time and the {\it exact} first collision-location, see Theorem  \ref{victoire2} in Appendix \ref{sec:1DCase}). The collision-location for these two one-dimensional examples were illustrated in Figure \ref{Fig1} - the asymmetric case in the right-hand side being set with $\alpha$ close to $4/3$.

In dimension two, the potential 
$$
u(x_1, x_2):= \frac{3}{2}\left(1 - x_1^2 - x_2^2\right)^2+\frac{1}{3}\left(x_1^2 - 2\right)^2+\frac{1}{6}\left((x_1 + x_2)^2 - 1\right)^2 +\frac{1}{6}\left((x_1 - x_2)^2-1\right)^2\,
$$
features an example of a symmetric double-wells potential with a gradient satisfying the related working assumptions $(\mathbf{A})$. The wells are located in $(0,1)$ and $(0,-1)$ and a saddle-point lies in $(0,0)$. The synchronization $({\bf A})$-$(iii)$ occurs providing that $\alpha>10$, the collision persists (again) in $\mathcal M_0=\{\lambda_0\}=\{(0,0)\}$ yielding (also again) to the collision-cost $\underline{H}_0=\alpha+1/2$.

\end{rem}

Theorem \ref{thm:main1} is consistent with one would expect from a naive application of \cite{HIP} Freidlin-Wentzell estimates~(\eqref{KramersSelfStab}) for self-stabilizing diffusions following, as discussed earlier, the interpretation of $C_\epsilon(\sigma)$ as an exit-time from the domain $(\mathbb R^d\times\mathbb R^d)\setminus\triangle_\varepsilon$ (with $\triangle_\varepsilon$ as introduced in \eqref{NaiveCollisionDomain} above) and $L_\epsilon(\sigma)$ as an exit-location. Indeed, taking~$m=2d$, $\mathcal W=(B,\widetilde{B})$, $b=(-\U,-\U)$, $\Phi=(\F,\F)$ and~$\gamma$ is the~$\mathbb R^{2d\times 2d}$-diagonal block matrix with $\Gamma$ as its main component, and provided that the domain~$(\mathbb R^d\times\mathbb R^d)\setminus\triangle_\varepsilon$ was stable, $C_\varepsilon(\sigma)$ would
obey to the Kramers' type law \eqref{KramersSelfStab} with the exit-cost $\inf_{(x,y)\in\partial \triangle_\varepsilon} V(x,y)$, and all possible exit-locations from $(\mathbb R^d\times\mathbb R^d)\setminus\triangle_\varepsilon$ situated in the
corresponding set of minimizers, $\argmin_{(x,y)\in\partial\triangle_\varepsilon} V$. The boundary $\partial \triangle_\varepsilon$ corresponding to the set~$\{(x,y)\in \mathbb
R^d\times \mathbb R^d\,:\,\|x-y\|=2\varepsilon\}$, as $\varepsilon$ is taken smaller and smaller, $\inf_{(x,y)\in\partial \triangle_\varepsilon} V(x,y)$ would become $\inf_{x\in\mathbb R^d} V(x,x)$, and yield exactly to the cost $\underline{H}_0=\inf H_0$ and to the clustering of the exit-locations in $\mathcal M_0$, as stated in Theorem \ref{thm:main1}.

Nevertheless this direct application of \eqref{KramersSelfStab} fails because of the lack of stability of $(\mathbb R^d\times\mathbb R^d)\setminus\triangle_\varepsilon$ for the force field $(-\U,-\U)$. A direct example is given by taking two orbits \eqref{eq:LimitDynamic} starting at two different points located at a distance more than $\varepsilon$ from each other but both lying in the same basin of attraction. If the basin is $\mathcal G(\lambda_1)$, the orbits are both attracted to $\lambda_1$. Indeed, in this setting, the dynamical system converges to $(\lambda_1,\lambda_1)$, which is not in $(\mathbb R^d\times\mathbb R^d)\setminus\triangle_\varepsilon$.
 
 The technical difficulty for a rigorous demonstration of Theorem \ref{thm:main1} first requires to tackle $C_\varepsilon(\sigma)$ 
  more delicately. To do so, one can observe that $C_\varepsilon(\sigma)$ can be written as \begin{equation}\label{KeyInterpretation}
 C_\varepsilon(\sigma)=\inf_{\lambda}\beta_{\lambda,\varepsilon}(\sigma)
 \end{equation}
  for $\beta_{\lambda,\varepsilon}(\sigma)$ the first time $X$ and $Y$ both enter the ball of radius $\mathbb B(\lambda;\varepsilon)$. Equality \eqref{KeyInterpretation} follows, on one side, from the triangular inequality yielding $C_\varepsilon(\sigma)\leq\inf_{\lambda}\beta_{\lambda,\varepsilon}(\sigma)$, and observing, on the other side, that
 $C_\varepsilon(\sigma)\geq\beta_{\lambda_{\varepsilon,\sigma},\varepsilon}(\sigma)\geq\inf_{\lambda}\beta_{\lambda,\varepsilon}(\sigma)$,
 for $\lambda_{\varepsilon,\sigma}:=2^{-1}(X_{C_\varepsilon(\sigma)}+Y_{C_\varepsilon(\sigma)})$.
 
 Equivalently, $\beta_{\lambda,\varepsilon}(\sigma)$ corresponds to the first exit-time of $(X,Y)$ from
 $$
 \left(\mathbb R^d\times\mathbb R^d\right)\setminus \big(\mathbb B(\lambda;\varepsilon)\times \mathbb B(\lambda;\varepsilon)\big)=\Big(\mathbb B(\lambda;\varepsilon)\times \mathbb B(\lambda;\varepsilon)\Big)^c.
 $$
 Although this domain is still not enough to guarantee stability for even applying core Freidlin-Wentzell's estimates (Theorem \ref{thm:KramersDZ}), $\beta_{\lambda,\varepsilon}(\sigma)$ provides a convenient proxy which allows to temporarily  parametrize the near collision-location of $(X,Y)$ at a given point $\lambda$. Up to an additional (non trivial) enlargement of the domain $\mathbb B(\lambda;\varepsilon)\times \mathbb B(\lambda;\varepsilon)$ and thanks to the continuity property of the related exit-costs, Freidlin-Wentzell's estimates are eventually expected to hold to successively to $\beta_{\lambda,\varepsilon}(\sigma)$ and next $C_\varepsilon(\sigma)$.  
 
The rest of the paper, dedicated to the proof of our main results is organized as follows. The above strategy is independent of the framework of self-stabilizing diffusions and is carried out in Section \ref{sec:LinearCase} in a preliminary case of two simplified measure-independent versions of \eqref{MV1}-\eqref{MV2}. Therein, we establish a Freidlin-Wentzell' type collision laws for the first near collision between two independent non-reversible diffusions, each driven by a different uniformly contractive force. Asymptotics are obtained for small $\varepsilon>0$ (Proposition \ref{lacollision}) and next descended to the limit $\varepsilon\downarrow 0$ (Theorem \ref{lacollisionthm}).  
 These preliminaries give a sufficient basis and guideline to handle self-stabilizing diffusions and particle systems after extending key coupling techniques from~\cite{Alea,Tugaut2021}.

Theorem~\ref{thm:main1} is demonstrated in Section~\ref{sec:SelfStabilizingCase}, with intermediate zero-noise asymptotics for fixed small $\varepsilon>0$ in Proposition \ref{label_bizarre8}. For the particle systems \eqref{MFSP1}-\eqref{MFSP2}, Theorem~\ref{thm:main2} is established in Section \ref{sec:ParticleCase}, with (again) intermediate results,  Proposition \ref{ned}, for small $\varepsilon$. 
Let us briefly point out that, if an uniform-in-time propagation of chaos result from \eqref{MFSP1}-\eqref{MFSP2} to \eqref{MV1}-\eqref{MV2} was holding true, for $N$ large enough, the closeness  between $(\mu^X_t,\mu^Y_t)$ and $(\lambda_1,\lambda_2)$ would naturally be preserved by  $(\frac{1}{N}\sum_{j=1}^N\delta_{X^{j,N}_t},\frac{1}{N}\sum_{j=1}^N\delta_{Y^{j,N}_t})$ and Theorem \ref{thm:main2} could be derived from Theorem \ref{thm:main1}. However,  to the best of our knowledge, our current model assumptions only allow (
see e.g. \cite{BRTV}) a local-in-time propagation of chaos
of the form: for any finite time-horizon $T$, and, for $(X^i,Y^i)$ the
copies of \eqref{MV1} and \eqref{MV2} generated by the noise $(B^i,\widetilde{B}^i)$,
\begin{equation*}
	\sup_{0\leq t\leq T}\mathbb{E}\left[\|X_t^i-X_t^{i,N}\|^2\right]+\sup_{0\leq t\leq
		T}\mathbb{E}\left[\|Y_t^i-Y_t^{i,N}\|^2\right]\le \frac {C(T)}{N},
\end{equation*}
the constant $C(T)$ growing exponentially w.r.t. $T$. Such dependency would naturally be avoided in the instance $\U=\nabla u$ with $u$ strictly convex or under global contracting one-sided Lipschitz condition for $-\U$ (e.g. \cite{CGM}). But this setting naturally is outside our model of interest. In the same way, a better uniform-in-time estimate (owing to an asynchronous coupling method, see e.g. \cite{DEGZ} and references therein) may be derived but this would be achieved (again to the best of our knowledge) at the cost of a non-vanishing noise. Also, albeit in the reversible case with $\Gamma=I_d$, one could take advantage of the work \cite{Monmarche25} in which the exponential growth (in $N$) of the first time $\tau^N$ such that the empirical distribution $\frac1{N}\sum_{i=1}^{N}\delta_{X^{i,N}_t}$ ``moves away'' from one (possible) invariant probability of the system, is established. This would 
grant the key Proposition \ref{label_bizarre9} and greatly simplify the proof of Theorem \ref{thm:main2}. Such argument nevertheless would require to extend \cite{Monmarche25} to our present 
framework.

The appendix section~\ref{Appendix-A} is dedicated to the with the proof of the well-posedness and moment estimates, stated in Section \ref{ssec:Wellposed-Moment}, for both self-stabilizing diffusions and particle systems.

Finally, the appendix section~\ref{sec:1DCase} discusses the one-dimensional situation where the exact first collision-times
\[
C(\sigma)=\inf\{t\ge 0\,:\,X_t=Y_t\}, \ \ \ C^i_N(\sigma)=\inf\{t\ge 0\,:\,X^{i,N}_t=Y^{i,N}_t\}\,,
\]
are well-defined. Analogues of Theorems~\ref{thm:main1} and~\ref{thm:main2} (for $\varepsilon=0$) 
 are there 
 established, with more direct proof arguments.

To conclude this section, we may (re-)note that extensions of our framework to integrate the presence of multi-wells landscapes, cross-interactions between the systems \eqref{MV1}-\eqref{MV2} (and between \eqref{MFSP1}-\eqref{MFSP2}) or the question of estimating the first global near collision-time and -location $\inf_{i}C^i_{\varepsilon,N}(\sigma)$ and 
$\inf_{i}L^i_{\varepsilon,N}(\sigma)$  of the particle systems are natural interesting questions which raise technical challenges beyond the prototypical framework aimed with the present paper. We plan to address these generalizations in future works.

\section{On the first near collision between two independent stochastically perturbed flows}\label{sec:LinearCase}

In this section, we establish the zero-noise asymptotic of the near collision-time and near-collision-location:
\begin{equation*}
c_\varepsilon(\sigma):=\inf\left\{t\ge0\,\,:\,\,\| x^\sigma_t-y^\sigma_t\| \leq2\varepsilon\right\}\,,\quad l_\varepsilon(\sigma)=\big(x^\sigma_{c_\varepsilon(\sigma)},y^\sigma_{c_\varepsilon(\sigma)}\big), \qquad \varepsilon>0\,,
\end{equation*}
related to the pair of diffusion processes:
\begin{equation}
\label{Linear-Systems}
\left\{
\begin{aligned}
&x^\sigma_t=x_0 -\int_0^t\Psi_1\left(x^\sigma_s\right)\dds+\sigma \int_0^t \Gamma(x^\sigma_s)\,\dd B_s,\qquad t\ge 0\,,\\
&y^\sigma_t=y_0 -\int_0^t\Psi_2\left(y^\sigma_s\right)\dds+\sigma\int_0^t \Gamma(y^\sigma_s)\,\dd\widetilde{B}_s,\qquad t\ge 0\,.
\end{aligned}
\right.
\end{equation}
Throughout this section, our main assumptions on \eqref{Linear-Systems} are as follows: $\Gamma$ satisfies $\mathbf{(A)}$-$(vi)$ and, $-\Psi_1$ and $-\Psi_2$ are of class $\mathcal C^1$ and are uniformly contractive
 one-sided Lipschitz continuous with, for $k=1,2$, and some $\theta_k>0$,
 \begin{equation}\label{ContractiveOneSided}
 -\Big(\Psi_k(x)-\Psi_k(y)\Big)\cdot (x-y)\le -\theta_k\|x-y\|^2,\qquad x,y\in\mathbb R^d\,.
 \end{equation}
 Additionally, we assume that the vector fields $\Psi_1$ and $\Psi_2$ have each exactly one distinct asymptotically stable point,~$\lambda_1$ for $\Psi_1$ and $\lambda_2$ for $\Psi_2$, with $\lambda_1\neq\lambda_2$ and that the initial states, $x_0$ and $y_0$, are distinct, and such that the orbits
\[
\varphi_t^{1,-}(x_0)=x_0-\int_0^t\Psi_1\left(\varphi_s^{1,-}(x_0)\right)\dds\,,\quad
\varphi_t^{2,-}(y_0)=y_0-\int_0^t\Psi_2\left(\varphi_s^{2,-}(y_0)\right)\dds\,,\qquad t\ge 0\,,
\]
never hit each other at any time $t$. Specifically, we assume that 
\begin{equation}\label{Linear:Init}
	\varepsilon_0:=\frac 1{2}\inf_{t\geq0}\|\varphi_t^{1,-}(x_0)-\varphi_t^{2,-}(y_0)\|> 0,
\end{equation}
the threshold $\varepsilon_0$ including the minimal distance between  initial states, $x_0$ and $y_0$, and between the attracting points, $\lambda_1$ and $\lambda_2$, making further the collision between $\{\varphi^{1,-}_t(x_0)\}_{t\ge 0}$ and $\{\varphi^{2,-}_t(y_0)\}_{t\ge 0}$ non-trivial.
Indeed, if $\varepsilon$ was chosen larger than $\varepsilon_0$, then, in view of the large deviation principle~\eqref{LDP}, there would exist
$\sigma_0>0$ and $T_0<\infty$ such that $\mathbb P\big(c_\varepsilon(\sigma)\le T_0\big)=1$ for any $\sigma\le \sigma_0$ and $c_\varepsilon(\sigma)$
would be bounded a.s..

The one-sided Lipschitz conditions \eqref{ContractiveOneSided} ensure that the flows generated by $\Psi_1$ and $\Psi_2$ converge respectively to $\lambda_1$ and $\lambda_2$ independently of the initial states. Equivalently the reversed orbits started from neighborhoods of $\lambda_1$ and $\lambda_2$ span the whole space.

As outlined in Section \ref{sec:mainresults}, our strategy to establish the asymptotic behavior of the collision-time $c_\varepsilon(\sigma)$ and of the collision-location $\big(x^\sigma_{c_\varepsilon(\sigma)},y^\sigma_{c_\varepsilon(\sigma)}\big)$ lies with the interpretation
$c_\varepsilon(\sigma)=\inf_{\lambda\in\bRb^d}\tau_{\lambda,\varepsilon}(\sigma)$
for $\tau_{\lambda,\varepsilon}(\sigma)$ defining the first time $x^\sigma$ and $y^\sigma$ simultaneously reach the closed ball $\overline{\mathbb B(\lambda;\varepsilon)}$ for a given point $\lambda$ of $\mathbb R^d$, namely
\begin{equation}
\label{label_bizarre1}
\tau_{\lambda,\varepsilon}(\sigma):=\inf\left\{t\ge 0\,\,:\,\,\|x^\sigma_t-\lambda\|\leq\varepsilon\mbox{ and
}\|y^\sigma_t-\lambda\|\leq\varepsilon\right\}\,.
\end{equation}
Equivalently
$\tau_{\lambda,\varepsilon}(\sigma)$
corresponds to the first time $(x^\sigma,y^\sigma)$ exits from the domain
$$
\left(\bRb^d\times\bRb^d\right)\setminus\left(\overline{\mathbb{B}(\lambda;\varepsilon)}\times\overline{\mathbb{B}(\lambda;\varepsilon)}\right)=\left(\overline{\mathbb{B}(\lambda;\varepsilon)}\times\overline{\mathbb{B}(\lambda;\varepsilon)}\right)^c.
$$
From this interpretation and up to an additional enlargement $\mathbb B(\lambda;\varepsilon)\times \mathbb B(\lambda;\varepsilon)$ 
to make the exit-domain suitable for an application of Freidlin-Wentzell exit-time estimates, below we successively: determine, for any $\lambda\in\mathbb R^d$ and $\varepsilon>0$, the zero-noise asymptotic of $\tau_{\lambda,\varepsilon}(\sigma)$ and the parameterized collision-location~$(x^\sigma_{\tau_{\varepsilon,\lambda}(\sigma)},y^\sigma_{\tau_{\lambda,\varepsilon}(\sigma)})$ (see Lemma~\ref{dale-bis} after);
deduce next the limiting behavior for $c_\varepsilon(\sigma)$ and~$(x^\sigma_{c_\varepsilon(\sigma)},y^\sigma_{c_\varepsilon(\sigma)})$ for
small enough $\varepsilon>0$ (Proposition~\ref{lacollision}); and finally conclude on the related asymptotic at the limit
$\varepsilon\downarrow 0$ (Theorem~\ref{lacollisionthm}).

\begin{rem}
A critical step in establishing the main Proposition~\ref{lacollision} and Theorem~\ref{lacollisionthm} of this section stands with Lemma~\ref{dale-bis}, which highlights key Freidlin-Wentzell's exit-laws for the domain $\big(\mathbb B(\lambda;\varepsilon)\times \mathbb B(\lambda;\varepsilon)\big)^c$. The proof of Lemma~\ref{dale-bis} is intricate and occupies most of the next section but directly addresses the ineludible issue arising with adapting Theorem \ref{thm:KramersDZ} to the specific exit-domain $\big(\mathbb B(\lambda;\varepsilon)\times \mathbb B(\lambda;\varepsilon)\big)^c$ with $\lambda$ close or away for the attractors of $x^\sigma$ and $y^\sigma$. 
Regarding this issue, we may briefly discuss an alternative and potentially valuable strategy suggested by an anonymous referee during the review process of the earlier version \href{arXiv:2206.04542}{ArXiv} of the present work, that we unfortunately did not have the opportunity to address at that time. This strategy proposes, at first glance, a simpler and more elegant than our approach, but ultimately runs into the same fundamental difficulty of extending Freidlin-Wentzell theory to the domain $\big(\mathbb B(\lambda;\varepsilon)\times \mathbb B(\lambda;\varepsilon)\big)^c$.  

The key idea of this possible alternative is to exploit the re-normalized exponential exit-law obtained in the seminal paper of  M.~V.~Day \cite{Day}. These asymptotic specifically (see \cite[Corollary 2 and Theorem 4]{Day}) state that, under additional suitable smoothness ($\mathcal C^3$) of an exit-domain $\cGc$, the normalized first exit-time $\frac{\tau_\cGc(\sigma)}{\EE[\tau_\cGc(\sigma)]}$ weakly converges towards an exponential law of parameter $1$ as $\sigma$ tends to $0$ (the results of \cite{Day} hold for the non-reversible inhomogeneous diffusion setting \eqref{GenSGF}).
		
Then, to obtain the first time that the independent diffusions $x^\sigma$ and $y^\sigma$ reach simultaneously the ball  $\mathbb B(\lambda;\varepsilon)$, one may study the different excursions between its equilibrium and this ball for one given diffusion, say $x^\sigma$. Denoting by $\tau^{1,\ell}_{\lambda,\varepsilon}(\sigma)$ the $\ell$-th time that $x^\sigma$ reaches the ball, and assuming moreover that the law of the second diffusion converges, at large time, exponentially fast to its invariant probability measure, one can estimate the probability that, simultaneously, $y^\sigma$ lies in $\mathbb B(\lambda;\varepsilon)$ at time $\tau^{1,\ell}_{\lambda,\varepsilon}(\sigma)$. As $x^\sigma$ and $y^\sigma$ are independent, this probability is proportional to $\exp[-2H_2/\sigma^2]$ where $H_2$ is nothing else than the exit-cost of $y^\sigma$ from $\mathbb R^d\setminus \mathbb B(\lambda;\varepsilon)$. Using a geometric law, one may then obtain that the first time when the two diffusions are in $\mathbb B(\lambda;\varepsilon)$ is given by $\tau^{1,L}_{\lambda,\varepsilon}(\sigma)$ with $L$ exponentially equivalent to $\exp[+2H_2/\sigma^2]$. Thus, by independence of the excursions (due to the Markov property of the systems), one could infer the result stated in Lemma~\ref{dale-bis} (for some very small $\varepsilon$) and all the subsequent results of this section.

This proof would so essentially propose an elegant coupling between $x^\sigma$ and $y^\sigma$ to estimate $\tau_{\lambda,\varepsilon}(\sigma)$ rather than the extensive \textit{domain enlargement} approach we develop here. Yet, in order to proceed with this coupling approach, one would need to apply the result in~\cite{Day} and so would need \emph{again} to handle the general lack of stability of the exit-domain~$\mathbb R^d\setminus\mathbb B(\lambda;\varepsilon)$. Hence one would face \emph{again} the question of constructing a suitable close approximation of the domain, which is stable by the drift $-\Psi_1$ and $\Psi_2$, and is also smooth enough to apply \cite{Day}. This would not simplify the proof of the aforementioned lemma since its crux is to construct such suitable domain. Additionally, the alternative proof would also require to showing the exponential convergence of the law of one of the diffusion towards a stationary state, which is achievable only in the reversible setting.
\end{rem}

\subsection{Asymptotic estimates for \texorpdfstring{$\tau_{\lambda,\varepsilon}(\sigma)$}{tau(lambda,epsilon,sigma)}}
\label{subsec:Linear-collisionA}

As (mentioned previously) the
set~$\left(\overline{\mathbb{B}(\lambda;\varepsilon)}\times\overline{\mathbb{B}(\lambda;\varepsilon)}\right)^c$
is not necessarily stable by $(-\Psi_1,-\Psi_2)$, Theorem~\ref{thm:KramersDZ} cannot be directly applied to deduce the
asymptotic of $\tau_{\lambda,\varepsilon}(\sigma)$. This technical difficulty can be bypassed by a suitable two-steps modification of the
exit-set~$\left(\overline{\mathbb{B}(\lambda;\varepsilon)}\times\overline{\mathbb{B}(\lambda;\varepsilon)}\right)^c$.

 As a first modification, let us consider the sets
\begin{equation*}
\mathcal{D}_{\lambda,\varepsilon}^1:={\left\{\varphi_{T_0,t}^{1,+}(x)\,:\,-\infty<T_0\le t<\infty
,\,x\in\overline{\mathbb{B}\left(\lambda;\varepsilon\right)}\right\}}\,,\,\, 
\mathcal{D}_{\lambda,\varepsilon}^2:={\left\{\varphi_{T_0,t}^{2,+}(y)\,:\,-\infty<T_0\le t<\infty ,\,y\in\overline{\mathbb{B}\left(\lambda;\varepsilon\right)}\right\}}\,,
\end{equation*}
for $\varphi^{1,+}(x)$ and $\varphi^{2,+}(y)$ corresponding to {\it time-reversed versions} of $\varphi^{1,-}$ and $\varphi^{2,-}$:
\begin{equation*}
\varphi_{T_0,t}^{1,+}(x)=x+\int_{T_0}^t\Psi_1\left(\varphi_{T_0,s}^{1,+}(x)\right)\dds,\quad
\varphi_{T_0,t}^{2,+}(y)=y+\int_{T_0}^t\Psi_2\left(\varphi_{T_0,s}^{2,+}(y)\right)\dds,\qquad t\ge T_0
\,.
\end{equation*}
For $i=1,2$, $\mathcal{D}_{\lambda,\varepsilon}^i$ defines a closed enlargement of $\overline{\mathbb{B}(\lambda;\varepsilon)}$ consisting of all the  points
attainable by the family of flows $\{\varphi^{i,+}_{T_0,\cdot}(x)\}_{T_0,x}$ starting from $\overline{\mathbb{B}(\lambda;\varepsilon)}$. In particular, each element  $z=\varphi^{i,+}_{T_0,t}(x)=\varphi^{i,+}_{0,t-T_0}(x)$, $x\in\overline{\mathbb B(\lambda;\varepsilon)}$, corresponds to the starting point of an orbit $\varphi^{i,-}(z)$ reaching $\overline{\mathbb{B}(\lambda;\varepsilon)}$ after a period $t-T_0$ (by time-reversal, $0\le s\le t-T_0$, $\varphi^{i,-}_{s}(z)=\varphi^{i,+}_{t-T_0-s,0}(x)$). As such, the set $\mathcal{D}_{\lambda,\varepsilon}^i$ contains all the orbits $\varphi^{i,-}$ hitting the ball at some time, and, subsequently, the complementary~$\left(\mathcal{D}_{\lambda,\varepsilon}^i\right)^c:=\mathbb R^d\setminus\mathcal{D}_{\lambda,\varepsilon}^i$ corresponds to the starting points for which $\varphi^{i,-}$ never cross $\overline{\mathbb{B}(\lambda;\varepsilon)}$. By extension, $\left(\mathcal{D}_{\lambda,\varepsilon}^i\right)^c$ corresponds to the largest set stable by
$-\Psi_i$ and contained in $\overline{\mathbb{B}(\lambda;\varepsilon)}^c$.
The inherent stability property of each set further guarantees that
 the domain 
 $$
 (\mathcal D^{1}_{\lambda,\varepsilon}\times \mathcal D^{2}_{\lambda,\varepsilon})^c:=\left(\bRb^d\times \bRb^d\right)\setminus(\mathcal{D}_{\lambda,\varepsilon}^1\times\mathcal{D}_{\lambda,\varepsilon}^2)=\left(\bRb^d\times\left(\mathcal{D}_{\lambda,\varepsilon}^2\right)^c\right)\bigcup\left(\left(\mathcal{D}_{\lambda,\varepsilon}^1\right)^c\times\bRb^d\right),
 $$ 
is stable by~$(-\Psi_1,-\Psi_2)$. Indeed, for any element $(x',y')$
of~$\left(\bRb^d\times \bRb^d\right)\setminus(\mathcal{D}_{\lambda,\varepsilon}^1\times\mathcal{D}_{\lambda,\varepsilon}^2)$, either $x'$
lies in $\bRb^d\setminus\mathcal{D}_{\lambda,\varepsilon}^1$ or $y'$ lies in $\bRb^d\setminus\mathcal{D}_{\lambda,\varepsilon}^2$. Therefore, in
each case, at least one of the marginal domain is stable and so the domain is necessarily
stable by $(-\Psi_1,-\Psi_2)$. Observing that $(\lambda_1,\lambda_2)$ necessarily lies within $(\mathcal{D}_{\lambda,\varepsilon}^1\times\mathcal{D}_{\lambda,\varepsilon}^2)^c$ (for any $\lambda\in\mathbb R^d$), Theorem \ref{thm:KramersDZ} applies with the related exit-cost:
\[
\underline h=\tilde h_\varepsilon(\lambda):=\inf_{(x,y)\in \partial (\mathcal D^{1}_{\lambda,\varepsilon}\times \mathcal D^{2}_{\lambda,\varepsilon})^c}v(x,y)\,, 
\]
where
\[
v(x,y):=\inf_{-\infty<T_1\le T_2<\infty}\inf_{\phi}\left\{i_{T_1,T_2}(\phi)\,:\,\phi=(\phi^1,\phi^2)\in \mathcal A^{(\lambda_1,\lambda_2),(x,y)}[T_1,T_2]\right\}\,,
\]
\[
i_{T_1,T_2}(\phi)=\frac 1{4}\int_{T_1}^{T_2}\left(\Big\|\frac{\dd\phi^1_t}{\ddt}+\Psi_1(\phi^1_t)\Big\|^2_{
		(\Gamma\Gamma^*)^{-1}(\phi^1_t)}+
		\Big\|\frac{\dd\phi^2_t}{\ddt}+\Psi_2(\phi^2_t)\Big\|^2_{(\Gamma\Gamma^*)^{-1}(\phi^2_t)}\right)\,\ddt\,.
\]
While $\mathcal{D}_{\lambda,\varepsilon}^1\times \mathcal{D}_{\lambda,\varepsilon}^2$ is
larger than $\mathbb{B}(\lambda;\varepsilon)\times \mathbb{B}(\lambda;\varepsilon)$ (and so $\big(\mathcal D^{1}_{\lambda,\varepsilon}\times \mathcal D^{2}_{\lambda,\varepsilon}\big)^c\subset\big(\mathbb{B}(\lambda;\varepsilon)\times \mathbb{B}(\lambda;\varepsilon)\big)^c$), 
$\tilde h_\varepsilon$ can nevertheless be readily computed in two appropriate situations: $(a)$ when $\lambda$ is at a distance strictly larger than $\varepsilon$ from both wells; $(b)$ when $\lambda$
lies in a $\varepsilon$-neighborhood of one of the two wells.

\medskip

$\bullet$ \underline{For $(a)$}: whenever $\min_{i=1,2}(\|\lambda-\lambda_i\|)>\varepsilon$, the ball $\mathbb B(\lambda;\varepsilon)$ and the domains 
$\mathcal{D}_{\lambda,\varepsilon}^i$, $i=1,2$, do not contain any well. As such, any path connecting $(\lambda_1,\lambda_2)$ to $\overline{\mathbb B(\lambda;\varepsilon)}\times\overline{\mathbb B(\lambda;\varepsilon)}$  necessarily crosses the boundary of the domain $\overline{\mathbb B(\lambda;\varepsilon)}\times\overline{\mathbb B(\lambda;\varepsilon)}$. We deduce that the infimum of $v(x,y)$ for $(x,y)$ in $\overline{\mathbb B(\lambda;\varepsilon)}\times\overline{\mathbb B(\lambda;\varepsilon)}$ is equal to the infimum of the same function on $\partial(\mathbb B(\lambda;\varepsilon)\times \mathbb B(\lambda;\varepsilon))$. Therefore, recalling that  $\overline{\mathbb B(\lambda;\varepsilon)}\times\overline{\mathbb B(\lambda;\varepsilon)}\subset\mathcal{D}_{\lambda,\varepsilon}^1\times \mathcal{D}_{\lambda,\varepsilon}^2$, 
\[
\tilde h_\varepsilon(\lambda)=\inf_{(x,y)\in \mathcal D^{1}_{\lambda,\varepsilon}\times \mathcal D^{2}_{\lambda,\varepsilon}}v(x,y)\le \inf_{(x,y)\in \overline{\mathbb B(\lambda;\varepsilon)}\times \overline{\mathbb B(\lambda;\varepsilon)}}v(x,y)=\inf_{(x,y)\in \partial(\mathbb B(\lambda;\varepsilon)\times \mathbb B(\lambda;\varepsilon))}v(x,y).
\]
 On the other hand, for any $(x,y)$ in $\mathcal{D}_{\lambda,\varepsilon}^1\times\mathcal{D}_{\lambda,\varepsilon}^2$, the representation $(x,y)=(\varphi_{T_x,t_x}^{1,+}(x'),\varphi_{T_y,t_y}^{2,+}(y'))$ holds for some $x',y'$ in $\overline{\mathbb
	B(\lambda;\varepsilon)}$ - or equivalently $(x',y')=(\varphi_{t_x-T_x}^{1,-}(x),\varphi_{t_y-T_y}^{2,-}(y))$. 
	Since the action functional $i_{T_1,T_2}$ is zero when computed along the paths $\phi=(\varphi^{1,-},\varphi^{2,-})$ on any interval $[T_1,T_2]$,  the minimal cost of connecting $(\lambda_1,\lambda_2)$ to $(x',y')$ is less than the cost of connecting $(\lambda_1,\lambda_2)$ to $(x,y)$. In particular
	\[
	v(x,y)=v(\varphi_{T_x,t_x}^{1,+}(x'),\varphi_{T_y,t_y}^{1,+}(y'))\geq v(x',y'),
	\] 
which yields $\inf_{(x,y)\in \mathcal D^{1}_{\lambda,\varepsilon}\times \mathcal D^{2}_{\lambda,\varepsilon}}v(x,y)\ge\inf_{(x,y)\in \partial(\mathbb B(\lambda;\varepsilon)\times \mathbb B(\lambda;\varepsilon))}v(x,y)$.

\medskip

$\bullet$ \underline{For $(b)$}: In the case where $\lambda$ is in a close neighbourhood of one of the attracting points,
say~$\lambda_1$ with $\|\lambda-\lambda_1\|=:\tilde \varepsilon <\varepsilon$, then $\mathbb B(\lambda;\varepsilon)$ is stable
by $-\Psi_1$ and necessarily~$\mathcal{D}_{\lambda,\varepsilon}^1=\bRb^d$. Since~$\varepsilon<\varepsilon_0$,
$\lambda_2$ is located outside $\mathbb B(\lambda;\varepsilon)$. And since~$\varphi^{2,+}_{T,t}(\lambda_2)=\lambda_2$ for all
$t\ge T$, necessarily $\lambda_2\notin \mathcal D^1_{\lambda,\varepsilon}$ and $(\mathcal D^2_{\lambda,\varepsilon})^c$ is
stable by $-\Psi_2$. Replicating the same arguments as for $(a)$ (only for the marginal), the exit-cost related to $\big(\mathcal D^1_{\lambda,\varepsilon}\times\mathcal D^2_{\lambda,\varepsilon})^c$ is thus given by
\begin{align*}
\inf_{(x,y)\in\partial \big(\mathcal D^1_{\lambda,\varepsilon}\times\mathcal D^2_{\lambda,\varepsilon}\big)^c} v(x,y)&=\inf_{x\in\mathbb R^d,y\in\partial \mathcal D^2_{\lambda,\varepsilon}}v(x,y)=
\inf_{x\in\mathbb R^d,y\in\partial 	B(\lambda;\varepsilon)}v(x,y)\\
=&\frac 1{4}\inf_{y\in\partial 	B(\lambda;\varepsilon)}\inf_{-\infty<T_1\le T_2<\infty}\inf_{\phi}\left\{ \int_{T_1}^{T_2} 
\Big\|\frac{\dd\phi_t}{\ddt}+\Psi_2(\phi_t)\Big\|^2_{A^{-1}(\phi_t)}\,\ddt\,:\,\phi\in \mathcal A^{\lambda_2,y}[T_1,T_2]\right\}\,,
\end{align*}
recalling that $A=\Gamma\Gamma^{*}$.

The analogue can be naturally drawn in the case $\|\lambda-\lambda_2\|=\tilde \varepsilon$ with the resulting exit-cost:
\begin{align*}
	\inf_{(x,y)\in\partial \big(\mathcal D^1_{\lambda,\varepsilon}\times\mathcal D^2_{\lambda,\varepsilon}\big)^c} v(x,y)&=\inf_{x\in\partial \mathbb	B(\lambda;\varepsilon) ,y\in\mathbb R^d }v(x,y)\\
	&=\frac 1{4}\inf_{-\infty<T_1\le T_2<\infty}\inf_{\phi}\left\{ \int_{T_1}^{T_2} 
	\Big\|\frac{\dd\phi_t}{\ddt}+\Psi_1(\phi_t)\Big\|^2_{A^{-1}(\phi_t)}\,\ddt\,:\,\phi\in \mathcal A^{\lambda_1,x}[T_1,T_2]\right\}.
\end{align*} 
The remaining case ``$(c)$:  $\lambda$ is exactly at a distance $\varepsilon$ of $\lambda_1$ or $\lambda_2$'' (that is: one of the wells
is located at the boundary of $\mathbb B(\lambda;\varepsilon)$)
is a degenerate situation  where the exit-cost $\tilde h_\varepsilon$ cannot be simply identified. This issue can be overcame by slightly modifying $\mathcal
D^1_{\lambda,\varepsilon}\times \mathcal D^2_{\lambda,\varepsilon}$ into
\begin{equation*}
\mathcal O_{\lambda,\varepsilon,\rho}:=
\left\{
\begin{aligned}
&\mathcal D^1_{\lambda,\rho\varepsilon}\times \mathcal D^2_{\lambda,\varepsilon}\,\,\mbox{if}\,\| \lambda-\lambda_{1}\| =\varepsilon\,,\\
&\mathcal D^1_{\lambda,\varepsilon}\times \mathcal D^2_{\lambda,\rho\varepsilon}\,\,\mbox{if}\,\| \lambda-\lambda_{2}\| =\varepsilon\,,\\
&\mathcal D^1_{\lambda,\varepsilon}\times \mathcal D^2_{\lambda,\varepsilon}\,\,\mbox{otherwise}\,,
\end{aligned}
\right.
\end{equation*}
for $\rho$ arbitrarily chosen in the interval $(0,1)$. Rescaling $\varepsilon$ to $\rho \varepsilon$ whenever
$\|\lambda-\lambda_1\|=\varepsilon$ or~$\|\lambda-\lambda_2\|=\varepsilon$ ensures immediately $\mathcal O_{\lambda,\varepsilon,\rho}$ satisfies to the
situation $(a)$.

Considering separately the cases $\|\lambda-\lambda_1\|=\varepsilon$, $\|\lambda-\lambda_2\|=\varepsilon$, and the remaining one, and according to the discussion above, the set $\mathcal
O_{\lambda,\varepsilon,\rho}$ is stable by $(-\Psi_1,-\Psi_2)$. The related exit-cost
\[
\widehat \h^\rho_\varepsilon(\lambda):=\inf_{(x,y)\in\partial \mathcal
O_{\lambda,\varepsilon,\rho}} v(x,y)
\]
is equivalently given by
\begin{equation}\label{cost1}
\widehat \h^\rho_\varepsilon(\lambda)=\left\{
\begin{aligned}
	&\inf_{x\in\partial\mathbb B(\lambda;\rho\varepsilon),y\in\partial \mathbb B(\lambda;\varepsilon)} v(x,y)\quad\mbox{if}\quad\|
	\lambda-\lambda_1\|=\varepsilon\,,\\
	&\inf_{x\in\partial\mathbb B(\lambda;\varepsilon),y\in\partial \mathbb
		B(\lambda;\rho\varepsilon)} v(x,y)
	\quad\mbox{if}\quad\| \lambda-\lambda_2\|=\varepsilon\,,\\
	&\inf_{(x,y)\in\partial(\mathcal D^{1}_{\lambda,\varepsilon}\times\mathcal D^2_{\lambda,\varepsilon})^c} v(x,y)\quad\mbox{otherwise}\,.
\end{aligned}
\right.
\end{equation}

Applying Theorem~\ref{thm:KramersDZ},
we derive the following Freidlin-Wentzell' type law for the first exit-time
\begin{equation*}
\widehat{\tau}^\rho_{\lambda,\varepsilon}(\sigma):=\inf\left\{t\ge 0\,:\,(x_t^\sigma,y_t^\sigma)\notin (\mathbb R^d\times \mathbb
R^d)\setminus \mathcal O_{\lambda,\varepsilon,\rho}\right\}\,.
\end{equation*}
\begin{lem}\label{negan}
For any $\lambda$ in $\bRb^d$, $0<\varepsilon<\varepsilon_0$ (with $\varepsilon_0$ as in \eqref{Linear:Init}), $0<\rho<1$ and for any $\delta>0$,

\begin{equation}
\label{gouverneur1}
\lim_{\sigma\to0}\PP\left(\exp\left[\frac{2}{\sigma^2}\left(\widehat \h^\rho_\varepsilon(\lambda)-\delta\right)\right]<
\widehat{\tau}^\rho_{\lambda,\varepsilon}(\sigma)<\exp\left[\frac{2}{\sigma^2}\left(\widehat
\h^\rho_\varepsilon(\lambda)+\delta\right)\right]\right)=1\,.
\end{equation}
Moreover, we have:
\begin{equation}
\label{gouverneur2}
\lim_{\sigma\to0}\PP\left({\rm
dist}\left((x^\sigma_{\widehat{\tau}^\rho_{\lambda,\varepsilon}(\sigma)},y^\sigma_{\widehat{\tau}^\rho_{\lambda,\varepsilon}(\sigma)}),\mathbb{B}(\lambda;\varepsilon)\times\mathbb{B}(\lambda;\varepsilon)\right)\leq\delta\right)=1\,.
\end{equation}
\end{lem}

We stress that the three cases ($||\lambda-\lambda_1||=\epsilon$, $||\lambda-\lambda_2||=\epsilon$ and $\min\{||\lambda-\lambda_1||,||\lambda-\lambda_2||\}>\epsilon$) are taken into account in Equations~\eqref{gouverneur1}-\eqref{gouverneur2}.

 \begin{proof}
 
 The asymptotic \eqref{gouverneur1} is a direct consequence of \eqref{GenericKramers} in Theorem \ref{thm:KramersDZ}, and the
 estimate~\eqref{gouverneur2} stating the persistence of the exit-location of $(x^\sigma,y^\sigma)$ on $\mathbb
 B(\lambda;\varepsilon)\times \mathbb B(\lambda;\varepsilon)$ follows from \eqref{GenericKramers-bis}. Precisely, as
 $\sigma\downarrow 0$,
 $(x^\sigma_{\widehat{\tau}^\rho_{\lambda,\varepsilon}(\sigma)},y^\sigma_{\widehat{\tau}^\rho_{\lambda,\varepsilon}(\sigma)})$ concentrates
 on the points on the boundary $\partial \mathcal O_{\lambda,\varepsilon,\rho}$ where 
 $(x,y)\mapsto v(x,y)$
  is minimal. In view of \eqref{cost1}, these minimizers are located on $\partial \mathbb B(\lambda;\varepsilon)$ or  $\partial \mathbb
  B(\lambda;\rho\varepsilon)$. And so the exit-location has to persist on
  $\mathbb{B}(\lambda;\varepsilon)\times\mathbb{B}(\lambda;\varepsilon)$.
\end{proof}

\noindent
From Lemma \ref{negan}, we gradually derive  Freidlin-Wentzell' type estimates for $\tau_{\lambda,\varepsilon}(\sigma)$ through the two following
lemmas.

\begin{lem}
\label{dale}
Define, for $\rho>0$,
\begin{equation*}
\tau^\rho_{\lambda,\varepsilon}(\sigma)=\left\{
\begin{aligned}
&\inf\left\{t\ge 0\,:\,(x_t^\sigma,y_t^\sigma)\in \mathbb B(\lambda;\rho\varepsilon)\times \mathbb
B(\lambda;\varepsilon)\right\}\,\,\mbox{if}\,\| \lambda-\lambda_{1}\| =\varepsilon\,,\\
&\inf\left\{t\ge 0\,:\,(x_t^\sigma,y_t^\sigma)\in \mathbb B(\lambda;\varepsilon)\times \mathbb
B(\lambda;\rho\varepsilon)\right\}\,\,\mbox{if}\,\| \lambda-\lambda_{2}\| =\varepsilon\,,\\
&\inf\left\{t\ge 0\,:\,(x_t^\sigma,y_t^\sigma)\in \mathbb B(\lambda;\varepsilon)\times \mathbb
B(\lambda;\varepsilon)\right\}\,\,\mbox{otherwise}\,.
\end{aligned}
\right.
\end{equation*}
Then, for any $\lambda\in\bRb^d$,  $0<\varepsilon<\varepsilon_0$, $0<\rho<1$ and for any $\delta>0$:

\begin{equation}
\label{merle}
\lim_{\sigma\to0}\PP\left(\exp\left[\frac{2}{\sigma^2}\left(\widehat\h^\rho_\varepsilon(\lambda)-\delta\right)\right]<
\tau^\rho_{\lambda,\varepsilon}(\sigma)<\exp\left[\frac{2}{\sigma^2}\left(\widehat\h^\rho_\varepsilon(\lambda)+\delta\right)\right]\right)=1\,.
\end{equation}
Moreover \eqref{gouverneur2} still holds true with
$(x^\sigma_{\tau^\rho_{\lambda,\varepsilon}(\sigma)},y^\sigma_{\tau^\rho_{\lambda,\varepsilon}(\sigma)})$ in place of
$(x^\sigma_{\widehat\tau^\rho_{\lambda,\varepsilon}(\sigma)},y^\sigma_{\widehat\tau^\rho_{\lambda,\varepsilon}(\sigma)})$.
\end{lem}

We point out that $\widehat
\tau^\rho_{\lambda,\varepsilon}(\sigma)$ consists in the first entering-time in the enlargement (through the ascending flows) of the ball of center $\lambda$ and radius $\varepsilon$ (or $\rho\varepsilon$ if $\lambda$ is in a $\varepsilon$-neighbourhood of one of the two attractors) whereas $\tau^\rho_{\lambda,\varepsilon}(\sigma)$ directly deals with these balls as entering domain.

\begin{proof}
Since ${\mathbb B}(\lambda;\varepsilon)$, for the case {\bf $(a)$}, and ${\mathbb B}(\lambda;\rho\varepsilon)$, for the case {\bf $(b)$}, are both contained in each domain $\mathcal
D_{\lambda,\varepsilon}^i$, necessarily the inequality~$\tau^\rho_{\lambda,\varepsilon}(\sigma)\ge
\widehat{\tau}^\rho_{\lambda,\varepsilon}(\sigma)$ holds almost surely. As \eqref{gouverneur1} ensures that
\[
\lim_{\sigma\to0}\PP\left(\widehat{\tau}^\rho_{\lambda,\varepsilon}(\sigma)\le \exp\left[\frac{2}{\sigma^2}\left(\widehat
\h^\rho_\varepsilon(\lambda)-\delta\right)\right]\right)=0\,,
\]
the lower tail in \eqref{merle} follows. To obtain the upper-tail
\[
\lim_{\sigma\to0}\PP\left(\tau^\rho_{\lambda,\varepsilon}(\sigma)<\exp\left[\frac{2}{\sigma^2}\left(\widehat
\h^\rho_\varepsilon(\lambda)+\delta\right)\right]\right)=1\,,
\]
fix $\delta>0$, let $\xi>0$ be smaller than $\rho$ and use the inequality:
\begin{align*}
\mathbb P\left( \tau^\rho_{\lambda,\varepsilon}(\sigma)\ge
\exp\left[\frac{2}{\sigma^2}\left(\widehat\h^\rho_\varepsilon(\lambda)+\delta\right)\right]\right)
&\le \mathbb P\left(\widehat{\tau}^\xi_{\lambda,\varepsilon}(\sigma)\ge
\exp\left[\frac{2}{\sigma^2}\left(\widehat\h^\rho_\varepsilon(\lambda)+\delta\right)\right] \right)\\
&\quad+\mathbb P\left( \tau^\rho_{\lambda,\varepsilon}(\sigma)\ge
\exp\left[\frac{2}{\sigma^2}\left(\widehat\h^\rho_\varepsilon(\lambda)+\delta\right)\right],\,\widehat{\tau}^\xi_{\lambda,\varepsilon}(\sigma)<
\tau^\rho_{\lambda,\varepsilon}(\sigma)\right)\\
&\le \mathbb P\left(\widehat{\tau}^\xi_{\lambda,\varepsilon}(\sigma)\ge
\exp\left[\frac{2}{\sigma^2}\left(\widehat\h^\rho_\varepsilon(\lambda)+\delta\right)\right] \right)
+\mathbb P\left(\widehat{\tau}^\xi_{\lambda,\varepsilon}(\sigma)< \tau^\rho_{\lambda,\varepsilon}(\sigma)\right).
\end{align*}
Since $\eta\mapsto \widehat\h^\eta_\varepsilon(\lambda)$ is continuous, we can choose $\xi$ close enough to $\rho$ so
that~$\widehat\h^\rho_\varepsilon(\lambda)>\widehat\h^\xi_\varepsilon(\lambda)-\delta/2$. This way, the event
$\left\{\widehat{\tau}^\xi_{\lambda,\varepsilon}(\sigma)\ge
\exp\left[\frac{2}{\sigma^2}\left(\widehat\h^\rho_\varepsilon(\lambda)+\delta\right)\right]\right\}$ is embedded into the event
$\left\{\widehat{\tau}^\xi_{\lambda,\varepsilon}(\sigma)\ge
\exp\left[\frac{2}{\sigma^2}\left(\widehat\h^\xi_\varepsilon(\lambda)+\frac{\delta}{2}\right)\right]\right\}$ and the upper-tail
estimate for $\widehat{\tau}^\xi_{\lambda,\varepsilon}(\sigma)$ in~\eqref{gouverneur1} ensures that $\mathbb
P\left\{\widehat{\tau}^\xi_{\lambda,\varepsilon}(\sigma)\ge
\exp\left[\frac{2}{\sigma^2}\left(\widehat\h^\rho_\varepsilon(\lambda)+\delta\right)\right] \right\}$ vanishes as $\sigma$ tends to~$0$.
For the remaining component, the event $\{\widehat{\tau}^\xi_{\lambda,\varepsilon}(\sigma)< \tau^\rho_{\lambda,\varepsilon}(\sigma)\}$
implies that the vector
$\big(x^\sigma_{\widehat{\tau}^\xi_{\lambda,\varepsilon}(\sigma)},y^\sigma_{\widehat{\tau}^\xi_{\lambda,\varepsilon}(\sigma)}\big)$ does
not belong to $\overline{\mathbb B(\lambda;\varepsilon)}\times\overline{\mathbb B(\lambda;\varepsilon)}$. Recalling~\eqref{gouverneur2}
from Lemma \ref{negan}, this event becomes negligible as $\sigma\downarrow 0$ and so
$\mathbb P\left(\widehat{\tau}^\xi_{\lambda,\varepsilon}(\sigma)< \tau^\rho_{\lambda,\varepsilon}(\sigma)\right)$ vanishes as $\sigma$
tends to $0$.

  The concentration of
  $(x_{\tau^\rho_{\lambda,\varepsilon}(\sigma)},y_{\tau^\rho_{\lambda,\varepsilon}(\sigma)})$ is a straightforward consequence of the very
  definition of $\tau_{\lambda,\varepsilon}^\rho(\sigma)$.
\end{proof}

 \begin{lem}
\label{dale-bis}
Let $\tau_{\lambda,\varepsilon}(\sigma)$ be defined as in \eqref{label_bizarre1}.
For any $\lambda\in\bRb^d$ and $0<\varepsilon<\varepsilon_0$, it holds: for any $\delta>0$,
\begin{equation}\label{shane}
\lim_{\sigma\to0}\PP\left(\exp\left[\frac{2}{\sigma^2}\left(\widehat\h_\varepsilon(\lambda)-\delta\right)\right]<
\tau_{\lambda,\varepsilon}(\sigma)<\exp\left[\frac{2}{\sigma^2}\left(\widehat\h_\varepsilon(\lambda)+\delta\right)\right]\right)=1\,,
\end{equation}
for

\begin{equation}\label{gouverneur3}
\widehat\h_\varepsilon(\lambda):=\lim_{\rho\rightarrow 1}\widehat\h^\rho_{\varepsilon}(\lambda)=\left\{
\begin{array}{ll}
\inf_{x\in\mathbb R^d,y\in\partial \mathbb B(\lambda;\varepsilon)}v(x,y) &\quad\mbox{if}\,\,\|\lambda-\lambda_1\|<
\varepsilon,\\
\inf_{x\in\partial \mathbb B(\lambda;\varepsilon),y\in\mathbb R^d}v(x,y)&\quad\mbox{if}\,\,\|\lambda-\lambda_2\|<
\varepsilon,\\
\inf_{x\in\partial \mathbb B(\lambda;\varepsilon),y\in\partial \mathbb
B(\lambda;\varepsilon)}v(x,y)&\quad\mbox{if}\,\,\min_{i=1,2}\|\lambda-\lambda_i\|\ge \varepsilon.	
\end{array}
\right.
\end{equation}
Additionally,
\begin{equation}
\label{gouverneur4}
\lim_{\sigma\to0}\PP\left({\rm
dist}\left((x^\sigma_{\tau_{\lambda,\varepsilon}(\sigma)},y^\sigma_{\tau_{\lambda,\varepsilon}(\sigma)}),\mathbb{B}(\lambda;\varepsilon)\times\mathbb{B}(\lambda;\varepsilon)\right)\leq\delta\right)=1\,.
\end{equation}
\end{lem}

\begin{proof}
The estimate \eqref{gouverneur4} is straightforward by definition of $\tau_{\lambda,\varepsilon}(\sigma)$. Next observe that $\tau_{\lambda,\varepsilon}(\sigma)\le \tau^\rho_{\lambda,\varepsilon}(\sigma)$ almost surely, and so, for any $\delta'>0$,
\[
\lim_{\sigma\rightarrow 0}\mathbb P\left(\tau_{\lambda,\varepsilon}(\sigma)<\exp\left[\frac
2{\sigma^2}(h^\rho_\varepsilon(\lambda)+\delta')\right]\right)=1\,.
\]
Also as $\lim_{\rho\rightarrow 1}\widehat\h^\rho_\varepsilon(\lambda)=\widehat\h_\varepsilon(\lambda)$, taking $\delta>0$ arbitrary, and
choosing $\rho,\delta'$ small enough so that $\widehat\h^\rho_\varepsilon(\lambda)+\delta'\le \widehat\h_\varepsilon(\lambda)+\delta$
yields the upper-tail:
\[
\lim_{\sigma\rightarrow 0}\mathbb P\left(\tau_{\lambda,\varepsilon}(\sigma)<\exp\left[\frac
2{\sigma^2}(\widehat\h_\varepsilon(\lambda)+\delta)\right]\right)=1\,.
\]
For the lower-tail
\[
\lim_{\sigma\rightarrow 0}\mathbb P\left(\tau_{\lambda,\varepsilon}(\sigma)>\exp\left[\frac
2{\sigma^2}(\widehat\h_\varepsilon(\lambda)-\delta)\right]\right)=1\,,
\]
let us consider the situation $\|\lambda-\lambda_1\|=\varepsilon$ which implies that $\|\lambda-\lambda_2\|>\varepsilon$, $\big(\mathcal D^2_{\lambda,\varepsilon}\big)^c$ is stable by $-\Psi_2$ and $\widehat\h_\varepsilon(\lambda)$ reduces to
$\inf_{x\in\mathbb R^d,y\in\partial\mathbb B(\lambda;\varepsilon)}v(x,y)$.
Observing that $\tau_{\lambda,\varepsilon}(\sigma)$ is greater or equal to $\tilde \tau_{\lambda,\varepsilon}(\sigma):=\inf\{t\ge
0\,:\,y_t^\sigma\in\mathbb B(\lambda;\varepsilon)\}$, the lower-tail estimate follows from the inequality
\begin{align*}
&\mathbb P\left(\tau_{\lambda,\varepsilon}(\sigma)>\exp\left[\frac 2{\sigma^2}(\widehat\h_\varepsilon(\lambda)-\delta)\right] \right)\ge
\mathbb P\left(\tilde \tau_{\lambda,\varepsilon}(\sigma)>\exp\left[\frac 2{\sigma^2}\bigg(\inf_{x\in\mathbb R^d,y\in\partial\mathbb
B(\lambda;\varepsilon)}v(x,y)-\delta\bigg) \right]\right)
\end{align*}
and by applying Theorem \ref{thm:KramersDZ} to $\tilde \tau_{\lambda,\varepsilon}(\sigma)$.

\noindent
Replicating the same reasoning to the case $\|\lambda-\lambda_2\|=\varepsilon$, the claim follows. Finally, whenever
$\|\lambda-\lambda_1\|\neq\varepsilon$ and $\|\lambda-\lambda_2\|\neq\varepsilon$, $\tau_{\lambda,\varepsilon}(\sigma)$ simply reduces to
$\widehat{\tau}^\rho_{\lambda,\varepsilon}(\sigma)$.
\end{proof}

\subsection{Asymptotic estimates for \texorpdfstring{$c_\varepsilon(\sigma)$}{c(epsilon,sigma)}}\label{subsec:Linear-collisionB}
In light of Lemma \ref{dale-bis}, the characteristic {\it exit-cost} of $c_\varepsilon(\sigma)$ can be identified, at least heuristically, as follows:
assuming that \eqref{shane} and \eqref{gouverneur4} still hold true by minimizing over all intermediate points $\lambda$. The exit-cost
governing the asymptotic~$\sigma\downarrow 0$ would be given by $\inf_\lambda\widehat\h_\varepsilon(\lambda)$ and the exit-location
$(x^\sigma_{c_\varepsilon(\sigma)},y^\sigma_{c_\varepsilon(\sigma)})$ would concentrate on the domain $\mathbb
B(\lambda_\varepsilon;\varepsilon)\times \mathbb B(\lambda_\varepsilon;\varepsilon)$ where $\lambda_\varepsilon$ belongs to the set of minimizers of
$\widehat\h_{\varepsilon}$. This set possibly contains multiple elements, and, in view of \eqref{gouverneur3}, can be split into two
main parts: (1): the family of minimizers belonging either to $\mathbb B(\lambda_1;\varepsilon)$ or to $\mathbb B(\lambda_2;\varepsilon)$ for which
related minima of $\widehat\h_{\varepsilon}$ are given by
   \[
  m_{1,\varepsilon}:=\inf_{\lambda \in \mathbb B(\lambda_1;\varepsilon)}\,\inf_{x\in\mathbb R^d,y\in\partial \mathbb
  	B(\lambda;\varepsilon)}v(x,y),\quad m_{2,\varepsilon}:=\inf_{\lambda \in \mathbb
  	B(\lambda_2;\varepsilon)}\,\inf_{x\in\partial \mathbb B(\lambda;\varepsilon),y\in\mathbb R^d}v(x,y)\,,
  \]
 and (2): the family of minimizers belonging to $\mathbb R^d\setminus \big(\mathbb B(\lambda_1;\varepsilon)\cup\mathbb
 B(\lambda_2;\varepsilon)\big) $ for which the corresponding minima are given by
 $m_\varepsilon:=\displaystyle \inf_{\lambda\in\mathbb R^d} \big\{h_\varepsilon(\lambda)\::\:\inf_i\|\lambda-\lambda_i\|\ge \varepsilon\big\}$ for $h_\varepsilon$ being itself given
 by
 \[
 \h_\varepsilon(\lambda):=\inf_{x\in\partial \mathbb B(\lambda;\varepsilon),y\in\partial
 \mathbb B(\lambda;\varepsilon)}v(x,y)\,.
 \]
Provided that $\varepsilon$ is small enough, the latter will predominate and will drive the persistence region
of $(x^\sigma_{c_\varepsilon(\sigma)},y^\sigma_{c_\varepsilon(\sigma)})$. Indeed, as $\varepsilon\downarrow 0$, owing to the continuity of $(x,y)\mapsto v(x,y)$ (since all $\Psi_1$ and $\Psi_2$ are $\cCc^1$), for any given $\lambda\in\mathbb R^d$, as $\varepsilon$ tends to $0$,
\[
m_{1,\varepsilon}\rightarrow\inf_{\lambda \in \mathbb B(\lambda_1;0)}\,\inf_{x\in\mathbb R^d,y\in\partial \mathbb
  	B(\lambda;0)}v(x,y)=\inf_{x\in\bRb^d}v(x,\lambda_1)=v(\lambda_1,\lambda_1)\,.
\]
Indeed, by its very definition, $v(x,\lambda_1)$ is the cost for going from $(\lambda_1,\lambda_2)$ to $(x,\lambda_1)$. However, the two dynamics ($x^\sigma$ and $y^\sigma$) have separated rate functions, see the definition of $i_{T_1,T_2}(\phi)$. Consequently, $v(x,\lambda_1)$ is the cost for the first coordinate for going from $\lambda_1$ to $x$ plus the cost for the second coordinate for going from $\lambda_2$ to $\lambda_1$. It is then immediate that $\displaystyle\inf_{x\in\bRb^d}v(x,\lambda_1)=v(\lambda_1,\lambda_1)$. In the same way, we have, as $\varepsilon\downarrow0$,
\[
m_{2,\varepsilon}(\lambda)\rightarrow\inf_{\lambda \in \mathbb
  	B(\lambda_2;0)}\,\inf_{x\in\partial \mathbb B(\lambda;0),y\in\mathbb R^d}v(x,y)=\inf_{y\in\bRb^d}v(\lambda_2,y)=v(\lambda_2,\lambda_2)\,.
\]
However, for any given $\lambda\in\mathbb R^d$,  $\h_\varepsilon(\lambda)$
converges to the function
\begin{equation}\label{IntermediateCollisionCost}
\h_0(\lambda):=v(\lambda,\lambda),
\end{equation}
further yielding $\inf_{\lambda \in \mathbb R^d}\h_\varepsilon(\lambda)\rightarrow\inf_{\lambda \in \mathbb R^d} v(\lambda,\lambda)$. As such, 
	\[
	\lim_{\varepsilon \rightarrow 0}\inf_\lambda \widehat\h_{\varepsilon}(\lambda)=\lim_{\varepsilon \rightarrow 0}\inf_\lambda \h_{\varepsilon}(\lambda)=\inf_\lambda \h_{0}(\lambda)
	\] 
	and the set of minimizers of $\widehat\h_\varepsilon$ converges to $\text{argmin}\,\h_0$. 

Following this preliminary discussion, let us so introduce the threshold
\begin{equation*}
	\varepsilon_c:=\inf\left\{\varepsilon\,\in\,(0,\varepsilon_0)\,:\,
	m_{i,\varepsilon}\ge \inf_\lambda \h_\varepsilon(\lambda),\,\,i\in\{1,2\}\right\}
\end{equation*}
 which defines the largest radius for which the exit-cost $\inf_\lambda\widehat\h_\varepsilon(\lambda)$ reduces into
$\inf_{\lambda}\h_\varepsilon(\lambda)$. In particular, for all $\varepsilon\le \varepsilon_c$, $m_{i,\varepsilon}>\inf_\lambda h_\varepsilon(\lambda)$ and $\inf_\lambda h_\varepsilon(\lambda)=\inf_\lambda \widehat h_\varepsilon(\lambda)$. Remark further that whenever $\varepsilon<\varepsilon_c$, 
$$
\inf_{\lambda}
 \h_\varepsilon(\lambda)=\inf_{\lambda} \left\{\h_\varepsilon(\lambda)\::\:\inf_i\|\lambda-\lambda_i\|\ge \varepsilon\right\}\,.
 $$ 
 With this, we are ready to connect $c_\varepsilon(\sigma)$ and $h_\varepsilon$ at the zero-noise limit.
\begin{prop} 
\label{lacollision} Given $\varepsilon<\varepsilon_c$ and $\underline{\h}_{\varepsilon}=\inf_{\lambda \in \bRb^d}\h_{\varepsilon}(\lambda)$, we have: for any $\delta>0$,
\begin{equation}
\label{tdog}
\lim_{\sigma\to0}\PP\left(\exp\left[\frac{2}{\sigma^2}\left(\underline{\h}_{\varepsilon} -\delta\right)\right]<
c_\varepsilon(\sigma)<\exp\left[\frac{2}{\sigma^2}\left(\underline{\h}_{\varepsilon}+\delta\right)\right]\right)=1\,.
\end{equation}

In addition,  for $\mathcal M_\varepsilon$ the set of minimizers of $\lambda\mapsto \h_{\varepsilon}(\lambda)$, it holds:
\begin{equation}
\label{tdog2}
\lim_{\sigma\to0}\PP\left(\inf_{\lambda_\varepsilon\in\mathcal M_\varepsilon}\max\bigg({\rm
dist}\big(x^\sigma_{c_\varepsilon(\sigma)},\mathbb B(\lambda_\varepsilon;\varepsilon)\big),{\rm
dist}\big(y^\sigma_{c_\varepsilon(\sigma)},\mathbb B(\lambda_\varepsilon;\varepsilon) \big)\bigg)\ge \delta\right)=0\,.
\end{equation}
 \end{prop}

\begin{proof} 

\textbf{Outline of the proof:} We will first derive the upper-tail in \eqref{tdog}, and next introduce a suitable barycentre between $x^\sigma$ and $y^\sigma$ and its limit(~\eqref{ben} below) behaviour along $c_\varepsilon(\sigma)$. With some compactness arguments, these will enable to deduce both  the lower-tails \eqref{tdog2} and in \eqref{tdog}. The proof of \eqref{ben} concludes the claim.

\noindent\textbf{Upper-tail in \eqref{tdog}:}
Since $c_\varepsilon(\sigma)$ is the minimum of~$\tau_{\lambda,\varepsilon}(\sigma)$ (\eqref{label_bizarre1}) over all possible $\lambda$, the limit
$$
\lim_{\sigma\to0}\PP\left(c_\varepsilon(\sigma)<\exp\left[\frac{2}{\sigma^2}\left(\underline{\h}_{\varepsilon}+\delta\right)\right]\right)=1
$$
follows immediately from the upper-tail~\eqref{shane} in Lemma~\ref{dale-bis} applied to $\lambda=\lambda_\varepsilon$ a minimizer of the function~$\h_{\varepsilon}$.

\noindent
\textbf{Lower-tails \eqref{tdog} and in \eqref{tdog2}:} To obtain the lower-tail estimate and further the estimate  \eqref{tdog2}, it is sufficient to show that the barycentre $z^\sigma_t:=(x^\sigma_t+y^\sigma_t)/2$, $t\ge 0$, satisfies the exit-location property:
\begin{equation}\label{ben}
\lim_{\sigma\rightarrow 0}\mathbb P\left(z^\sigma_{c_\varepsilon(\sigma)}\in \overline{\mathcal{M}_{\varepsilon,\delta}}\right)= 1\,,
\end{equation}
for $\mathcal{M}_{\varepsilon,\delta}:=\{z\in\mathbb R^d\,:\,\mathrm{dist}(z,\mathcal{M}_{\varepsilon})<\delta\}$ with an arbitrary $\delta>0$. Indeed, since $\|x^\sigma_{c_\varepsilon(\sigma)}-y^\sigma_{c_\varepsilon(\sigma)}\|= 2\varepsilon$, we have naturally 

$$
\left\{z^\sigma_{c_\varepsilon(\sigma)}\in \overline{\mathcal{M}_{\varepsilon,\delta}}\right\}\subset\left\{ \inf_{\lambda_\varepsilon\in\mathcal M_\varepsilon}\max\bigg({\rm
	dist}\big(x^\sigma_{c_\varepsilon(\sigma)},\mathbb B(\lambda_\varepsilon;\varepsilon)\big),{\rm
	dist}\big(y^\sigma_{c_\varepsilon(\sigma)},\mathbb B(\lambda_\varepsilon;\varepsilon) \big)\bigg)\ge \delta\right\}^c\,,
$$
and so \eqref{ben} gives \eqref{tdog2}.
Additionally, since
\begin{align*}
&\mathbb P\left(c_\varepsilon(\sigma)
\le \exp\left[\frac{2}{\sigma^2}(\underline \h_\varepsilon-\delta)\right]\right)
\le \mathbb P\left(z^\sigma_{c_\varepsilon(\sigma)}\notin \overline{\mathcal{M}_{\varepsilon,\delta}}\right)
+\mathbb P\left(c_\varepsilon(\sigma)\le \exp\left[\frac{2}{\sigma^2}(\underline
\h_\varepsilon-\delta)\right],\,\,z^\sigma_{c_\varepsilon(\sigma)}\in \overline{\mathcal{M}_{\varepsilon,\delta}}\right)\,,
\end{align*}
the limit \eqref{ben} reduces the proof of the lower tail in \eqref{tdog} to establish that
\begin{align*}
\lim_{\sigma\rightarrow 0}\mathbb P\left(c_\varepsilon(\sigma)\le \exp\left[\frac{2}{\sigma^2}(\underline
\h_\varepsilon-\delta)\right],\,\,z^\sigma_{c_\varepsilon(\sigma)}\in \overline{\mathcal{M}_{\varepsilon,\delta}}\right)=0\,.
\end{align*}
For $\delta>0$, since $\overline{\mathcal{M}_{\varepsilon,\delta}}$ is compact (this property being a consequence of the compactness of
$\mathcal M_\varepsilon$  which follows from the global contractivity of $-\Psi_1$ and of $-\Psi_2$), one can construct a finite
covering~$\cup_{\ell=1}^L\mathbb B(\lambda^\ell;r)\supset \overline{\mathcal{M}_{\varepsilon,\delta}}$ - where $r$, aimed to be small, will be
chosen later on - and for which the event
\[
\left\{c_\varepsilon(\sigma)\le \exp\left[\frac{2}{\sigma^2}(\underline \h_\varepsilon-\delta)\right],\,z^\sigma_{c_\varepsilon(\sigma)}\in
\overline{\mathcal{M}_{\varepsilon,\delta}}\right\}
\]
is embedded into the union
\[
\bigcup_{\ell=1}^L\left\{c_\varepsilon(\sigma)\le \exp\left[\frac{2}{\sigma^2}(\underline
\h_\varepsilon-\delta)\right],\,z^\sigma_{c_\varepsilon(\sigma)}\in \overline{\mathbb B(\lambda^\ell;r)}\right\}.
\]
For any $\ell$,  the event $\{z^\sigma_{c_\varepsilon(\sigma)}\in\overline{\mathbb B(\lambda^\ell;r)}\}$ implies that the events $\left\{\|
x^\sigma_{c_\varepsilon(\sigma)}-\lambda^\ell\|\le \varepsilon+r\right\}$ and~$\left\{\| y^\sigma_{c_\varepsilon(\sigma)}-\lambda^\ell \|\le
\varepsilon+r\right\}$ occur simultaneously, and so does $\{c_\varepsilon(\sigma)\ge \tau_{\lambda^\ell,\varepsilon+r}(\sigma)\}$. Choosing~$r$
small enough so that $\h_{\varepsilon+r}(\lambda^\ell)-\delta'\ge \underline \h_\varepsilon-\delta$ for $\delta'>0$, Lemma~\ref{dale-bis}
yields, for any $\ell$,
\begin{align*}
&\lim_{\sigma\rightarrow 0}\mathbb P\left(\tau_{\lambda^\ell,\varepsilon+r}(\sigma)\le \exp\left[\frac{2}{\sigma^2}\bigg(\underline
\h_\varepsilon-\delta\bigg)\right]\right)
\le \lim_{\sigma\rightarrow 0}\mathbb P\left(\tau_{\lambda^\ell,\varepsilon+r}(\sigma)\le
\exp\left[\frac{2}{\sigma^2}\bigg(\h_{\varepsilon+r}(\lambda^\ell)-\delta'\bigg)\right]\right)=0\,.
\end{align*}
Therefore
\begin{align*}
\sum_{\ell=1}^L\mathbb P\left(c_\varepsilon(\sigma)\le \exp\left[\frac{2}{\sigma^2}(\underline
\h_\varepsilon-\delta)\right],\,z^\sigma_{c_\varepsilon(\sigma)}\in \overline{\mathbb B(\lambda^\ell;r)}\right)\le 
\sum_{\ell=1}^L\mathbb P\left(\tau_{\lambda^\ell,\varepsilon+r}(\sigma)\le
\exp\left[\frac{2}{\sigma^2}\bigg(\h_{\varepsilon+r}(\lambda^\ell)-\delta'\bigg)\right]\right)=0
\end{align*}
vanishes as $\sigma$ tends to $0$, yielding
\[
\lim_{\sigma\rightarrow 0}\mathbb P\left(c_\varepsilon(\sigma)\le \exp\left[\frac{2}{\sigma^2}(\underline
\h_\varepsilon-\delta)\right]\right)=0\,.
\]
\textbf{Proof of \eqref{ben}:}  Let us close the proof with \eqref{ben}. To this aim, for $\xi>0$, define the level set
\begin{equation*}
	S_{\varepsilon,\xi}=\left\{\lambda\in\mathbb R^d\,:\,v(\lambda_1,\lambda)\ge \underline\h_\varepsilon+3\xi\right\},
\end{equation*}
for 
\begin{equation*}
	\tilde v(\lambda_1,\lambda):=\frac 1{4}\inf_{0\le T<\infty}\inf_{\phi}\left\{\int_{0}^{T}\Big\|\dot{\phi}_t+\W_{\lambda_1}(\phi_t)\Big\|^2_{A^{-1}(\phi_t)}\,\ddt:\,\phi\in \mathcal C^1 ([0,T]), \,\phi(0)=\lambda_1,\,\phi(T)=\lambda \right\}.
\end{equation*}
Since $\varepsilon<\varepsilon_c$ ensures that $\lambda_1\notin \mathcal M_\varepsilon$, $S_{\varepsilon,\xi}$ is necessarily non-empty for $\xi$ small enough (for $\lambda_1\in S_{\varepsilon,\xi}$). Define
next, $M_{\varepsilon,\xi}$ a minimal radius, strictly larger than $\varepsilon$ and such that $\mathbb R^d\setminus\mathbb
B(\lambda_1;M_{\varepsilon,\xi}-\varepsilon)$ lies in $S_{\varepsilon,\xi}$. Equivalently, for any $\lambda\in\mathbb R^d$ such that
$\|\lambda-\lambda_1\|>M_{\varepsilon,\xi}-\varepsilon$, $v(\lambda_1,\lambda)$ is larger
than~$\underline\h_\varepsilon+3\xi$. From this, one can derive the bound:
\begin{align*}
&\mathbb P\left(z^\sigma_{c_\varepsilon(\sigma)}\notin \overline{\mathcal{M}_{\varepsilon,\delta}}\right)
\le \mathbb P\left(z^\sigma_{c_\varepsilon(\sigma)}\notin \overline{\mathbb B(\lambda_1;M_{\varepsilon,\xi})}\right)
+\mathbb P\left(z^\sigma_{c_\varepsilon(\sigma)}\in \overline{\mathbb B(\lambda_1;M_{\varepsilon,\xi})}\setminus
\overline{\mathcal{M}_{\varepsilon,\delta}}\right)=:I_1(\sigma)+I_2(\sigma)\,,
\end{align*}
and next check that $I_i(\sigma)$ for $i=1,2$ tends to $0$.

For the limit $\lim_{\sigma\rightarrow 0}I_1(\sigma)=0$: using again the fact that
$\|x^\sigma_{c_\varepsilon(\sigma)}-y^\sigma_{c_\varepsilon(\sigma)}\|=2\varepsilon$, the distance between~$x^\sigma_{c_\varepsilon(\sigma)}$ and
$\lambda_1$ is larger than $\| z^\sigma_{c_\varepsilon(\sigma)}-\lambda_1\|-\varepsilon$. Introducing
the first time $x^\sigma$ exits the ball~$\mathbb B(\lambda_1;M_{\varepsilon,\xi}-\varepsilon)$:
\[
\tilde \tau_{M_{\varepsilon,\xi}-\varepsilon}(\sigma):=\inf\left\{t\ge 0\,:\,\| x^\sigma_t-\lambda_1\| \ge
M_{\varepsilon,\xi}-\varepsilon\right\}\,,
\]
we have
\begin{align*}
\mathbb P\left(z^{\sigma}_{c_\varepsilon(\sigma)}\notin \overline{\mathbb{B}(\lambda_1;M_{\varepsilon,\xi})}\right)\le \mathbb
P\left(x^{\sigma}_{c_\varepsilon(\sigma)}\notin \overline{\mathbb{B}(\lambda_1;M_{\varepsilon,\xi}-\varepsilon)}\right)\le \mathbb
P\left(\tilde \tau_{M_{\varepsilon,\xi}-\varepsilon}\le c_\varepsilon(\sigma)\right)\,.
\end{align*}
Since $\mathbb B(\lambda_1;M_{\varepsilon,\xi}-\varepsilon)$ is stable by $-\Psi_1$,
Theorem \ref{thm:KramersDZ} applies for $\tilde{\tau}_{M_{\varepsilon,\xi}-\varepsilon}$ with the exit-cost
\begin{equation*}
\tilde h_\varepsilon:=\inf_{x\in \partial
\mathbb{B}(\lambda_1;M_{\varepsilon,\xi}-\varepsilon)}v(\lambda_1,\lambda)\geq\underline\h_\varepsilon+3\xi\,.
\end{equation*}
In addition, since, for any $\delta'>0$,
\begin{align*}
\mathbb P\left(\tilde \tau_{M_{\varepsilon,\xi}-\varepsilon}(\sigma)\le c_\varepsilon(\sigma)\right)
&\le \mathbb P\left(\tilde{\tau}_{M_{\varepsilon,\xi}-\varepsilon}(\sigma)\le \exp\left[\frac{2}{\sigma^2}(\underline
\h_\varepsilon+\delta')\right]\right)+\mathbb P\left(c_\varepsilon(\sigma)> \exp\left[\frac{2}{\sigma^2}(\underline \h_\varepsilon+\delta')\right]\right),
\end{align*}
choosing $\delta'<3\xi$ so that $\underline \h_\varepsilon+\delta'\le \tilde \h_\varepsilon-\delta''$ for $\delta''>0$ yields
\begin{align*}
\lim_{\sigma\rightarrow 0}\mathbb P\left(\tilde \tau_{M_{\varepsilon,\xi}-\varepsilon}\le \exp\left[\frac{2}{\sigma^2}(\underline
\h_\varepsilon+\delta')\right]\right)\le \lim_{\sigma\rightarrow 0}\mathbb P\left(\tilde \tau_{M_{\varepsilon,\xi}-\varepsilon}\le
\exp\left[\frac{2}{\sigma^2}(\tilde \h_\varepsilon-\delta'')\right]\right)=0\,.
\end{align*}
Owing to the upper-tail of $c_\varepsilon(\sigma)$ established at the beginning of the proof, we deduce that
\begin{align*}
&\lim_{\sigma\rightarrow 0}\mathbb P\left(z^{\sigma}_{c_\varepsilon(\sigma)}\notin
\overline{\mathbb{B}(\lambda_1;M_{\varepsilon,\xi})}\right)\\
&\le \lim_{\sigma\rightarrow 0}\mathbb P\left(\tilde \tau_{M_{\varepsilon,\xi}-\varepsilon}\le \exp\left[\frac{2}{\sigma^2}(\underline
\h_\varepsilon+\delta')\right]\right)+\lim_{\sigma\rightarrow 0}\mathbb P\left(c_\varepsilon(\sigma)>
\exp\left[\frac{2}{\sigma^2}(\underline \h_\varepsilon+\delta')\right]\right)=0\,.
\end{align*}

For the limit $\lim_{\sigma\rightarrow 0}I_2(\sigma)=0$: introducing a new covering $\cup_{\ell=1}^{\hat L} \mathbb B(\hat \lambda^\ell;\hat r)\supset
\overline{\mathbb B(\lambda_1;M_{\varepsilon,\xi})}\setminus \overline{\mathcal{M}_{\varepsilon,\delta}}$~-~where $\hat r$ will be again chosen later on - we derive the upper-bound
\begin{align*}
\mathbb P\bigg(z^\sigma_{c_\varepsilon(\sigma)}\in \overline{\mathbb B(\lambda_1;M_{\varepsilon,\xi})}\setminus
\overline{\mathcal{M}_{\varepsilon,\delta}}\bigg)\le \sum_{\ell=1}^{\hat L}\mathbb P\bigg(z^\sigma_{c_\varepsilon(\sigma)}\in \mathbb B(\hat
\lambda^\ell;\hat r)\bigg)\,.
\end{align*}
For any $\ell$, we also have
\begin{align*}
\mathbb P\bigg(z^\sigma_{c_\varepsilon(\sigma)}\in \mathbb B(\hat \lambda^\ell;\hat r)\bigg)
&\le \mathbb P\bigg(x^\sigma_{c_\varepsilon(\sigma)}\in \mathbb B(\hat \lambda^\ell;\hat r+\varepsilon),y^\sigma_{c_\varepsilon(\sigma)}\in
\mathbb B(\hat \lambda^\ell;\hat r+\varepsilon)\bigg)\le \mathbb P\bigg(\tau_{\hat \lambda^\ell,\varepsilon+\hat r}(\sigma)\le c_\varepsilon(\sigma)\bigg)\,.
\end{align*}
According to Lemma \ref{dale-bis}, for any $\delta>0$,
\begin{align*}
\lim_{\sigma\rightarrow 0}\mathbb P\bigg(\tau_{\hat \lambda^\ell,\varepsilon+\hat r}(\sigma)\le
\exp\left[\frac{2}{\sigma^2}(\h_{\varepsilon+\hat r}(\hat \lambda^\ell)-\delta)\right]\bigg)=0\,.
\end{align*}
Choosing $\hat r$ and $\delta$ so that, for some $\hat\delta>0$, $\h_{\varepsilon+\hat r}(\hat \lambda^l)-\hat \delta\ge  \underline
\h_\varepsilon+\delta$ then gives
\begin{align*}
&\lim_{\sigma\rightarrow 0}\mathbb P\bigg(\tau_{\hat \lambda^\ell,\varepsilon+\hat r}(\sigma)\le \exp\left[\frac{2}{\sigma^2}(\underline
\h_\varepsilon+\delta)\right]\bigg)\le
\lim_{\sigma\rightarrow 0} \mathbb P\bigg(\tau_{\hat \lambda^\ell,\varepsilon+\hat r}(\sigma)\le
\exp\left[\frac{2}{\sigma^2}(\h_{\varepsilon+\hat r}(\hat \lambda^\ell)-\hat \delta)
\right]\bigg)=0\,.
\end{align*}
Therefore, using the inequality
\begin{align*}
\mathbb P\left(\tau_{\hat \lambda^\ell,\varepsilon+\hat r}(\sigma)\le c_\varepsilon(\sigma)\right)&\le \mathbb P\left(\tau_{\hat
\lambda^\ell,\varepsilon+\hat r}(\sigma)\le \exp\left[\frac{2}{\sigma^2}(\underline \h_\varepsilon+\delta)\right]\right)+\mathbb P\left(c_\varepsilon(\sigma)\ge \exp\left[\frac{2}{\sigma^2}(\underline \h_\varepsilon+\delta)\right]\right)\,,
\end{align*}
and again Lemma \ref{dale-bis} and the upper-tail of $c_\varepsilon(\sigma)$ yields to
\[
\lim_{\sigma\rightarrow 0}\sum_{\ell=1}^{\hat{L}}\mathbb P\bigg(z^\sigma_{c_\varepsilon(\sigma)}\in \mathbb B(\hat \lambda^\ell;\hat
r)\bigg)=0\,.
\]
This immediately implies that $\lim_{\sigma\rightarrow 0}I_2(\sigma)=0$ and ends the proof of \eqref{ben}.
\end{proof}
Following Proposition \ref{lacollision}, we can eventually conclude on the limit $\varepsilon\downarrow 0$ associated to $c_\varepsilon(\sigma)$ with the following.
\begin{thm}
\label{lacollisionthm} For $\mathcal M_0$ the set of minimizers of $\h_0$ and $\underline{\h}_0:=\h_0(\lambda_0)$ for some $\lambda_0\in \mathcal M_0$, we have: for any $\delta>0$,  

\begin{equation*}
\lim_{\varepsilon\to0}
\lim_{\sigma\to0}\PP\left(\exp\left[\frac{2}{\sigma^2}\left(\underline{\h}_0
-\delta\right)\right]<c_\varepsilon(\sigma)<
\exp\left[\frac{2}{\sigma^2}\left(\underline{\h}_0+\delta\right)\right]\right)=1\,
\end{equation*}
and
\begin{equation*}
\lim_{\varepsilon\to0}\lim_{\sigma\to0}\PP\left(\max\bigg({\rm dist}( x^\sigma_{c_\varepsilon(\sigma)},\m_0),{\rm dist}(
y^\sigma_{c_\varepsilon(\sigma)},\m_0)\le \delta\right)=1\,.
\end{equation*}

\end{thm}

Theorem \ref{lacollisionthm} provides the natural analogue to Theorem \ref{thm:main1} for the reduced systems \eqref{Linear-Systems} and is a straightforward consequence of Proposition
\ref{lacollision}, the continuity of $\varepsilon\rightarrow\underline \h_\varepsilon=\inf_\lambda h_\varepsilon(\lambda)$ (since all the coefficients are of class $\cCc^1$) and of the limit $\lim_{\varepsilon\rightarrow 0}\underline h_\varepsilon=\underline h_0$.
 \begin{rem} Let us briefly comment on the set of minimizers $\mathcal M_\varepsilon$ and highlight the possible collision-locations
 $\lambda_\varepsilon\in\mathcal M_\varepsilon$ in Proposition \ref{lacollision} assuming $\Psi_i=\nabla \psi_i$, $i=1,2$, where the $\psi_i$'s are $\mathcal C^2$-uniformly convex potentials. In this case,
\begin{align*}
\inf_\lambda h_\varepsilon(\lambda)&= \inf_{\lambda}\inf_{(x,y)\in\partial \mathbb{B}(0;\varepsilon)\times \partial \mathbb{B}(0;\varepsilon)}
\Big(\psi_1(\lambda+x)-\psi_1(\lambda_1)+\psi_2(\lambda+y)-\psi_2(\lambda_2)\Big)\\
&=\inf_{(x,y)\in\partial \mathbb{B}(0;\varepsilon)\times \partial \mathbb{B}(0;\varepsilon)}\inf_{\lambda}
\Big(\psi_1(\lambda+x)-\psi_1(\lambda_1)+\psi_2(\lambda+y)-\psi_2(\lambda_2)\Big)\,,
\end{align*}
the first equality following from shifting the minimization over $(x,y)\in\partial \mathbb{B}(\lambda;\varepsilon)\times \partial
\mathbb{B}(\lambda;\varepsilon)$ to the set of points $(x+\lambda,y+\lambda)$ for $(x,y)\in\partial \mathbb{B}(0;\varepsilon)\times
\partial \mathbb{B}(0;\varepsilon)$ and the second equality from the  \emph{min-min} principle.
The potentials $\psi_1$ and $\psi_2$ being uniformly convex, the minimizers of
$\lambda\mapsto\psi_1(\lambda+x)-\psi_1(\lambda_1)+\psi_2(\lambda+y)-\psi_2(\lambda_2)$ are explicitly given by
\[\lambda(x,y)=\Big(\nabla \psi_1(\cdot+x)+\nabla \psi_2(\cdot+y)\Big)^{-1}(0)\,,
\]
independently of $(x,y)$. As such, $\inf_\lambda h_\varepsilon(\lambda)$ rewrites as
\[
\inf_{(x,y)\in\partial \mathbb{B}(0;\varepsilon)\times \partial
\mathbb{B}(0;\varepsilon)}\Bigg(\psi_1(\lambda(x,y)+x)-\psi_1(\lambda_1)+\psi_2(\lambda(x,y)+y)-\psi_2(\lambda_2)
\Bigg)\,,
\]
and, subsequently, any $\lambda_\varepsilon$ of $\mathcal M_{\varepsilon}$ corresponds to a point
$\lambda_\varepsilon(x^*_\varepsilon,y^*_\varepsilon)$ where $(x^*_\varepsilon,y^*_\varepsilon)$ achieves the above minimum. Compared to the
limit collision-location, $\lambda_0=(\nabla \psi_1+\nabla \psi_2)^{-1}(0)$, the minimizers $\lambda_\varepsilon$ are
perturbations of $\lambda_0$ in a direction of magnitude $\varepsilon$. As $\varepsilon$ decreases to~$0$, the regularity of $\Psi_1$
and $\Psi_2$ guarantees that $\mathcal M_\varepsilon$  concentrates on the single point~$\{\lambda_0\}$ and $\inf_\lambda
h_\varepsilon$ converges naturally to $\inf_\lambda h_0$. Illustratively, consider the case
where $\psi_1$ and $\psi_2$ are quadratic potentials of the form $\psi_i(z)=\gamma_i\|z-\lambda_i\|^2/2$ for~$\gamma_i>0$. The first
collision-location $\lambda_0$ is then explicitly given by $\frac{\gamma_1\lambda_1+\gamma_2\lambda_2}{\gamma_1+\gamma_2}$ and the related exit-cost by $\underline{\h}_0=\frac{\gamma_1\gamma_2}{2(\gamma_1+\gamma_2)}\|\lambda_2-\lambda_1\|^2$.
Meanwhile, for any $(x,y)$, $\lambda(x,y)$ is given by~$\frac{\gamma_1(\lambda_1-x)+\gamma_2(\lambda_2-y)}{\gamma_1+\gamma_2}$.
\end{rem}

\section{On the first near-collision between two self-stabilizing processes}\label{sec:SelfStabilizingCase}
This section is dedicated to the proof of Theorem
\ref{thm:main1}. The proof essentially consists in obtaining, over a suitably large time-interval,  a coupling estimate 
between \eqref{MV1}-\eqref{MV2}, and their 
``linearized'' versions set at their respective stable points:
\begin{subequations}
	\begin{equation}
		\label{Linearized-MV1}
		x_{T,t}^\sigma=X_{T} -\int_{T}^t\Big(\mathbb{U}(x_{T,s}^\sigma)+\F(x_{T,s}^\sigma-\lambda_1)\Big)\dds+\sigma\int_T^t\Gamma\left(x_{T,s}^\sigma\right)\dd B_s\,,\qquad t\ge T\,,
	\end{equation}
	and
	\begin{equation}
		\label{Linearized-MV2}
		y_{T,t}^\sigma=Y_{T}-\int_{T}^t\Big(\mathbb{U}(y_{T,s}^\sigma)+\F(y_{T,s}^\sigma-\lambda_2)\Big)\dds+\sigma\int_T^t\Gamma\left(y_{T,s}^\sigma\right)\dd\widetilde{B}_s\,,\qquad t\ge T\,.
	\end{equation}
\end{subequations}
The suitable starting time $T$ of the couplings will depend on a parameter $\kappa$ measuring the average closeness between \eqref{MV1}-\eqref{MV2} and their associated attractors: this parameter is specified in Lemma~\ref{chante} below and coupling estimates are stated in Proposition \ref{label_bizarre2}. Under Assumption {\bf (A)}, Proposition \ref{lacollision}, established in the previous section, applies to \eqref{Linearized-MV1} and \eqref{Linearized-MV2}, and the asymptotic of the near collision-time $C_\varepsilon(\sigma)$ and of the collision-location $L_\varepsilon(\sigma)$ can be then deduced, first at an arbitrary small radius $\varepsilon$ (Proposition \ref{label_bizarre8}) and next to the limit $\varepsilon\downarrow 0$, the latter achieving the claim in Theorem \ref{thm:main1}.

\subsection{Coupling estimates for the self-stabilizing diffusions}\label{ssec:CouplingSS}

\begin{lem}
\label{chante} 
For any $\kappa>0$, there exist a non-random time $T_\kappa\in [0,\infty)$ and $\sigma_\kappa>0$ such that, for any $\sigma<\sigma_\kappa$,
\begin{equation}
\label{MeanCoupling-MV}
\max\Big\{\mathbb E\left[\|X_{T_\kappa}-\lambda_1\|^2\right];\mathbb E\left[\|Y_{T_\kappa}-\lambda_2\|^2\right]
\Big\}\leq\frac{\kappa^2}{4}\,. 
\end{equation}
\end{lem}

\begin{proof}
Define  $\xi_1(t)=\EE\left[\|X_t-\varphi_t^{(1)}\|^2\right]$ for $\varphi^{(1)}=\big{\{}\varphi_t^{(1)}=x_1-\int_0^t\U\left(\varphi_s^{(1)}\right)\dds\big{\}}_{t\ge 0}$ and set $\F\star\mu_t^X(x):=\int \F(x-y)\mu_t^X(\ddy)$. By It\^o's formula, we get: for any $t\geq0$,
\[
\xi_1(t)=-2\mathbb E\Big[\int_0^t\left(X_s-\varphi_s^{(1)}\right)\cdot\left(\W_{\mu^X_s}(X_s)-\U(\varphi^{(1)}_s)\right)\,\dds\Big]+\sigma^2\mathbb E\Big[\int_0^t\mathrm{Trace}\big(\Gamma\Gamma^*\big)(X_s)\,\dds\Big]\,,
\]
or equivalently
\[
\frac{\dd\xi_1}{\ddt}(t)=-2\mathbb E\Big[\left(X_t-\varphi_t^{(1)}\right)\cdot\left(\W_{\mu^X_t}(X_t)-\U(\varphi^{(1)}_t)\right)\Big]+\sigma^2\mathbb E\Big[\mathrm{Trace}\big(\Gamma\Gamma^*\big)(X_t)\Big]\,.
\] 
Adding and subtracting $\F\star \mu^X_t(\varphi^{(1)}_t)$, by Assumptions $\textbf{(A)}$-$(iv)$ (synchronization) and $\textbf{(A)}$-$(vi)$ (boundedness and Lipschitz continuity of $\Gamma$), it follows that
\begin{align*}
	 \frac{\dd\xi_1}{\ddt}(t)
	 &= -2\EE\left[\left(X_t-\varphi_t^{(1)}\right)\cdot\left(\W_{\mu^X_t}(X_t)-\W_{\mu^X_t}(\varphi^{(1)}_t)\right)\right]-2\EE\left[\left(X_t-\varphi_t^{(1)}\right)\cdot\F\star\mu_t^X(\varphi^{(1)}_t)\right]\\
	 &\qquad+\sigma^2\EE\left[\mathrm{Trace}\big(\Gamma(X_t)\Gamma^*(X_t)\big)\right]\\
	 &\le -\theta\xi^1(t)-2\EE\left[\left(X_t-\varphi_t^{(1)}\right)\cdot\F\star\mu_t^X(\varphi^{(1)}_t)\right] +2d(\sigma\Lambda_+)^2\,. 
\end{align*}
Due to $(\mathbf{A})$-$(iii)$, we can write, for $\tilde X$ being an independent copy of $X$
also defined on $(\Omega,\mathcal F,\{\mathcal F_t\}_{t\ge 0},\mathbb P)$,
\begin{align*}
&-\EE\left[\left(X_t-\varphi_t^{(1)}\right)\cdot\F\star\mu_t^X(\varphi_t^{(1)})\right]\\
&=-\EE\left[\left(X_t-\varphi_t^{(1)}\right)\cdot\F\left(\tilde X_t-\varphi_t^{(1)}\right)\right]=\EE\left[\left(X_t-\varphi_t^{(1)}\right)\cdot\frac{\left(\tilde X_t-\varphi_t^{(1)}\right)}{\|\tilde X_t-\varphi_t^{(1)}\|}G\left(\big{\|}\tilde X_t-\varphi_t^{(1)}\big{\|}\right)\right]\\
&=\EE\left[\big{\|}X_t-\varphi_t^{(1)}\big{\|}\big{\|}\tilde{X}_t-\varphi_t^{(1)}\big{\|}\widetilde{G}\left(\big{\|}\tilde X_t-\varphi_t^{(1)}\big{\|}\right)\right]\,,
\end{align*}
where $\widetilde{G}(x):=\|x\|^{-1}G(\|x\|)$ defines a polynomial function of order $2n-1$. Applying Cauchy-Schwarz inequality, and using the independence of $X$ and $\tilde X$ and the moments control \eqref{Moment_SS}, it follows that
\begin{align*}
\EE\left[\left(X_t-\varphi_t^{(1)}\right)\cdot\F\star\mu_t^X(\varphi_t^{(1)})\right]&\leq
c\EE\left[\big{\|}X_t-\varphi_t^{(1)}\big{\|}\big{\|}\tilde X_t-\varphi_t^{(1)}\big{\|}\left(1+\big{\|}\tilde X_t-\varphi_t^{(1)}\big{\|}^{2n}\right)\right]\\
&\leq c\sqrt{\EE\left[\big{\|}X_t-\varphi_t^{(1)}\big{\|}^2\right]\EE\left[\big{\|}\tilde X_t-\varphi_t^{(1)}\big{\|}^2\right]}\sqrt{\EE\left[\left(1+\big{\|}\tilde X_t-\varphi_t^{(1)}\big{\|}^{2n}\right)^2\right]}\\
&\leq C(4n-1)!\,\EE\left[\big{\|}X_t-\varphi_t^{(1)}\big{\|}^2\right]=C(n)\xi_1(t)\,,
\end{align*}
for $c$ a positive constant only depending on the coefficient of $\tilde G$ and $C:=c\sqrt{2}\big(1+\max_{t\ge 0}\mathbb E[\|X_t\|^{4n}]\big)$. Thus,
\[
\frac{\dd\xi_1 }{\ddt}(t)\leq-\theta\xi_1(t) +2 C(n)\xi_1(t)+2d(\sigma\Lambda_+)^2,\qquad t\ge 0\,,
\]
and since $\xi_1(0)=0$, Gronwall's inequality then implies:
\[
\EE\left[\|X_t-\varphi_t^{(1)}\|^2\right]=\xi_1(t)\leq 4dC(n)(\sigma\Lambda_+)^2 \exp\{(2C(n)-\theta) t\},\qquad t\ge 0\,.
\]
Thus we have
\[
\EE\left[\|X_t-\lambda_1\|^2\right]\le 2\big{\|}\varphi_t^{(1)}-\lambda_1\big{\|}^2+\sigma^2C \exp\{C t\},\qquad t\ge 0\,.
\]
Replicating the preceding arguments to $(Y_t)_{t\ge 0}$ and $\varphi^{(2)}=\big{\{}\varphi_t^{(2)}=x_2-\int_0^t\U\left(\varphi_s^{(2)}\right)\dds\big{\}}_{t\ge 0}$, we also obtain 
\[
\EE\left[\big{\|}Y_t-\lambda_2\big{\|}^2\right]\leq2\big{\|}\varphi_t^{(2)}-\lambda_2\big{\|}^2+\sigma^2C
\exp\{C t\},\qquad t\ge 0\,.
\]
According to Assumption $(\mathbf{A})$-$(iv)$, as $x_1\in\mathcal{G}(\lambda_1)$ and $x_2\in\mathcal{G}(\lambda_2)$, we have
$\lim_{t\to\infty}\varphi_t^{(1)}=\lambda_1$ and $\lim_{t\to\infty}\varphi_t^{(2)}=\lambda_2$. Setting $T_\kappa$ large enough such that
\[
2\big{\|}\varphi^{(1)}_{T_\kappa} -\lambda_1\big{\|}^2+2\big{\|}\varphi^{(2)}_{T_\kappa}-\lambda_2\big{\|}^2\leq\frac{\kappa^2}{8}\,,
\]
one can next take $\sigma$ small enough such that $dC(\Lambda_+\sigma)^2
\exp\{C T_\kappa\}\leq\frac{\kappa^2}{8}$, that is $\sigma<\sigma_\kappa=\kappa\big(8Cd\exp\{CT_\kappa\}\big)^{-1/2}$. Then, for any $\sigma<\sigma_\kappa$,
\[
\big(\mathbb E\left[\|X_{T_\kappa}-\lambda_1\|^2\right]\big)^{1/2}\leq\frac{\kappa}{4}\quad\mbox{and}\quad \big(\mathbb E\left[\|Y_{T_\kappa}-\lambda_2\|^2\right]\big)^{1/2}\leq\frac{\kappa}{4}\,.
\]
This achieves the proof.

\end{proof}

Next proposition ensures that the self-stabilizing processes and their associated linearized versions stay close to each other during an exponentially long time interval. The result also significantly expands the coupling techniques of the second author, e.g. \cite{Alea,Tugaut2021}, to non-reversible self-stabilizing diffusions with non-constant dispersion matrix $\Gamma$ and with $\F$ non-necessarily quadratic.

\begin{prop}
\label{label_bizarre2}
For any $\xi,H>0$, there exists $T_\xi\in(0,\infty)$, depending only on $\U,\F, \Gamma$ and $\xi$, such that
\begin{equation*}
\lim_{\sigma\to0}\PP\left(\sup_{T_\xi\leq t\leq
\exp\left[\frac{2H}{\sigma^2}\right]}\|X_t-x_{T_\xi,t}^\sigma\|\geq\xi\right)=0=\lim_{\sigma\to0}\PP\left(\sup_{T_\xi\leq t\leq
\exp\left[\frac{2H}{\sigma^2}\right]}\|Y_t-y_{T_\xi,t}^\sigma\|\geq\xi\right)\,.
\end{equation*}
\end{prop}

\begin{proof}[Proof of Proposition \ref{label_bizarre2}] We only focus the demonstration on the pair $\left(X_t\right)_{t\geq0}$ and
	$\left(x_{T_\kappa,t}^\sigma\right)_{t\geq0}$, the coupling between $\left(Y_t\right)_{t\geq0}$ and
	$\left(y_{T_\kappa,t}^\sigma\right)_{t\geq0}$ being handled in the same way. 
	
	For convenience, the tail estimate on $\|X_t-x_{T_\kappa,t}^\sigma\|$ is split into four steps: the two first steps {\bf Steps 1} and {\bf 2} simply handle the diffusive and martingale parts of $\|X_t-x_{T_\kappa,t}^\sigma\|^2$ (using Proposition \ref{prop:WellposedMV} and Lemma \ref{lem:MartingaleConcentration}). {\bf Step 3} addresses the control of the force fields $\F$ and $\U$ and establish a first version of the claim up to a certain random time ($S_\kappa$ below). In {\bf Step 4}, this bound is proved to eventually increase faster than $\exp[2H/\sigma^2]$ yields to the end of the proof.

\noindent{}{\bf Step 1. }Fix $\xi>0$ and, given~$\kappa>0$~-~to be chosen later - and $\sigma>0$, let $T_\kappa$ be given by
\eqref{MeanCoupling-MV}. Almost surely, for any $t\ge T_\kappa$,
\begin{align*}
\|X_t-x_{T_\kappa,t}^\sigma\|^2=&-2\int_{T_\kappa}^t\left(X_s-x_{T_\kappa,s}^\sigma\right)\cdot\bigg(\W_{\mu_s^{X}}\left(X_s\right)-\W_{\lambda_1}\left(x_{T_\kappa,s}^\sigma\right)\bigg)\dds\\
&+\sigma^2\int_{T_\kappa}^t\mathrm{Trace}\Big(\left(\Gamma(X_s)-
\Gamma(x_{T_\kappa,s}^\sigma)\right)\left(\Gamma^*(X_s)-
\Gamma^*(x_{T_\kappa,s}^\sigma)\right)\Big)\dds\\
&+2\sigma\int_{T_\kappa}^t\left(X_s-x_{T_\kappa,s}^\sigma\right)\cdot\bigg(\Gamma\left(X_s\right)-\Gamma\left(x_{T_\kappa,s}^\sigma\right)\bigg)\dd B_s\,,
\end{align*}
recalling that $\mu_t^{X}=\mathcal{L}\left(X_t\right)$, $\W_\mu(x)=\U(x)+\F\star\mu(x)$ and
$\W_{\lambda_1}(x)=\U(x)+\F(x-\lambda_1)$. 
Since $\Gamma$ is Lipschitz continuous:
\begin{align*}
	\|X_t-x_{T_\kappa,t}^\sigma\|^2\le &-2\int_{T_\kappa}^t\left(X_s-x_{T_\kappa,s}^\sigma\right)\cdot\bigg(\W_{\mu_s^{X}}\left(X_s\right)-\W_{\lambda_1}\left(x_{T_\kappa,s}^\sigma\right)\bigg)\dds\\
	&+2\sigma\int_{T_\kappa}^t\left(X_s-x_{T_\kappa,s}^\sigma\right)\cdot\bigg(\Gamma\left(X_s\right)-\Gamma\left(x_{T_\kappa,s}^\sigma\right)\bigg)\dd B_s+C\sigma^2\int_{T_\kappa}^t\|X_s-x_{T_\kappa,s}^\sigma\|^2\,\dds\,.
\end{align*}
\noindent{}{\bf Step 2. }Let us first handle the control of the martingale part: 
\[
M_{T_\kappa,t}=\int_{T_\kappa}^t\Theta_s\,\dd B_s:=\int_{T_\kappa}^t\left(X_s-x_{T_\kappa,s}^\sigma\right)\cdot\bigg(\Gamma\left(X_s\right)-\Gamma\left(x_{T_\kappa,s}^\sigma\right)\bigg)\dd B_s,\ \ t\ge T_\kappa\,.
\]
For $\langle M\rangle$ the quadratic variation of $M$:
\[
\langle M\rangle_{T_\kappa,t}=\int_{T_\kappa}^t\|\Theta_s\|^2\,\dds=\int_{T_\kappa}^t\left(X_s-x_{T_\kappa,s}^\sigma\right)\cdot\left(\big(\Gamma\left(X_s\right)-\Gamma(x_{T_\kappa,s}^\sigma)\big)\big(\Gamma(X_s)-\Gamma(x_{T_\kappa,s}^\sigma\big)^*\right)\left(X_s-x_{T_\kappa,s}^\sigma\right)\,\dds\,,
\]
for $t\ge T_\kappa$, define the set
\[
\mathcal{A}^\sigma:=\left\{\omega\in\Omega\,\,:\,\,\forall t\in\left[T_\kappa,\exp\Big[\frac{2H}{\sigma^2}\Big]\right],\, M_t<\frac{1}{2}\left\langle
M\right\rangle_t+\frac{1}{\sqrt{\sigma}}\right\}\,.
\] 
In light of the boundedness of $\Gamma$, the uniform-in-time moment control \eqref{Moment_SS} in Proposition \ref{prop:WellposedMV} and the analog for $(x_{T_\kappa,t})_{t\ge 0}$:\footnote{This estimate can be derived following the proof of \eqref{Moment_SS}, or the proof of Lemma \ref{lem:Moment_particle}, in Appendix \ref{Appendix-A} taking the interaction-free drift $\W_{\lambda_1}$ in each computation.}
$$
\sup_{t\ge 0}\mathbb E[\|x_{T_\kappa,t}\|^{2p}]\le c^p(p-1)!\,,
$$
we have, for any $p\ge 1$,
$$
\sup_{t\ge0}\mathbb E\Big[|\left(X_s-x_{T_\kappa,s}^\sigma\right)\cdot\Big(\Gamma\left(X_s\right)-\Gamma\left(x_{T_\kappa,s}^\sigma\right)\Big)|^{2p}\Big]\le C^p(p-1)!\,. 
$$
Hence Lemma \ref{lem:MartingaleConcentration} applies and yields 
\begin{align*}
	\PP\left(\left(\mathcal{A}^\sigma\right)^c\right)&=\PP\left(\exists t\geq T_\kappa\,\,:\,\, M_t\geq\frac{1}{2}\left\langle
	M\right\rangle_t+\frac{1}{\sqrt{\sigma}}\right)\le\exp\Big[-\frac{1}{\sqrt{\sigma}}\Big]\,.
\end{align*}
For any $\omega\in\mathcal{A}^\sigma$ and for any $t\geq T_\kappa$, 
  we have the inequality $\sigma
M_t\leq\frac{\sigma}{2}\left\langle M\right\rangle_t+\sqrt{\sigma}$. Due to the boundedness of $\Gamma$,
\[
\left\langle M\right\rangle_t\leq (\Lambda_+)^2d\int_{T_\kappa}^t\big{\|}X_s-x_{T_\kappa,s}^\sigma\big{\|}^2\dds\,.
\]
 Hence, on $\mathcal{A}^\sigma$, we have that
\begin{align}\label{Coupling_SS_Proof_1}
	\|X_t-x_{T_\kappa,t}^\sigma\|^2\leq&\sqrt{\sigma}+C\big(\sigma^2+\frac{\sigma}{2}\big)\int_{T_\kappa}^t\big{\|}X_s-x_{T_\kappa,s}^\sigma\big{\|}^2\dds\nonumber\\
	&-2\int_{T_\kappa}^t\left(X_s-x_{T_\kappa,s}^\sigma\right)\cdot\bigg(\W_{\mu_s^{X}}\left(X_s\right)-\W_{\lambda_1}\left(x_{T_\kappa,s}^\sigma\right)\bigg)\dds.
\end{align}

\noindent{}{\bf Step 3. }We now deal with the following term:
$-\left(X_t-x_{T_\kappa,t}^\sigma\right)\cdot\Big(\W_{\mu_t^{X}}\left(X_t\right)-\W_{\lambda_1}\left(x_{T_\kappa,t}^\sigma\right)\Big)$. Adding and subtracting $\W_{\mu_t^{X}}(x_{T_\kappa,t}^\sigma)$ 
yields
\begin{align*}
&{-}\left(X_t-x_{T_\kappa,t}^\sigma\right)\cdot\big(\W_{\mu_t^{X}}\left(X_t\right)-\W_{\lambda_1}\left(x_{T_\kappa,t}^\sigma\right)\big)\\
&={-}2\left(X_t-x_{T_\kappa,t}^\sigma\right)\cdot\left(\W_{\mu_t^{X}}\left(X_t\right)-\W_{\mu_t^{X}}\left(x_{T_\kappa,t}^\sigma\right)\right){-}2\left(X_t-x_{T_\kappa,t}^\sigma\right)\cdot\left(\W_{\mu_t^{X}}\left(x_{T_\kappa,t}^\sigma\right)-\W_{\lambda_1}\left(x_{T_\kappa,t}^\sigma\right)\right)\\
&\leq{-}2\left(X_t-x_{T_\kappa,t}^\sigma\right)\cdot\left(\W_{\mu_t^{X}}\left(X_t\right)-\W_{\mu_t^{X}}\left(x_{T_\kappa,t}^\sigma\right)\right)+\big{\|}X_t-x_{T_\kappa,t}^\sigma\big{\|}\,\big{\|}\F\star\mu_t^X\left(x_{T_\kappa,t}^\sigma\right)-\F(x_{T_\kappa,t}^\sigma-\lambda_1)\big{\|}\,.
\end{align*}

\noindent{}{\bf Step 3.1. }  Due to the contractivity \eqref{Hyp:Synchro} of
$-\W_{\mu^X_t}$, we have

\begin{equation}\label{temp-0-1}
-\left(X_t-x_{T_\kappa,t}^\sigma\right)\cdot\left(\W_{\mu_t^X}\left(X_t\right)-\W_{\mu_t^X}\left(x_{T_\kappa,t}^\sigma\right)\right)\leq-\theta\big{\|}X_t-x_{T_\kappa,t}^\sigma\big{\|}^2,\qquad t\ge T_{\kappa}\,.
\end{equation}
Let us now estimate the difference 
$\F\star\mu_t^X\left(x_{T_\kappa,t}^\sigma\right)-\F(x_{T_\kappa,t}^\sigma-\lambda_1)$: according to $\mathbf{(A)}$-$(iii)$,
\begin{align*}
&\big{\|}\F\star\mu_t^X\left(x_{T_\kappa,t}^\sigma\right)-\F(x_{T_\kappa,t}^\sigma-\lambda_1)\big{\|}\\
&\leq\int_{\bRb^d}\big{\|}\F\left(x_{T_\kappa,t}^\sigma-u\right)
-\F(x_{T_\kappa,t}^\sigma-\lambda_1)\big{\|}\mu^X_t(\ddu)\\
&\leq
C\int_{\bRb^d}\|u-\lambda_1\|\left(1+\big{\|}x_{T_\kappa,t}^\sigma-u\big{\|}^{2n}+\big{\|}x_{T_\kappa,t}^\sigma-\lambda_1\big{\|}^{2n}\right)\mu^X_t(\ddu)\\
&\le 
C\int_{\bRb^d}\|u-\lambda_1\|\left(1+\big{\|}x_{T_\kappa,t}^\sigma-\lambda_1\big{\|}^{2n}+\big{\|}u-\lambda_1\big{\|}^{2n}\right)\mu^X_t(\ddu)\,.
\end{align*}
Using Cauchy-Schwarz inequality and the uniform-in-time moments control \eqref{Moment_SS} of $X$, and setting 
\begin{equation}\label{temp-0-0}
\psi(r):=C(n)(1+r^{2n}),
\end{equation}
we obtain
\begin{equation}\label{temp-0-2}
\big{\|}\F\star\mu_t^X\left(x_{T_\kappa,t}^\sigma\right)-\F(x_{T_\kappa,t}^\sigma-\lambda_1)\big{\|}
\leq \big(\mathbb E\left[\|X_{T_\kappa}-\lambda_1\|^2\right]\big)^{1/2}
\psi(\|x_{T_\kappa,t}^\sigma-\lambda_1\|)\,.
\end{equation}

\noindent{}{\bf Step 3.2. } Now, we introduce the \emph{deterministic} time:
\begin{equation*}
	S_\kappa^:=\inf\left\{t\geq
	T_\kappa\,\,:\,\,\mathbb E\left[\|X_{T_\kappa}-\lambda_1\|^2\right]\geq\kappa^2\right\}\,.
\end{equation*}

On the interval $[T_\kappa,S_\kappa]$, we have a good uniform control for \eqref{temp-0-2} for the coupling between the processes $X$ and $x_{T_\kappa,\cdot}^\sigma$. (As mentioned above, in the final {\bf Step 4.} below, we will show that
the time $S_\kappa$ is larger than $\exp\{\frac{2H}{\sigma^2}\}$ for any $H>0$, provided that $\sigma$ is small enough.)

For the moment, we consider the time interval $\big[T_\kappa,\min\left\{\exp\left[\frac{2H}{\sigma^2}\right],S_\kappa\right\}\big]$ and set \[
\displaystyle\mathcal{B}_{R}^\sigma:=\Big\{\omega\in\Omega\,\,:\,\,\sup_{T_\kappa\leq
t\leq\min\left\{\exp\left[\frac{2H}{\sigma^2}\right];S_\kappa\right\}}\big{\|}x_{T,t}^\sigma-\lambda_1\big{\|}\leq
R\Big\}=\left\{\tau_R\geq\min\left\{\exp\left[\frac{2H}{\sigma^2}\right],S_\kappa\right\}\right\}
\]
 where $\tau_R$ is the first time
that the process $x_{T_\kappa,\cdot}^\sigma$ exits from the ball $\mathbb B(\lambda_1;R)$.  
Let us show that, for $R$ large
enough, the probability of $\mathcal{B}_{R}^\sigma$ tends to $1$ as $\sigma$ goes to $0$, starting from:
\begin{align*}
\PP\left(\left(\mathcal{B}_{R}^\sigma\right)^c\right)&=\PP\left(\tau_R<\min\left\{\exp\left[\frac{2H}{\sigma^2}\right];S_\kappa\right\}\right)\leq\PP\left(\tau_R<\exp\left[\frac{2H}{\sigma^2}\right]\right)\,.
\end{align*}
Next, choose $R$ large enough such that the exit-cost associated to $\tau_R$ (recalling Theorem \ref{thm:KramersDZ})
\[ 
\underline{h}_R:=\frac 1{4}\inf_{x\in \partial \mathbb B(\lambda_1;R)}\inf_{-\infty<T_1\le T_2<\infty}\inf_\phi\left\{\int_{T_1}^{T_2} \Big\|\frac{\dd\phi_t}{\ddt}+\W_{\lambda_1}(\phi_t)\Big\|^2_{ A^{-1}(\phi_t)}\, \ddt\,:\, \phi\in\mathcal
A^{\lambda_1,x}[T_1,T_2]
\right\}
\]
is larger than $H+1$. Such a radius $R$ exists since, due to Assumption $(\mathbf{A})$-$(vi)$, we get, for $\theta$ as in \eqref{Hyp:Synchro},
\begin{align*}
\Big\|\frac{\dd\phi_t}{\ddt}+\W_{\lambda_1}(\phi_t)\Big\|^2_{ A^{-1}(\phi_t)}&\geq\Lambda_-\Big\|\frac{\dd\phi_t}{\ddt}+\W_{\lambda_1}(\phi_t)\Big\|^2=\Lambda_-\Big\|\frac{\dd\phi_t}{\ddt}+\theta\big(\phi_t-\lambda_1\big)+\W_{\lambda_ 1}(\phi_t)-\theta\big(\phi_t-\lambda_1\big)\Big\|^2\\
&\geq\Lambda_-\Big[\Big\|\frac{\dd\phi_t}{\ddt}+\theta(\phi_t-\lambda_1)\Big\|^2+\Big\|\W_{\lambda_1}(\phi_t)-\theta\big(\phi_t-\lambda_1\big)\Big\|^2\\
&\quad +2\left(\frac{\dd\phi_t}{\ddt}+\theta\big(\phi_t-\lambda_1\big)\right)\cdot\left(\W_{\lambda_1}(\phi_t)-\theta\big(\phi_t-\lambda_1\big)\right)\Big]\,.
\end{align*}
Owing to the synchronization assumption ${\bf (A)}-(iv)$, we have $(\phi_t-\lambda_1)\cdot \W_{\lambda_1}(\phi_t)\ge \theta \|\phi_t-\lambda_1\|^2$ and so
\begin{align*}
&\left(\frac{\dd\phi_t}{\ddt}+\theta\big(\phi_t-\lambda_1\big)\right)\cdot \left(\W_{\lambda_1}(\phi_t)-\theta\big(\phi_t-\lambda_1\big)\right)\\
&\geq\frac{\dd\phi_t}{\ddt}\cdot \left(\W_{\lambda_1}(\phi_t)-\theta\big(\phi_t-\lambda_1\big)\right)\geq-\frac1{2}\Big\|\frac{\dd\phi_t}{\ddt}\Big\|^2-\frac1{2}\Big\|\W_{\lambda_1}(\phi_t)-\theta\big(\phi_t-\lambda_1\big)\Big\|^2\,.
\end{align*}
Thus 
\[
\Big\|\frac{\dd\phi_t}{\ddt}+\W_{\lambda_1}(\phi_t)\Big\|^2_{ A^{-1}(\phi_t)}\geq\Lambda_-\left(\Big\|\frac{\dd\phi_t}{\ddt}+\theta\big(\phi_t-\lambda_1\big)\Big\|^2-\Big\|\frac{\dd\phi_t}{\ddt}\Big\|^2\right)\geq2\Lambda_-\theta\frac{\dd\phi_t}{\ddt}\cdot\big(\phi_t-\lambda_1\big)\,.
\]
Integrating the above over the time interval $[T_1,T_2]$, we get 
$$
\int_{T_1}^{T_2}\Big\|\frac{\dd\phi_t}{\ddt}+\W_{\lambda_1}(\phi_t)\Big\|^2_{ A^{-1}(\phi_t)}\,\ddt\ge \Lambda_-\theta\big(\|\phi_{T_2}-\lambda_1\|^2-\|\phi_{T_1}-\lambda_1\|^2\big)\,,
$$
and so 
\begin{align*}
\underline{h}_R &\ge\frac{1}{4}\inf_{x\in \partial \mathbb B(\lambda_1;R)}\inf_{-\infty<T_1\le T_2<\infty}\Lambda_-\theta\inf_\phi\left\{\big(\|\phi_{T_2}-\lambda_1\|^2-\|\phi_{T_1}-\lambda_1\|^2\big)\,:\, \phi\in\mathcal
A^{\lambda_1,x}[T_1,T_2]
\right\}= \Lambda_-\theta R^2\longrightarrow+\infty\,,
\end{align*}
as $R$ tends to infinity. (The last equality follows from the fact that; on $\mathcal
A^{\lambda_1,x}[T_1,T_2]$,  $\phi_{T_2}$ lies in the boundary of $\mathbb B(\lambda_1;R)$ and $\phi_{T_1}=\lambda_1$.) Then, by Theorem \ref{thm:KramersDZ} we have: for any $\delta>0$,
\[
\lim_{\sigma\to0}\PP\left(\tau_R<\exp\left[\frac{2(H+1)-\delta}{\sigma^2}\right]\right)=0\,.
\]
In particular, taking $\delta=2$ yields $\lim_{\sigma\to0}\PP\left(\left(\mathcal{B}_{R}^\sigma\right)^c\right)=0$.

Let us now consider \eqref{Coupling_SS_Proof_1} (with \eqref{temp-0-1} and \eqref{temp-0-2}) under the event $\mathcal{B}_{R}^\sigma\bigcap\mathcal{A}^\sigma$.
Setting $\zeta_t:=\|X_t-x_{T_\kappa,t}^\sigma\|^2$, since, by Lemma \ref{chante}, $\big(\mathbb E\left[\|X_{T_\kappa}-\lambda_1\|^2\right]\big)^{1/2}\leq\kappa$, we have, for all $T_{\kappa}\le t\le \min\{\exp[\frac{2H}{\sigma^2}];S_\kappa\}$,
\[
	\zeta_t\le \sqrt{\sigma}+\int_{T_\kappa}^t\Big(\big(C(\sigma^2+\frac\sigma{2})
	-2\theta\big)\sqrt{\zeta_s}+\psi(R)\kappa\Big)\sqrt{\zeta_s}\dds\,,
\]
with $\psi$ given by \eqref{temp-0-0}. Considering now $\sigma$ sufficiently small so that $C\big(\sigma^2+\frac{\sigma}{2}\big)-2\theta\le -\theta$, on the set $\Big\{\sup_{T_\xi\leq t\leq
	\exp\left[\frac{2H}{\sigma^2}\right]}\|X_t-x_{T_\xi,t}^\sigma\|\geq\xi\Big\}$ we get
	$$
	\bigg(C(\sigma^2+\frac\sigma{2})
	-2\theta\bigg)\sqrt{\zeta_s}+\psi(R)\kappa\le -\theta\xi+\psi(R)\kappa\,.
	$$
	Given $\xi$, we shall then choose $\kappa$ small enough so that $\psi(R)\kappa/\theta<\xi$ to eventually obtain:
\[
\zeta_t\leq \sqrt{\sigma},\qquad T_{\kappa}\le t\le\min\left\{\exp\Big[\frac{2H}{\sigma^2}\Big];S_\kappa\right\}\,. 
\]

Letting $\sigma$ tends to $0$ achieves to prove that
\[
\lim_{\sigma\to0}\PP\left(\mathcal{B}_{R}^\sigma\bigcap\mathcal{A}^\sigma\bigcap\left\{\sup_{T_\kappa\leq
t\leq\min\left\{\exp\left[\frac{2H}{\sigma^2}\right];S_\kappa\right\}}\|X_t-x_{T_\xi,t}^\sigma\|\geq\xi\right\}\right)=0\,.
\]
Recalling that
$\lim_{\sigma\to0}\PP\left(\left(\mathcal{B}_{R}^\sigma\right)^c\right)=\lim_{\sigma\to0}\PP\left(\left(\mathcal{A}^\sigma\right)^c\right)=0$, we conclude

\begin{equation}
\label{label_bizarre3}
\lim_{\sigma\to0}\PP\left(\sup_{T_\kappa\leq
t\leq\min\left\{\exp\left[\frac{2H}{\sigma^2}\right],S_\kappa\right\}}\|X_t-x_{T_\xi,t}^\sigma\|\geq\xi\right)=0\,.
\end{equation}

\noindent{}{\bf Step 4.} To conclude the proof it remains to prove that, for any $\kappa>0$, $S_\kappa\geq\exp\left[\frac{2H}{\sigma^2}\right]$ if $\sigma$ is
small enough. To do so, we first extend some arguments of \cite{Tugaut2021} (Lemma 4.1 therein), putting  $\xi_t:=\EE\left[\|X_t-\lambda_1\|^2\right]$, we can observe that
\begin{align*}
\frac{\dd\xi}{\ddt}(t)&=-2\mathbb E\Big[\left(X_t-\lambda_1\right)\cdot\left(\W_{\mu^X_t}(X_t)\right)\Big]+\sigma^2\mathbb E\Big[\mathrm{Trace}\big(\Gamma\Gamma^*\big)(X_t)\Big]\\
&\le -2\mathbb E\Big[\left(X_t-\lambda_1\right)\cdot\U(X_t)\Big]-2\mathbb E\Big[\left(X_t-\lambda_1\right)\cdot\left(\F\star\mu^X_t(X_t)\right)\Big]+\sigma^2(\Lambda_+)^2d\,.
\end{align*}
Using the property \eqref{Hyp-F2_bis} of $\F$ with $\mu=\mathcal L(X_t-\lambda_1)$, 
$$
\mathbb E\Big[\left(X_t-\lambda_1\right)\cdot\left(\F\star\mu^X_t(X_t)\right)\Big]\ge 0\,.
$$
Hence the following estimate holds
\begin{equation}\label{temp-3}
\frac{\dd}{\ddt}\xi_t\leq -2\rho\xi_t+M\sqrt{\PP(X_t\notin\mathcal{S}_\rho)}+ K\sigma^2\,,
\end{equation}
for $K$ and $M$ are positive constants, and for $\mathcal{S}_\rho$ the path-connected subset  which contains $\lambda_1$ of the set

\[
\left\{x\in\bRb^d\,\,:\,\,(x-\lambda_1)\cdot\U(x)\geq\rho\|x-\lambda_1\|^2\right\}\,.
\]

Next, choosing  $\rho>0$ is sufficiently small, this set is not empty since the Jacobian of $-\U$ in the point $\lambda_1$ is a negative definite matrix. By taking $r>0$ small enough, 
we have the inclusion $\mathbb{B}(\lambda_1;r)\subset\mathcal{S}_\rho$ which immediately implies:

\[
\PP\Big(X_t\notin\mathcal{S}_\rho\big)\leq\PP\Big(X_t\notin\mathbb{B}(\lambda_1;r)\Big)\,.
\]

We also remark

\begin{equation}\label{temp1}
\PP\Big(X_t\notin\mathbb{B}(\lambda_1;r)\Big)\leq\PP\left(\|X_t-x_{T_\kappa,t}^\sigma\|\geq\frac{r}{2}\right)+\PP\left(\|x_{T_\kappa,t}^\sigma-\lambda_1\|\geq\frac{r}{2}\right)\,.
\end{equation}

The second upper-bound is easily controlled: due to the contractivity of the drift $\W_{\lambda_1}$, proceeding similarly as in \cite{Tugaut2021}
and taking $T_\kappa$ large enough (that is to say $\kappa$ small enough), we obtain the following limit:

\begin{equation}\label{temp2}
\lim_{\sigma\to0}\PP\left(\|x_{T_\kappa,t}^\sigma-\lambda_1\|\geq\frac{r}{2}\right)=0\,.
\end{equation}

The first one will be controlled in the following way. Let us assume that there would exist a decreasing sequence
$(\sigma_n)_{n\in\mathbb{N}}$ with $\lim_{n\to\infty}\sigma_n=0$ such that $S_\kappa<\exp\left[\frac{2H}{\sigma_n^2}\right]$. 
Then, using successively \eqref{temp1}, \eqref{temp2} and \eqref{label_bizarre3} (and since, by assumption, $S_\kappa=\min\{S_\kappa;\exp\left[\frac{2H}{\sigma_n^2}\right]\}$)
\begin{align*}
\lim_{n\to\infty}\sup_{T_\kappa\leq t\leq
S_\kappa}\PP(X_t\notin\mathbb{B}(\lambda_1;r))&\leq\lim_{n\to\infty}\PP\left(\sup_{T_\kappa\leq t\leq
S_\kappa}\|X_t-x_{T_\kappa,t}^{\sigma_n}\|\geq\frac{r}{2}\right)\\
&\quad=\lim_{n\to\infty}\PP\left(\sup_{T_\kappa\leq t\leq
\min\left\{\exp\left[\frac{2H}{\sigma_n^2}\right];S_\kappa\right\}}\|X_t-x_{T_\kappa,t}^{\sigma_n}\|\geq\frac{r}{2}\right)=0\,.
\end{align*}
Then, coming back to \eqref{temp-3}, for any $T_\kappa\leq t\leq S_\kappa$, we find:
\[
\frac{\dd}{\ddt}\xi_t\leq f(\sigma_n)-2\rho\xi_t\,,
\]
$f$ being an increasing function with $f(0)=0$. Consequently, for any small enough $\sigma_n$ (that is large enough $n$), we would have
that $\xi_t<\frac{\kappa^2}{4}$ for any $T_\kappa\leq t\leq S_\kappa$. In particular: $\kappa^2=\xi_{S_\kappa}<\frac{\kappa^2}{4}$
which is impossible. We thus deduce $S_\kappa\geq\exp\left[\frac{2H}{\sigma^2}\right]$ as soon as $\sigma$ is small enough. The proof is achieved.

\end{proof}

\subsection{Proof of Theorem \ref{thm:main1}}
With the above, we are ready to achieve Theorem \ref{thm:main1}. Since the vector fields
\begin{equation}
\label{GlobalPot}
\W_{\lambda_1}(x)=\U(x)+\F(x-\lambda_1)\,,\,\quad\W_{\lambda_2}(x)=\U(x)+\F(x-\lambda_2)\,,
\end{equation}
driving \eqref{Linearized-MV1} and \eqref{Linearized-MV2} are contractive, applying Lemma \ref{dale-bis} in Section \ref{sec:LinearCase} - up to a time shift,
allowed by the Markov property of $(x_{T_\xi,\cdot}^\sigma,y_{T_\xi,\cdot}^\sigma)$ - the Freidlin-Wentzell' exit-time law Theorem \ref{thm:KramersDZ} applies to 
\[
\tau_{\lambda,\varepsilon}(\sigma):=\inf\left\{t\geq
T_\xi\,\,:\,\,(x_{T_\xi,t}^\sigma,y_{T_\xi,t}^\sigma)\in\mathbb{B}\left(\lambda;\varepsilon\right)\times\mathbb{B}\left(\lambda;\varepsilon\right)\right\}\,,
\]
where $T_\xi$ is given as in Proposition \ref{label_bizarre2}. Namely, it holds that: for any $\delta>0$:
\begin{equation}
\label{label_bizarre5}
\lim_{\sigma\to0}\PP\left(\exp\left[\frac{2}{\sigma^2}\left(\widehat{H}_\varepsilon(\lambda)-\delta\right)\right]<\tau_{\lambda,\varepsilon}(\sigma)<\exp\left[\frac{2}{\sigma^2}\left(\widehat{H}_\varepsilon(\lambda)+\delta\right)\right]\right)=1\,,
\end{equation}
\begin{equation*}
	\mbox{and}\quad\lim_{\sigma\to0}\PP\left({\rm
		dist}\left((x^\sigma_{\tau_{\lambda,\varepsilon}(\sigma)},y^\sigma_{\tau_{\lambda,\varepsilon}(\sigma)}),\mathbb{B}(\lambda;\varepsilon)\times\mathbb{B}(\lambda;\varepsilon)\right)\leq\delta\right)=1\,,
	\end{equation*}
where
\begin{equation*}
	\widehat H_\varepsilon(\lambda)=\left\{
	\begin{aligned}
		&\inf_{x\in\mathbb R^d,y\in\partial \mathbb
			B(\lambda;\varepsilon)}V(x,y)&\,\mbox{if}\,\,&\|\lambda-\lambda_1\|< \varepsilon\,,\\
		&\inf_{x\in\partial \mathbb
			B(\lambda;\varepsilon),y\in\mathbb R^d} V(x,y)&\,\mbox{if}\,\,&\|\lambda-\lambda_2\|< \varepsilon,\\
		&\inf_{x,y\in \partial
			\mathbb{B}(\lambda;\varepsilon)} V(x,y)&\,\mbox{if}\,\,&\min_i\|\lambda-\lambda_i\|>\varepsilon,
	\end{aligned}
	\right.
\end{equation*}
\begin{equation}
V(x,y)=\inf_{-\infty<T_1\le T_2<\infty}\inf_{\phi}\left\{I(T_1,T_2,\phi)\,:\,\phi \in\mathcal A^{(\lambda_1,\lambda_2),(x,y)}[T_1,T_2]\right\},
\end{equation}
the action functional $I$ and set $\mathcal A^{(\lambda_1,\lambda_2),(x,y)}[T_1,T_2]$ being defined as in Theorem \ref{thm:main1}.

The asymptotic \eqref{label_bizarre5} and the collision-location property stated in Lemma \ref{dale-bis} - taking also into account the
succeeding discussion in Section \ref{subsec:Linear-collisionB} on the simplification of the minimizer sets - both shift to the self-stabilizing diffusions, $X$ and $Y$ which leads to the
following result:

\begin{prop}
\label{torche} The first entering-time of $(X,Y)$ in the domain
$\mathbb{B}(\lambda;\varepsilon)\times\mathbb{B}\left(\lambda;\varepsilon\right)$:
\[
\beta_{\lambda,\varepsilon}(\sigma):=\inf\left\{t\geq0\,\,:\,\,(X_t,Y_t)\in\mathbb{B}\left(\lambda;\varepsilon\right)\times\mathbb{B}\left(\lambda;\varepsilon\right)
\right\}\,,
\]
satisfies, for any $\lambda\in\bRb^d$ and any $\delta>0$,

\begin{equation}
\label{label_bizarre7}
\lim_{\sigma\to0}\PP\left(\exp\left[\frac{2}{\sigma^2}\left(\widehat{H}_\varepsilon(\lambda)-\delta\right)\right]<\beta_{\lambda,\varepsilon}(\sigma)<\exp\left[\frac{2}{\sigma^2}\left(\widehat{H}_\varepsilon(\lambda)+\delta\right)\right]\right)=1\,,
\end{equation}
and
\begin{equation*}
\lim_{\sigma\to0}\PP\left({\rm dist}\bigg((X_{\widehat \beta_{\lambda,\varepsilon}(\sigma)},Y_{\widehat
\beta_{\lambda,\varepsilon}(\sigma)}),\mathbb B(\lambda;\varepsilon)\times\mathbb B(\lambda;\varepsilon)\bigg)\le \delta\right)=1\,.
\end{equation*}
\end{prop}

\begin{proof}Fix $\delta>0$ and let $\sigma$ and $\xi$ be small enough so that $T_\xi$ given in \eqref{MeanCoupling-MV} is smaller
than~$\exp\{\frac{2}{\sigma^2}(\widehat{H}_\varepsilon(\lambda)-\delta)\}$. As  $\lim_{\sigma\rightarrow
0}\beta_{\lambda,\varepsilon}(\sigma)=\infty$ a.s., the event $\{T_\xi> \beta_{\lambda,\varepsilon}(\sigma)\}$ becomes negligible at
the limit $\sigma\downarrow 0$. On the remaining event $\{T_\xi\le \beta_{\lambda,\varepsilon}(\sigma)\}$, according to
Lemma~\ref{label_bizarre2}, $(X_t,Y_t)$ and $(x_{T_\xi,t}^\sigma,y_{T_\xi,t}^\sigma)$ are at distance of at most $\xi$ from each others.
This way, for $\xi<\varepsilon$, $\widehat{\beta}_{\lambda,\varepsilon}(\sigma)$ necessarily lies in the interval
$[\tau_{\lambda,\varepsilon+\xi}(\sigma),\tau_{\lambda,\varepsilon-\xi}(\sigma)]$. As $\eta\mapsto \widehat{H}_\eta(\lambda)$ is
continuous, we can further choose $\xi$ again small enough  so that $H_\varepsilon(\lambda)+\delta\ge
\widehat{H}_{\varepsilon-\xi}(\lambda)+\delta'$ and $\widehat{H}_\varepsilon(\lambda)-\delta\le
\widehat{H}_{\varepsilon+\xi}(\lambda)-\delta''$, for some ${\delta'}
,{\delta''}>0$. The Freidlin-Wentzell' type law~\eqref{label_bizarre5} then ensures
\begin{align*}
\lim_{\sigma\rightarrow
0}\PP\left(\exp\left[\frac{2}{\sigma^2}\left(\widehat{H}_{\varepsilon}(\lambda)-\delta\right)\right]<\beta_{\lambda,\varepsilon}(\sigma)\right)\ge\lim_{\sigma\rightarrow
0}\PP\left(\exp\left[\frac{2}{\sigma^2}\left(\widehat{H}_{\varepsilon+\xi}(\lambda)-\delta''\right)\right]<\tau_{\lambda,\varepsilon+\xi}(\sigma)\right)\
=1\,.
\end{align*}
and
\begin{align*}
\lim_{\sigma\rightarrow
0}\PP\left(\beta_{\lambda,\varepsilon}(\sigma)<\exp\left[\frac{2}{\sigma^2}\left(\widehat{H}_{\varepsilon}(\lambda)+\delta\right)\right]\right)\ge \lim_{\sigma\rightarrow
0}\PP\left(\tau_{\lambda,\varepsilon-\xi}(\sigma)<\exp\left[\frac{2}{\sigma^2}\left(\widehat{H}_{\varepsilon-\xi}(\lambda)+\delta'\right)\right]\right)=1\,.
\end{align*}
The asymptotic of $(X_{\beta_{\lambda,\varepsilon}(\sigma)},Y_{\beta_{\lambda,\varepsilon}(\sigma)})$ is an immediate consequence of
the very definition of~$\beta_{\lambda,\varepsilon}(\sigma)$.
\end{proof}

As a consequence of the above, we immediately deduce the asymptotic of the first time that the diffusions $X$ and $Y$ are at a
distance $2\varepsilon$.

\begin{prop}
\label{label_bizarre8} For $\Psi_1$ and $\Psi_2$ as in~\eqref{GlobalPot}, define
\begin{equation}\label{epsiloncollisioncost}
	H_\varepsilon(\lambda):=\inf_{x,y\in \partial \mathbb{B}(\lambda;\varepsilon)}V(x,y)\,,
\end{equation}
and the threshold

\begin{equation}\label{ColliRad2}\varepsilon_c=\inf\left\{\varepsilon\le \varepsilon_0\,:\,\min\{M_{1,\varepsilon};M_{2,\varepsilon}\}>\inf_\lambda
H_\varepsilon(\lambda),\,\,i\neq j\in\{1,2\}\right\}\,.
\end{equation}
for 
\[
M_{1,\varepsilon}=\inf_{x\in\partial\mathbb B(\lambda_1;\varepsilon),y\in\mathbb R^d}V(x,y),\quad M_{2,\varepsilon}=\inf_{x\in\mathbb R^d,y\in\partial\mathbb B(\lambda_2;\varepsilon)}V(x,y).
\]
Let $C_\varepsilon(\sigma)$ be as in \eqref{near-collision-time-SS}, $\varepsilon\in(0,\varepsilon_c)$, $\underline
H_\varepsilon=\min H_\varepsilon$ and let $\mathcal M_\varepsilon$ be the set of minimizers of~$H_\varepsilon$. Then, for any
$\delta>0$,
\begin{equation*}
\lim_{\sigma\to0}\PP\left(\exp\left[\frac{2}{\sigma^2}\left(\underline
H_\varepsilon-\delta\right)\right]<{C}_\varepsilon(\sigma)<\exp\left[\frac{2}{\sigma^2}\left(\underline
H_\varepsilon+\delta\right)\right]\right)=1\,,
\end{equation*}
and
\begin{equation*}
\lim_{\sigma\to0}\PP\left(\inf_{\lambda_\varepsilon\in\mathcal M_\varepsilon}\max\bigg({\rm
dist}\big(X_{C_\varepsilon(\sigma)},\mathbb B(\lambda_\varepsilon;\varepsilon)\big),{\rm dist}\big(Y_{C_\varepsilon(\sigma)},\mathbb
B(\lambda_\varepsilon;\varepsilon)\big)\bigg)\ge \delta\right)=0\,.
\end{equation*}
\end{prop}
\begin{proof}The assumption $\varepsilon<\varepsilon_c$ immediately ensures $\inf_{\lambda}\widehat
H_\varepsilon(\lambda)=\inf_{\lambda} H_\varepsilon(\lambda)$,
and the proof is readily  similar to the one of Proposition \ref{lacollision} replacing $z_t^\sigma$ by the barycentre
$Z_t:=2^{-1}(X_t+Y_t)$.
\end{proof}

In the same way Proposition~\ref{lacollision} yielded Theorem~\ref{lacollisionthm}, Proposition \ref{label_bizarre8} yields to
Theorem~\ref{thm:main1}.

\section{On the first individual near-collision between the particle systems \texorpdfstring{\eqref{particles}}{}}\label{sec:ParticleCase}

Essentially, the proof steps to establish Theorem \ref{thm:main2} are similar to the ones carrying Theorem \ref{thm:main1}: observing that
\[
C^i_{\varepsilon,N}(\sigma)=\inf_{\lambda\in\mathbb R^d}\beta^i_{\lambda,\varepsilon,N}(\sigma), \qquad
\beta^i_{\lambda,\varepsilon,N}(\sigma)=\inf\Big\{t>0\,:\,(X^{i,N}_t,Y^{i,N}_t)\in\mathbb B(\lambda;\varepsilon)\times\mathbb
B(\lambda;\varepsilon)\Big\}\,,
\]
we aim to establish a ``good'' coupling estimate between the particle systems $(X^{i,N},Y^{i,N})$ and their linearized versions: 
\begin{subequations}
	\begin{equation}
		\label{xi}
		x_{T,t}^{i,\sigma}=X^{i,N}_T-\int_T^t\Big(\mathbb{U}(x_{T,s}^{i,\sigma})+\F(x_{T,s}^{i,\sigma}-\lambda_1)\Big)\,\dds+\sigma\int_{T}^t \Gamma(x_{T,t}^{i,\sigma})\,\dd B_s^i\,,\qquad\,t\ge T\,,
	\end{equation}
	and
	\begin{equation}
		\label{yi}
		y_{T,t}^{i,\sigma}=Y^{i,N}_T -\int_T^t\Big(\mathbb{U}(y_{T,s}^{i,\sigma})+\F(y_{T,s}^{i,\sigma}-\lambda_2)\Big)\,\dds+\sigma\int_{T}^t \Gamma(y_{T,s}^{i,\sigma})\,\dd\widetilde B_s^i\,,\qquad\,t\ge
		T\,,
	\end{equation}
\end{subequations}
starting at an appropriate time $T$. From this coupling, the exponential growth of $\beta^i_{\lambda,\varepsilon,N}(\sigma)$ and the location of the pair $(X^{i,N},Y^{i,N})$ at time $\beta^i_{\lambda,\varepsilon,N}(\sigma)$ become direct consequences of the results established in Section \ref{sec:LinearCase} (particularly Lemma \ref{dale-bis}). Next, replicating successively the proof arguments of Propositions \ref{lacollision} and \ref{torche} will lead to the asymptotics of $\inf_{\lambda\in\mathbb R^d}\beta^i_{\lambda,\varepsilon,N}(\sigma)$ and the location of the pair $(X^{i,N},Y^{i,N})$ at time $\inf_{\lambda}\beta^i_{\lambda,\varepsilon,N}(\sigma)$ for $\sigma\downarrow 0$ and $\varepsilon\downarrow 0$.
 
 The main difficulty here relies on exhibiting the coupling. For the case of self-stabilizing diffusions (Section \ref{sec:SelfStabilizingCase}), such coupling was essentially handled by establishing the concentration of the laws $(\mu^X_t,\mu^Y_t)$ on $(\lambda_1,\lambda_2)$  
 over an expanding time-period $\left[T_{\kappa},\exp\big[2H/\sigma^2\big]\right]$, $H>0$ (see Lemma \ref{chante} and {\bf Step 4.} in the proof of Proposition \ref{label_bizarre2}). Proceeding in the same way with the particle systems \eqref{MFSP1}-\eqref{MFSP2} requires to establish analogs of this control with the empirical measures $\frac{1}{N}\sum_{j=1}^N\delta_{X^{j,N}_t}$
and~$\frac{1}{N}\sum_{j=1}^N\delta_{Y^{j,N}_t}$ in place of  $\mu^X_t$ and $\mu^Y_t$, making the coupling naturally trickier than for the self-stabilizing diffusions. The counterpart of the coupling estimate established in Proposition \ref{label_bizarre2} is eventually obtained in Proposition \ref{label_bizarre9} below, and the proof of Theorem \ref{thm:main2} then completed in Section \ref{ssec:ProofMain2}.

\subsection{Coupling estimates for the particle systems}\label{ssec:CouplingParticle}

As a preliminary step, let us point out that the whole particle systems \eqref{MFSP1} and \eqref{MFSP2} can be equivalently formulated as some pair of $\mathbb R^{dN}$-valued diffusion processes $\mathbf X^N:=(X^{1,N},\cdots,X^{N,N})$ and $\mathbf
Y^N:=(Y^{1,N},\cdots,Y^{N,N})$ satisfying
\[
\mathbf X^{N}_t=\mathbf{x}_1^N-\int_0^t \mathbf{\Upsilon_N}(\mathbf X^N_s)\,\dds+\sigma\int_0^t\mathbf{\Gamma}(\mathbf X^{N}_s)\,\dd \mathbf B_s,\qquad t\ge 0\,,
\]
\[
\mathbf Y^{N}_t=\mathbf{x}_2^N-\int_0^t\mathbf{\Upsilon_N}(\mathbf Y^N_s)\,\dds+\sigma \int_0^t\mathbf{\Gamma}(\mathbf Y^{N}_s)\,\dd\widetilde{\mathbf B}_s,\qquad t\ge 0\,,
\]
the initial states $\mathbf{x}_1^N$ and $\mathbf{x}_2^N$ being respectively the common points $(x_1,\cdots,x_1)$ and $(x_2,\cdots,x_2) $, $\mathbf B:=(B^1,\cdots,B^N)$ and $\widetilde{\mathbf B}:=(\widetilde{B}^1,\cdots,\widetilde{B}^N)$, the diffusion $\mathbf \Gamma$ corresponding to the diagonal block $Nd\times Nd$-matrix function endowing $\Gamma$ as its main diagonal component, and the driving drift being given by the $\mathbb R^{dN}$-vector field:
\[
\mathbf{\Upsilon_N}:\mathbf x^N=(x_1,\cdots,x_N)\in\mathbb R^{dN}\mapsto \mathbf{\Upsilon_N}(\mathbf
x^N)=\Big(\U(x_1)+\frac{1}{N}\sum_{j=1}^N\F(x_1-x_j),\cdots,\U(x_N)+\frac{1}{N}\sum_{j=1}^N\F(x_N-x_j)\Big).
\]
\begin{rem} We may remark that the fixed (stable) points for the flows generated by $\mathbf{\Upsilon_N}$ are of the form $(x_0,\cdots,x_0)$ where $x_0$ is a zero of $\U$. Indeed, one can first be observed that all zeros $\mathbf{x^N}=(x_1,\cdots,x_N)$ of the $\mathbf{\Upsilon_N}$ have necessarily the common form $(x_1,\cdots,x_1)$ for some $x_1\in\mathbb R^d$. This directly follows from the synchronization assumption $\mathbf{(A)}$-$(iv)$  which grants, for any coordinates $x_i$ and $x_j$ of $\mathbf{x^N}$  and $\mu_{\mathbf{x^N}}=\frac1{N}\sum_{j=1}^N\delta_{\{x_j\}}$, 
	\begin{align*}
		0=(x_i-x_j)\cdot\Big(\left(\U+\F\star\mu_{\mathbf{x^N}}\right)(x_i)-\left(\U+\F\star\mu_{\mathbf{x^N}}\right)(x_j)\Big)=(x_i-x_j)\cdot\left(\W_{\mu_{\mathbf{x^N}}}(x_i)-\W_{\mu_{\mathbf{x^N}}}(x_j)\right)\ge \theta \|x_i-x_j\|^2\,.
	\end{align*}
	Next, the claim immediately follows as necessarily $\F\star \mu_{\mathbf{x^N}}(x_1)=0$ and so
	$$
	(x_1,\cdots,x_1)=\Big(\mathbf{\Upsilon_N}\Big)^{-1}(\mathbf{0})\Leftrightarrow 0=\U(x_1)+\F\star \mu_{\mathbf{x^N}}(x_1)=\U(x_1)\,.
	$$
\end{rem}

The analogue to the coupling estimate stated in Proposition \ref{label_bizarre2} for self-stabilizing diffusions is established for the particle systems with the following.

	\begin{prop}\label{label_bizarre9}
		For any $H>0$ and for any $\xi>0$ small enough, there exist a rank
		$N_{\xi}\ge 1$ and a positive finite time $T_{\xi}$ such that, for any $N\ge
		N_{\xi}$,
		\begin{align*}
			&\lim_{\sigma\rightarrow 0}\PP\left(\max_{T_{\xi}\le t\le \exp[\frac {2H}{\sigma^2}]}\|X_t^{i,N}-x_{T_{\xi},t}^{i,\sigma}\|\geq\xi\right)=0=\lim_{\sigma\rightarrow 0}\PP\left(\max_{T_{\xi}\le t\le \exp[\frac {2H}{\sigma^2}]}\|Y_t^{i,N}-y_{T_{\xi},t}^{i,\sigma}\|\geq\kappa\right)\,.
		\end{align*}
	\end{prop}
\begin{proof} 

\noindent{}{\bf Outline of the proof.} Each proof step is comparable to the one of Proposition \ref{label_bizarre2}. In a preliminary step, we set a starting time $T_\kappa$ at which the coupling between $X^{i,N}$ and $x^{i,\sigma}_{T_{\kappa},\cdot}$ can be suitably defined (the parameter $\kappa$ being taken later in function of the bound $\xi$). Subsequently, in {\bf Steps 1} to {\bf 3}, we derive a control for the quadratic distance of this coupling up to  $T_{\kappa}$ a time of the form $\min\{\exp\left[\frac{2H}{\sigma^2}\right];S^{(1)}_\kappa;S^{(2)}_\kappa\}$, $S^{(1)}_\kappa$ and $S^{(2)}_\kappa$ being some stopping-times. In a last {\bf Steps 4}, we prove that with a probability close to $1$ for evanescent noise, then $S^{(1)}_\kappa;S^{(2)}_\kappa$ are larger than $\exp\left[\frac{2H}{\sigma^2}\right]$, achieving so the proof. This step particulary notably carries the technical difficulty of controlling the coupling between $X^{i,N}$ and $x^{i,\sigma}_{T_{\kappa},\cdot}$ compared to the one in Proposition \ref{label_bizarre2}.
%

\medskip
 
\noindent{}{\bf Step 0. }Fix $\sigma$ and $\kappa>0$. Invoking the large deviation principle \eqref{LDP}, we may choose a finite non-random time $T_\kappa$ such that 

\[
\PP\left(\left\{\frac 1N\sum_{i=1}^N\|X^{i,N}_{T_\kappa}-\lambda_1\|^{2n}\geq \frac{\kappa^{2n}}{4}\right\}\bigcap\mathcal{C}^\sigma\right)=0\,,
\]
where $n$ is the polynomial order in ${\bf (A)}$-$(iii)$ and $\mathcal{C}^\sigma\subset\Omega$ is satisfying $\PP\left(\mathcal{C}^\sigma\right)\longrightarrow1$ as $\sigma$
 goes to zero.

\medskip

\noindent{}{\bf Step 1. }Applying It\^o's formula yields, for any $t\geq T_\kappa$,
\begin{align*}
\|X_t^{i,N}-x_{T_\kappa,t}^{i,\sigma}\|^{2}
=&-2\int_{T_\kappa}^t\left(X_s^{i,N}-x_{T_\kappa,s}^{i,\sigma}\right)\cdot\bigg(\W_{\mu_s^{N}}\left(X_s^{i,N}\right)-\W_{\lambda_1}\left(x_{T_\kappa,s}^{i,\sigma}\right)\bigg)\dds\\
&+\sigma^2\int_{T_\kappa}^t\mathrm{Trace}\Big(\left(\Gamma(X_s^{i,N})-
\Gamma(x_{T_\kappa,s}^{i,\sigma})\right)\left(\Gamma^*(X_s^{i,N})-
\Gamma^*(x_{T_\kappa,s}^{i,\sigma})\right)\Big)\dds\\
&+2\sigma\int_{T_\kappa}^t\left(X_s^{i,N}-x_{T_\kappa,s}^{i,\sigma}\right)\cdot\bigg(\Gamma\left(X_s^{i,N}\right)-\Gamma\left(x_{T_\kappa,s}^{i,\sigma}\right)\bigg)\dd B_s^i\,,
\end{align*}
setting $\mu_t^{N}=\frac{1}{N}\sum_{j=1}^N\delta_{X_t^{j,N}}$, and
recalling that $\W_\mu(x)=\U(x)+\F\star\mu(x)$ and $\W_{\lambda_1}(x)=\U(x)+\F(x-\lambda_1)$. As $\Gamma$ is Lipschitz continuous: 
\begin{align*}
	\|X_t^{i,N}-x_{T_\kappa,t}^{i,\sigma}\|^{2}\le &-2\int_{T_\kappa}^t\left(X_s^{i,N}-x_{T_\kappa,s}^{i,\sigma}\right)\cdot\bigg(\W_{\mu_s^{N}}\left(X_s^{i,N}\right)-\W_{\lambda_1}\left(x_{T_\kappa,s}^{i,\sigma}\right)\bigg)\dds\\
	&+2\sigma\int_{T_\kappa}^t\left(X_s^{i,N}-x_{T_\kappa,s}^{i,\sigma}\right)\cdot\bigg(\Gamma\left(X_s^{i,N}\right)-\Gamma\left(x_{T_\kappa,s}^{i,\sigma}\right)\bigg)\dd B_s^i+C\sigma^2\int_{T_\kappa}^t\|X_s^{i,N}-x_{T_\kappa,s}^{i,\sigma}\|^{2}\,\dds\,.
\end{align*}

\medskip
 
\noindent{}{\bf Step 2. }Let us first handle the martingale upper-bound of the above, writing
\[
M_{T_\kappa,t}^i:=\int_{T_\kappa}^t\Theta_s^i\,\dd B_s^i:=\int_{T_\kappa}^t\left(X_s^{i,N}-x_{T_\kappa,s}^{i,\sigma}\right)\cdot\bigg(\Gamma\left(X_s^{i,N}\right)-\Gamma\left(x_{T_\kappa,s}^{i,\sigma}\right)\bigg)\dd B_s^i,\ \ t\ge T_\kappa\,.
\]
and $\langle M\rangle$, the quadratic variation of $M$,
\begin{align*}
&\langle M\rangle_{T_\kappa,t}^i=\int_{T_\kappa}^t\|\Theta_s^i\|^2\,\dds\\
&=\int_{T_\kappa}^t\left(X_s^{i,N}-x_{T_\kappa,s}^{i,\sigma}\right)\cdot\left(\big(\Gamma\left(X_s^{i,N}\right)-\Gamma(x_{T_\kappa,s}^{i\sigma})\big)\big(\Gamma(X_s^{i,N})-\Gamma(x_{T_\kappa,s}^{i,\sigma}\big)^*\right)\left(X_s^{i,N}-x_{T_\kappa,s}^{i,\sigma}\right)\,\dds\,.
\end{align*}
Given $H>0$ arbitrary, define the set
\[
\mathcal{A}^\sigma:=\left\{\omega\in\Omega\,\,:\,\,\forall\,t\in\Big[T_\kappa,\exp\left[\frac{2H}{\sigma^2}\right]\Big],\, M_t^i<\frac{1}{2}\left\langle
M\right\rangle_t^i+\frac{1}{\sqrt{\sigma}}\right\}\,.
\] 
Proceeding as in {\bf Step 2} of the proof of Lemma \ref{label_bizarre2}; namely invoking the boundness of $\Gamma$, the uniform-in-time moment control~\eqref{Moment_particle} and the similar control
$$
\sup_{t\ge 0}\mathbb E[\|x_{T_\kappa,t}^{i,\sigma}\|^{2p}]\le c^p(p-1)!\,,
$$
 Lemma \ref{lem:MartingaleConcentration} yields
 \[
 \PP\left(\left(\mathcal{A}^\sigma\right)^c\right)\leq\exp\Big[-\frac{1}{\sqrt{\sigma}}\Big]\,.
 \]
Hence, for any $\omega\in\mathcal{A}^\sigma$ and for any $t\geq T_\kappa$, 
  we have the inequality $\sigma
M_t\leq\frac{\sigma}{2}\left\langle M\right\rangle_t+\sqrt{\sigma}$, where, Using again the boundedness of $\Gamma$,
\[
\left\langle M\right\rangle_t\leq (\Lambda_+)^2d\int_{T_\kappa}^t\big{\|}X_s^{i,N}-x_{T_\kappa,s}^{i\sigma}\big{\|}^2\dds\,.
\]
Hence, on $\mathcal{A}^\sigma$, we have 
\begin{align}\label{Coupling_SS_Proof_1bis}
	\|X_t^{i,N}-x_{T_\kappa,t}^{i,\sigma}\|^2\leq&\sqrt{\sigma}+C\big(\sigma^2+\frac{\sigma}{2}\big)\int_{T_\kappa}^t\big{\|}X_s^{i,N}-x_{T_\kappa,s}^{i,\sigma}\big{\|}^2\dds\nonumber\\
	&-2\int_{T_\kappa}^t\left(X_s^{i,N}-x_{T_\kappa,s}^{i,\sigma}\right)\cdot\bigg(\W_{\mu_s^{N}}\left(X_s^{i,N}\right)-\W_{\lambda_1}\left(x_{T_\kappa,s}^{i,\sigma}\right)\bigg)\dds\,.
\end{align}

\noindent{}{\bf Step 3. }We now deal with the expression:
$-\left(X_t^{i,N}-x_{T_\kappa,t}^{i,\sigma}\right)\cdot\Big(\W_{\mu_t^{N}}\left(X_t^{i,N}\right)-\W_{\lambda_1}\left(x_{T_\kappa,t}^{i,\sigma}\right)\Big)$. Adding and subtracting $\W_{\mu_t^{N}}(x_{T_\kappa,t}^{i,\sigma})$ yields

\begin{align*}
&-\left(X_t^{i,N}-x_{T_\kappa,t}^{i,\sigma}\right)\cdot\big(\W_{\mu_t^{N}}\left(X_t^{i,N}\right)-\W_{\lambda_1}\left(x_{T_\kappa,t}^{i\sigma}\right)\big)\\
&=-\left(X_t^{i,N}-x_{T_\kappa,t}^{i,\sigma}\right)\cdot\left(\W_{\mu_t^{N}}\left(X_t^{i,N}\right)-\W_{\mu_t^{N}}\left(x_{T_\kappa,t}^{i,\sigma}\right)\right)-\left(X_t^{i,N}-x_{T_\kappa,t}^{i,\sigma}\right)\cdot\left(\W_{\mu_t^{N}}\left(x_{T_\kappa,t}^{i,\sigma}\right)-\W_{\lambda_1}\left(x_{T_\kappa,t}^{i,\sigma}\right)\right)\\
&\leq-\left(X_t^{i,N}-x_{T_\kappa,t}^{i,\sigma}\right)\cdot\left(\W_{\mu_t^{N}}\left(X_t^{i,N}\right)-\W_{\mu_t^{N}}\left(x_{T_\kappa,t}^{i,\sigma}\right)\right)+\big{\|}X_t-x_{T_\kappa,t}^\sigma\big{\|}\,\big{\|}\F\star\mu_t^{N}\left(x_{T_\kappa,t}^ {i,\sigma}\right)-\F(x_{T_\kappa,t}^{i,\sigma}-\lambda_1)\big{\|}\,.
\end{align*}

\noindent{}{\bf Step 3.1. } According to \eqref{Hyp-F1}, we have the following estimate after using the triangular inequality:
\begin{align}\label{Control-F-Particles}
&\big{\|}\F\star\mu_t^{N}\left(x_{T_\kappa,t}^{i,\sigma}\right)-\F(x_{T_\kappa,t}^{i,\sigma}-\lambda_1)\big{\|}\le \int_{\mathbb R^d}\big{\|}\F(x_{T_\kappa,t}^{i,\sigma}-u)-\F(x_{T_\kappa,t}^{i,\sigma}-\lambda_1)\big{\|}\mu_t^{N}(\ddu)\nonumber\\
&\le  K\int_{\bRb^d}\|u-\lambda_1\|\left(1+\big{\|}x_{T_\kappa,t}^{i,\sigma}-u\big{\|}^{2n-1}+\big{\|}x_{T_\kappa,t}^{i,\sigma}-\lambda_1\big{\|}^{2n-1}\right)\mu_t^{N}(\ddu)\nonumber\\
&\leq \frac{K}{N}\sum_{j=1}^N\|X_t^{j,N}-\lambda_1\|\left(1+\big{\|}x_{T_\kappa,t}^{i,\sigma}-\lambda_1\big{\|}^{2n-1}+\|X_t^{j,N}-x_{T_\kappa,t}^{i,\sigma}\|^{2n-1}\right)\nonumber\\
&\leq \frac{K}{N}\sum_{j=1}^N\|X_t^{j,N}-\lambda_1\|\left(1+\big{\|}x_{T_\kappa,t}^{i,\sigma}-\lambda_1\big{\|}^{2n-1}+2^{2n-1}\left[\|X_t^{j,N}-\lambda_1\|^{2n-1}+\|x_{T_\kappa,t}^{i,\sigma}-\lambda_1\|^{2n-1}\right]\right)\nonumber\\
&\leq 2^{2n-1}\left(1+\big{\|}x_{T_\kappa,t}^{i,\sigma}-\lambda_1\big{\|}^{2n-1}\right)\frac{K}{N}\sum_{j=1}^N\|X_t^{j,N}-\lambda_1\|\left(1+\|X_t^{j,N}-\lambda_1\|^{2n-1}\right)\,.
\end{align}
Using Cauchy-Schwarz inequality next leads to
\[
\big{\|}\F\star\mu_t^{N}\left(x_{T_\kappa,t}^{i,\sigma}\right)-\F(x_{T_\kappa,t}^{i,\sigma}-\lambda_1)\big{\|}\leq\eta(\|x_{T_\kappa,t}^{i,\sigma}-\lambda_1\|)\sqrt{\frac{1}{N}\sum_{j=1}^N\|X_t^{j,N}-\lambda_1\|^2}\sqrt{1+\frac{1}{N}\sum_{j=1}^N\|X_t^{j,N}-\lambda_1\|^{4n-2}}\,,
\]
for
\begin{equation}\label{temp4}
\eta(r):=C(1+r^{2n})\,,
\end{equation}
 $C$ being some constant.

\noindent{}{\bf Step 3.2. } To handle the above, we introduce the \emph{random} times:
\begin{equation*}
S_\kappa^{(1)}:=\inf\left\{t\geq
T_\kappa\,\,:\,\,\frac{1}{N}\sum_{j=1}^N\|X_t^{j,N}-\lambda_1\|^{2}\geq\kappa^{2}\right\}\,,
\end{equation*}
and
\[
S_\kappa^{(2)}:=\inf\left\{t\geq
T_\kappa\,\,:\,\,\frac{1}{N}\sum_{j=1}^N\|X_t^{j,N}-\lambda_1\|^{4n-2}\geq R^{4n-2}\right\}\,,
\]
and then we put 
$$
S_\kappa:=\min\Big\{S_\kappa^{(1)};S_\kappa^{(2)};\exp\big[\frac{2H}{\sigma^2}\big]\Big\}\,.
$$
We remark that, at time $S_\kappa^{(2)}$, at least one of the particle has already exited the Euclidean ball of center $\lambda_1$ and radius $R$. In the following, $R>1$ is intended to be large enough.

In {\bf Step 4}, we will show that, in probability, $S_\kappa$ is of order $\exp[\frac{2H}{\sigma^2}]$ for any $H>0$, provided that $\sigma$ is small enough and $R$, defining $S^{(2)}_\kappa$, large enough. For the present step, put
\[
\mathcal{B}_{R}^{i,\sigma}:=\Big\{\omega\in\Omega\,\,:\,\,\sup_{T_\kappa\leq
t\leq S_\kappa}\big{\|}x_{T,t}^{i,\sigma}-\lambda_1\big{\|}\leq
R-1\Big\}=\left\{\tau_{R-1}^i\geq S_\kappa\right\}
\]
 where $\tau_{R-1}^i$ is the first time
that the process $x_{T_\kappa,\cdot}^{i,\sigma}$ exits from the ball $\mathbb B(\lambda_1;R-1)$. As in {\bf Step 3.2. } of the proof of Proposition~\ref{label_bizarre2}, for $R$ large
enough, the probability of $\mathcal{B}_{R}^{i,\sigma}$ tends to $1$ as $\sigma\downarrow 0$. Setting  $\mathcal{B}_{R}^{\sigma}:=\bigcap_{i=1}^N\mathcal{B}_{R}^{i,\sigma}$, let us consider \eqref{Coupling_SS_Proof_1bis} under the set $\mathcal{B}_{R}^\sigma\bigcap\mathcal{A}^\sigma$. Due to the contractivity \eqref{Hyp:Synchro} of
$x\mapsto -\W_{\mu_t^{N}}(x)$, we have
\[
-\left(X_t^{i,N}-x_{T_\kappa,t}^{i,\sigma}\right)\cdot\left(\W_{\mu_t^{N}}\left(X_t^{i,N}\right)-\W_{\mu_t^{N}}\left(x_{T_\kappa,t}^{i,\sigma}\right)\right)\leq-\theta\big{\|}X_t^{i,N}-x_{T_\kappa,t}^{i,\sigma}\big{\|}^2,\qquad t\ge T_{\kappa}\,.
\]
Coming back to \eqref{Coupling_SS_Proof_1bis}, and setting $\zeta_t^i:=\|X_t^{i,N}-x_{T_\kappa,t}^{i,\sigma}\|^2$, for all $T_{\kappa}\le t\le \min\{\exp[\frac{2H}{\sigma^2}];S_\kappa\}$
, it holds
\begin{align*}
	\zeta_t^i\le \sqrt{\sigma}+\int_{T_\kappa}^t\Big(\big(C(\sigma^2+\frac\sigma{2})
	-2\theta\big)\sqrt{\zeta_s^i}+\eta(R)\kappa\Big)\sqrt{\zeta_s^i}\dds\,.
\end{align*}
for $\eta$ given in \eqref{temp4}. 
Considering now $\sigma$ sufficiently small so that $C\big(\sigma^2+\frac{\sigma}{2}\big)-2\theta\le -\theta$, on the set $\Big\{\sup_{T_\kappa\leq t\leq
	\exp\left[\frac{2H}{\sigma^2}\right]}\|X_t^{i,N}-x_{T_\xi,t}^{i,\sigma}\|\geq\xi\Big\}$, we have:
	$$
	\bigg(C(\sigma^2+\frac\sigma{2})
	-2\theta\bigg)\sqrt{\zeta_s^i}+\eta(R)\kappa\le -\theta\xi+\eta(R)\kappa\,.
	$$
Given now $\xi>0$, we shall then choose $\kappa$ small enough so that $\eta(R)\kappa/\theta<\xi$ to eventually obtain:
\[
\zeta_t^i\leq \sqrt{\sigma},\qquad T_{\kappa}\le t\le \min\left\{\exp\Big[\frac{2H}{\sigma^2}\Big];S_\kappa\right\}\,.
\]
Without loss of generality, we relabel $T_\kappa$ into $T_\xi$.
Letting $\sigma$ tends to $0$ achieves to prove that
\[
\lim_{\sigma\to0}\PP\left(\mathcal{C}^\sigma\bigcap\mathcal{B}_{R}^\sigma\bigcap\mathcal{A}^\sigma\bigcap\left\{\sup_{T_\xi\leq
t\leq\min\left\{\exp\left[\frac{2H}{\sigma^2}\right],S_\kappa\right\}}\|X_t^{i,N}-x_{T_\xi,t}^{i,\sigma}\|\geq\xi\right\}\right)=0\,,
\]
for any $1\leq i\leq N$. Recalling that
$\lim_{\sigma\to0}\PP\left(\left(\mathcal{C}^\sigma\right)^c\right)=\lim_{\sigma\to0}\PP\left(\left(\mathcal{B}_{R}^\sigma\right)^c\right)=\lim_{\sigma\to0}\PP\left(\left(\mathcal{A}^\sigma\right)^c\right)=0$, we deduce that
\[
\lim_{\sigma\to0}\PP\left(\sup_{1\leq i\leq N}\sup_{T_\xi\leq
t\leq\min\left\{\exp\left[\frac{2H}{\sigma^2}\right],S_\kappa\right\}}\|X_t^{i,N}-x_{T_\kappa,t}^{i,\sigma}\|\geq\xi\right)=0,
\]
and so conclude our claim up to the random time $S_\kappa$. 

\medskip
	
\noindent{}{\bf Step 4.} To conclude the proof, it now remains to show that as $\sigma\rightarrow 0$, the probability of the event $\left\{\exp\left[\frac{2H}{\sigma^2}\right]> S_\kappa\right\}$ tends to $0$. This final point is achieved in the two next sub-steps.

\noindent{}{\bf Step 4.1.} Suppose that  
$S_\kappa^{(2)}<\min\left\{\exp\left[\frac{2H}{\sigma^2}\right];S_\kappa^{(1)}\right\}$. If so, on the interval $(S^{(2)}_\kappa,S_\kappa]=\left(S^{(2)}_\kappa,\min\left\{\exp\left[\frac{2H}{\sigma^2}\right];S_\kappa^{(1)}\right\}\right]$, there would exist a time $t$ such that, for some $1\leq i\leq N$, $\|X_t^{i,N}-x_{T_\kappa,t}^{i,\sigma}\|\leq\kappa$ with high probability (this probability being the one of $\mathcal{C}^\sigma\bigcap\mathcal{B}_{R}^\sigma\bigcap\mathcal{A}^\sigma$). However, since $\|X_t^{i,N}-\lambda_1\|\leq\|x_{T_\kappa,t}^{i,\sigma}-\lambda_1\|+\|X_t^{i,N}-x_{T_\kappa,t}^{i,\sigma}\|$, we would have at least one particle say with label $i$ such that $\|x^{i,\sigma}_{T_\kappa,t}-\lambda_1\|\ge R-1$, providing that $\kappa<1$. Yet denoting by $\tau^{i}(\sigma)$ the first exit-time of $x^{i,\sigma}_{T_\kappa,\cdot}$ outside the ball $\mathbb B(\lambda_1,R-1)$, we have
\begin{align*}
\mathbb P\Big(\inf_{1\le i\le N}\tau^{i}(\sigma)\le \exp\left[\frac{2H}{\sigma^2}\right]\Big)\le N\mathbb P\Big(\tau^{i}(\sigma)\le \exp\left[\frac{2H}{\sigma^2}\right]\Big)
\end{align*}
According to Theorem \ref{thm:KramersDZ},  $\frac{\sigma^2}2\log\big(\tau^i(\sigma)\big)\overset{\mathbb P}{\underset{\sigma\downarrow 0}{\asymp}}H_R$, 
where $H_R$ is the exit-cost of $x^{i,\sigma}_{T_\kappa,\cdot}$ from the ball $\mathbb B(\lambda_1;R-1)$. Recalling the argument of {\bf Step 3.2} from the proof of Proposition \ref{label_bizarre2}, $H_R$ grows to $\infty$ as $R$ increases. Choosing $R$ large enough  so that $H_R\ge H$ yields $S^{(2)}_\kappa\ge \exp\left[\frac{2 H}{\sigma^2}\right]$ and by extension to the contradiction $S^{(2)}_\kappa\ge \inf\left\{\exp\left[\frac{2 H}{\sigma^2}\right];S^{(1)}_\kappa\right\}$, with a probability $f(\sigma)$ tending to $1$ as $\sigma$ vanishes.

\noindent{}{\bf Step 4.2.} Proceeding as in {\bf Step 4} for Proposition \ref{label_bizarre2},  we now prove that for any $\kappa>0$, $S^{(1)}_\kappa\geq\exp\left[\frac{2H}{\sigma^2}\right]$ if $\sigma$ is
small enough. To do so, we work on the event  $\mathcal{A}^\sigma\bigcap\mathcal{B}_R^\sigma\bigcap\mathcal{C}^\sigma$, where $\mathcal{C}^\sigma$ (introduced in {\bf Step 0}) and $\mathcal{A}^\sigma$ and $\mathcal{B}_R^\sigma$ (set in {\bf Step 3}) have each probability $1$ as $\sigma$ tends to $0$.

\noindent{}{\bf Step 4.2.1.} We first observe that 
\begin{align*}
\xi_t^N&:=\frac1N\sum_{i=1}^N\|X^{i,N}_t-\lambda_1\|^{2}\nonumber\\
&=\frac1N\sum_{i=1}^N	\|X^{i,N}_{T_\kappa}-\lambda_1\|^{2}-\frac{2}N\sum_{i=1}^N\int_{T_\kappa}^t(X^{i,N}_s-\lambda_1)\cdot\U(X^{i,N}_s)\,\dds-\frac{2}{N^2}\sum_{i,j=1}^N\int_{T_\kappa}^t(X^{i,N}_s-\lambda_1)\cdot\F(X^{i,N}_s-X^{j,N}_s)\,\dds\nonumber\\
 		&\quad+\frac{2\sigma}N\sum_{i=1}^N\int_{T_\kappa}^t\big(X^{i,N}_s-\lambda_1\big)\cdot\Gamma(X^{i,N}_s)\,\dd B^i_s +\frac{\sigma^2}N\sum_{i=1}^N\int_{T_\kappa}^t\mathrm{Trace}\big(\Gamma(X^{i,N}_s)\Gamma^*(X^{i,N}_s)\big)\,\dds\,,
 \end{align*}
 	where, for all $s\ge T_\kappa$, according to the Property \eqref{Hyp-F2} of $\F$,
\[
\sum_{i,j=1}^N(X^{i,N}_s-\lambda_1)\cdot\F(X^{i,N}_s-X^{j,N}_s)=-\frac12\sum_{i,j=1}^N\Big(X^{i,N}_s-X^{j,N}_s\Big)\cdot\F(X^{i,N}_s-X^{j,N}_s)\le 0\,.
\]
Using Lemma \ref{lem:MartingaleConcentration}, as in {\bf Step 3.}, on  $\mathcal{C}^\sigma\bigcap\mathcal{A}^\sigma$, 
$$
\frac{2\sigma}N\sum_{i=1}^N\int_{T_\kappa}^t\big(X^{i,N}_s-\lambda_1\big)\cdot\Gamma(X^{i,N}_s)\,\dd B^i_s<\sqrt{\sigma}+\frac{d(\Lambda_+)^2\sigma^2}{N^2}\sum_{i=1}^N\int_{T_\kappa}^t\|X_s^{i,N}-\lambda_1\|^2\dds\,.
$$
Therefore
\begin{align}\label{EmpiricalIto_bis}
	\xi_t^N&\le \xi_{T_\kappa}^N-\frac2{N}\sum_{i=1}^N\int_{T_\kappa}^t(X^{i,N}_s-\lambda_1)\cdot\U(X^{i,N}_s)\,\dds+\frac{d(\Lambda_+)^2\sigma^2}{N}\int_{T_\kappa}^t\xi_s^N\dds+\sqrt{\sigma}+\sigma^2\int_{T_\kappa}^td(\Lambda_+)^2\dds\,.
\end{align}

\noindent{}{\bf Step 4.2.2.} Introducing $\mathcal{S}_\rho:=\left\{x\in\bRb^d\,\,:\,\,(x-\lambda_1)\cdot\U(x)\geq\rho\|x-\lambda_1\|^2\right\}$, we recall from {\bf Step 4.} of Proposition~\ref{label_bizarre2} that this set is not empty and contains $\lambda_1$ in its interior if $\rho>0$ is sufficiently small. 
Recalling the property \eqref{Hyp-V1} of $\U$, the set $\mathbb R^d\setminus\mathbb B(\lambda_1;R'')\subset\mathcal{S}_\rho$ for $\rho$ (also) small enough and thus the complementary 

$$
\mathcal{S}_\rho^c=\left\{x\in\bRb^d\,\,:\,\,(x-\lambda_1)\cdot\U(x)<\rho\|x-\lambda_1\|^2\right\}
$$
 is included into the ball of center $\lambda_1$ and radius $R_\rho>0$ large enough. Recalling also the property \eqref{Hyp-V2}, let $\gamma>0$ be such that for any $x\in\bRb^d$, we have $(x-\lambda_1)\cdot\U(x)\geq-\gamma\|x-\lambda_1\|^2$. Then we deduce from \eqref{EmpiricalIto_bis}:
\begin{align*}
\xi_t^N
&\leq\xi_{T_\kappa}^N+\sqrt{\sigma}+(2\rho+\gamma)\sum_{j=1}^N\int_{T_\kappa}^t\|X_s^{j,N}-\lambda_1\|^{2}\mathds{1}_{\{X_s^{j,N}\notin\mathcal{S}_\rho\}}\dds+
\int_{T_\kappa}^t\Big(\big(\frac{d(\Lambda_+)^2\sigma^2}{N}-2\rho\big)\xi_s^N+\sigma^2d(\Lambda_+)^2\Big)\,\dds
 \,,
\end{align*}
using the splitting:
\begin{align*}
-\sum_{i=1}^N\int_{T_\kappa}^t(X^{i,N}_s-\lambda_1)\cdot\U(X^{i,N}_s)\mathds{1}_{\{X_s^{j,N}\in\mathcal{S}_\rho\}}\,\dds&\le -\rho\sum_{i=1}^N\int_{T_\kappa}^t\|X^{i,N}_s-\lambda_1\|^2\mathds{1}_{\{X_s^{j,N}\in \mathcal{S}_\rho\}}\,\dds\\
&\quad=-\rho\sum_{i=1}^N\int_{T_\kappa}^t\|X^{i,N}_s-\lambda_1\|^2\,\dds+\rho\sum_{i=1}^N\int_{T_\kappa}^t\|X^{i,N}_s-\lambda_1\|^2\mathds{1}_{\{X_s^{j,N}\notin \mathcal{S}_\rho\}}\,\dds
\end{align*}
and 
\begin{align*}
	-\sum_{i=1}^N\int_{T_\kappa}^t(X^{i,N}_s-\lambda_1)\cdot\U(X^{i,N}_s)\mathds{1}_{\{X_s^{j,N}\notin\mathcal{S}_\rho\}}\,\dds\le \gamma  \sum_{i=1}^N\int_{T_\kappa}^t\|X^{i,N}_s-\lambda_1\|^2\mathds{1}_{\{X_s^{j,N}\notin\mathcal{S}_\rho\}}\,\dds\,.
\end{align*}
Due to the boundedness of $\mathcal{S}_\rho^c$, for sufficiently small $\sigma$, we eventually obtain the estimate
\begin{equation}\label{Intermediate-2706}
\xi_t^N\leq\xi_{T_\kappa}^N+\sqrt{\sigma}+(2\rho+\gamma)(R_\rho)^2\int_{T_\kappa}^tp_s^{N}\dds+\int_{T_\kappa}^t\Big(C\sigma^2+(\frac{\sigma^2d(\Lambda_+)^2}{N}-2\rho)\xi_s^{N}\Big)\dds\,,
\end{equation}
where $p_s^{N}$ is the proportion of particles $\{X^{i,N}\}_{i=1}^N$ not in $\mathcal{S}_\rho$ at time $s$.

\noindent{}{\bf Step 4.2.3.} By taking $r>0$ small enough, we know the inclusion $\mathbb{B}(\lambda_1;r)\subset\mathcal{S}_\rho$ holds which immediately implies $p_s^{N}\leq q_s^{N}$ where $q_s^{N}$ is the proportion of particles $X^{i,N}$ not in the ball $\mathbb{B}(\lambda_1;r)$ at time $s$. However, taking $\kappa$ smaller than $\frac{r}{2}$, we find that, before time $S_\kappa$, on the event $\mathcal{A}^\sigma\bigcap\mathcal{B}_R^\sigma\bigcap\mathcal{C}^\sigma$ (which, we recall, is of probability close to $1$ as $\sigma$ tends to $0$), $\|X_s^{i,N}-x_{T_\kappa,s}^{i,\sigma}\|<\frac{r}{2}$. In particular, this implies that $q_s^{N}$ is dominated by $\widetilde{q}_s^{N}:=\frac{1}{N}\#\left\{1\leq i\leq N\,\,:\,\,\|x_{T_\kappa,s}^{i,\sigma}-\lambda_1\|>\frac{r}{2}\right\}$ again with high probability. More precisely,
 we have: for all $c\in(0,1)$,
\begin{equation}\label{temp5}
\mathbb P\Big(q_s^{N}\le \widetilde{q}_s^{N}+c,\, s\in[T_{\kappa},S_{\kappa}]\Big)\underset{\sigma\rightarrow0}{\longrightarrow}1\,.
\end{equation}
Indeed, almost surely,
$$
q_s^{N}\le \widetilde{q}_s^{N}+\frac{1}{N}\#\left\{1\leq i\leq N\,\,:\,\,\|X^{i,N}_t-x_{T_\kappa,t}^{i,\sigma}\|>\frac{r}{2}\right\},
$$
and, for any fixed $i$, 
$$
\nu(\sigma):=\mathbb P\Big(\|X^{i,N}_s-x_{T_\kappa,s}^{i,\sigma}\|>\frac{r}{2},\,\forall\, t\in[T_{\kappa},S_{\kappa}]\Big)\underset{\sigma\rightarrow 0}{\longrightarrow} 0,
$$
Taking $k$, the lower integer of $cN$, we obtain
\begin{align*}
	&\mathbb P\Big(\frac{1}{N}\#\left\{1\leq i\leq N\,\,:\,\,\max_{s\in[T_\kappa,S_{\kappa}]}\|X^{i,N}_s-x_{T_\kappa,s}^{i,\sigma}\|>\frac{r}{2}\right\}\ge c\Big)\\
	&=\mathbb P\Big(\#\left\{1\leq i\leq N\,\,:\,\,\max_{s\in[T_\kappa,S_{\kappa}]}\|X^{i,N}_s-x_{T_\kappa,s}^{i,\sigma}\|>\frac{r}{2}\right\}\ge k\Big)\\
	&\le\sum_{\ell=k}^N\mathbb P\Big(\mbox{Exactly}\,\ell\,\mbox{particles satisfy}\,\max_{s\in[T_\kappa,S_{\kappa}]}\|X^{i,N}_s-x_{T_\kappa,s}^{i,\sigma}\|>\frac{r}{2}\Big)\le \sum_{\ell=k}^N\binom{N}{\ell}\big(\nu(\sigma)\big)^\ell\underset{\sigma\rightarrow0}{\longrightarrow}0\,.
\end{align*}
%
Coming back to \eqref{Intermediate-2706}, the above \eqref{temp5} yields, with a probability close to $1$ as $\sigma$ tends to $0$,
\begin{equation}\label{Intermediate-2706-2}
	\xi_t^N\leq\xi_{T_\kappa}^N+\sqrt{\sigma}+(2\rho+\gamma)(R_\rho)^2\int_{T_\kappa}^t \big(\widetilde{q}_s^{N}+c\big)\dds+\int_{T_\kappa}^t\Big(C\sigma^2+\big(\frac{d(\Lambda_+)^2\sigma^2}{N}-2\rho\big)\xi_s^{N}\Big)\dds\,,
\end{equation}
where $c$ can be chosen arbitrarily small.

\noindent{}{\bf Step 4.2.4.}  It now remains to prove that $\widetilde{q}_s^{N}$ is small with $\sigma$ for any $T_\kappa\leq s\leq\exp\left[\frac{2H}{\sigma^2}\right]$. For $c'\in(0,1)$ (again) arbitrary and $r>0$, we have
\begin{align*}
\mathbb P\Big(\widetilde{q}^N_s\ge c',\,\forall s\in[T_\kappa,S_\kappa]\Big)&=\mathbb P\Big(\#\left\{i\,:\,x^{i,\sigma}_{T_\kappa,s}\notin\mathbb B(\lambda_1;\frac r{2})\right\}\ge c'N,\,\forall s\in[T_\kappa,S_\kappa]\Big)\\
&\le \mathbb P\Big(\sum_{i=1}^N\|x^{i,\sigma}_{T_\kappa,s}-\lambda_1\|^2\ge r^2c'N/4,\,\forall s\in[T_\kappa,S_\kappa]\Big)\,.
\end{align*}
Introducing the exit-time
$$
{\bf \tau}^{N}(\sigma,r,c')=\inf\left\{s\ge T_\kappa\,:\,\sum_{i=1}^N\|x^{i,\sigma}_{T_\kappa,s}-\lambda_1\|^2\ge r^2c'N/4\right\}\,,
$$
we get
$$
\mathbb P\Big(\widetilde{q}^N_s\ge c',\,\forall s\in[T_\kappa,S_\kappa]\Big)\le \mathbb P\Big({\bf \tau}^{N}(\sigma,r,c')\le S_\kappa\Big)\le \mathbb P\left({\bf \tau}^{N}(\sigma,r,c')\le \exp\left[\frac{2H}{\sigma^2}\right]\right)\,.
$$
Applying Theorem \ref{thm:KramersDZ} to the whole particle system $(x^{1,\sigma}_{T_\kappa},\cdots, x^{1,\sigma}_{T_\kappa})$ and to the exit domain 
$$
\mathcal O:=\{{\bf x}=(x_1,\cdots,x_N)\in\mathbb R^{dN}\,:\,\frac1{N}\sum_{i=1}^N\|x_i-\lambda_1\|^2\le c'r^{2}/4\}
$$ 
(which is naturally stable by the drift of the system), we know that $\frac{\sigma^2}2\log\big({\bf \tau}^{N}(\sigma,r,c')\big)\overset{\mathbb P}{\underset{\sigma\downarrow 0}{\asymp}}{\bf H}_{N,r,c'}$, 
and where
\begin{align*}
{\bf H}_{N,r,c'}
	=\frac 1{4}\inf_{{\bf y}\in\partial\mathcal O }\inf_{-\infty<T_1\le T_2<\infty}\inf_{{\bf \phi}=(\phi^1,\cdots,\phi^N)}\left\{ \sum_{i=1}^N\int_{T_1}^{T_2} 
	\Big\|\frac{\dd  \phi^i_t}{\ddt}+\W_{\lambda_1}(\phi^i_t)\Big\|^2_{A^{-1}(\phi^i_t)}\,\ddt\,:\,{\bf \phi}\in \mathcal A^{(\lambda_1,\cdots,\lambda_1),{\bf y}}[T_1,T_2]\right\}
\end{align*}
where $\partial \mathcal O=\left\{{\bf y}=(y_1,\cdots,y_N)\in\mathbb R^{Nd},\,\frac1{N}\sum_{i=1}^N\|y_i-\lambda_1\|^2=c'r^2/4\right\}$. Replicating the arguments from {\bf Step 3.2.} in the proof of Proposition \ref{label_bizarre2}, we have
\begin{align*}
\sum_{i=1}^N\int_{T_1}^{T_2} 
\Big\|\frac{\dd  \phi^i_t}{\ddt}+\W_{\lambda_1}(\phi^i_t)\Big\|^2_{A^{-1}(\phi^i_t)}\,\ddt&\ge \Lambda_-\theta\sum_{i=1}^N\Big(\|\phi^i_{T_2}-\lambda_1\|^2-\|\phi^i_{T_1}-\lambda_1\|^2\Big)=\Lambda_-\theta N\frac{c'r^2}{4}\,.
\end{align*}
We can next take $N$ large enough so that $N\ge (\Lambda_-\theta c'r^2/4)^{-1}H$ which implies 
$$
\mathbb P\Big(\widetilde{q}^N_s\ge c',\,\forall s\in[T_\kappa,S_\kappa]\Big)\le \mathbb P\Big({\bf \tau}^{N}(\sigma,r,c')\le S_\kappa\Big)\le \mathbb P\left({\bf \tau}^{N}(\sigma,r,c')\le \exp\left[\frac{2H}{\sigma^2}\right]\right)\underset{\sigma\downarrow 0}{\longrightarrow} 0\,.
$$
Plugged into \eqref{Intermediate-2706-2}, the above eventually yields, for all $t\in[T_\kappa,S_\kappa]$, with a high probability as $\sigma\downarrow 0$:
\begin{equation*}
	\xi_t^N\leq\xi_{T_\kappa}^N+\sqrt{\sigma}+\int_{T_\kappa}^t\Big(C\sigma^2+(2\rho+\gamma)(R_\rho)^2  \big(c'+c\big)+\left(\frac{d(\Lambda_+)^2\sigma^2}{N}-2\rho\right)\xi_s^{N}\Big)\dds\,,
\end{equation*}
where $\rho$, $c,c'$ can be chosen arbitrarily small, independently of $\sigma$.

\noindent{}{\bf Step 4.2.5.}  By Jensen's inequality $\big(\frac 1N\sum_{i=1}^N\|X^{i,N}_{T_\kappa}-\lambda_1\|^{2}\big)^{2n}\le \frac 1N\sum_{i=1}^N\|X^{i,N}_{T_\kappa}-\lambda_1\|^{4n}< \frac{\kappa^{4n}}{4}$, and so necessarily $\xi^N_{T_\kappa}\le \frac{\kappa^2}{2^{1/n}}$. Hence, for $\rho$ fixed, we shall choose $c,c'$ so that $(2\rho+\gamma)(R_\rho)^2(c+c')\le 2\rho 2^{-1/n}\kappa^2$ so that 
\begin{equation*}
	\xi_t^N\leq\xi_{T_\kappa}^N+\sqrt{\sigma}+\int_{T_\kappa}^tC\sigma^2+\frac{2\rho\kappa^2}{4^{1/n}}+\left(\frac{d(\Lambda_+)^2\sigma^2}{N}-2\rho\right)\xi_s^{N}\dds\,.
\end{equation*}
Choosing finally $\sigma$ small enough, we can impose that $\xi^N_t\le \frac{\kappa^2}{4^{1/n}}$ for all $t\in[T_\kappa,\exp[2H/\sigma^2]$,
which achieves to ensure that $S^{(1)}_\kappa\ge \exp[2H/\sigma^2]$ after a \emph{reductio ad absurdum}.

\noindent
This argument gives the claim $S_\kappa\overset{\mathbb P}{\underset{\sigma\downarrow 0}{\asymp}}\exp\big[\frac{2H}{\sigma^2}\big]$ of {\bf Step 4}.
\end{proof}

\subsection{Proof of Theorem \ref{thm:main2}}\label{ssec:ProofMain2}
\noindent
According to Lemma \ref{dale-bis} and Proposition~\ref{label_bizarre9}, and choosing $\varepsilon<\varepsilon_c$ with
$\varepsilon_c$ as in \eqref{ColliRad2}, and given $(X_{T_{\varepsilon,\kappa}}^{i,N},Y_{T_{\varepsilon,\kappa}}^{i,N})$, the
hitting-times
\[
\widehat{\beta}^i_{\lambda,\varepsilon,N}(\sigma)=\inf\left\{t\ge
T_{\varepsilon,\kappa}\,:\,(x^{i,\sigma}_{T_{\varepsilon,\kappa},t},y^{i,\sigma}_{T_{\varepsilon,\kappa},t})\in\mathbb{B}(\lambda;\varepsilon)\times
\mathbb{B}(\lambda;\varepsilon)\right\},\,1\le i\le N,
\]
all satisfy the Freidlin-Wentzell' type law with the exit-cost $H_\varepsilon(\lambda)$ as in \eqref{label_bizarre7} and the exit-property:
\[
\lim_{\sigma\rightarrow 0}\mathbb P^{i,N}_{T_{\varepsilon,\kappa};(x,y)}\left({\rm
dist}\Big(\big(x^{i,\sigma}_{T_{\varepsilon,\kappa},\widehat{\beta}^i_{\lambda,\varepsilon,N}(\sigma)},y^{i,\sigma}_{T_{\varepsilon,\kappa},\widehat{\beta}^i_{\lambda,\varepsilon,N}(\sigma)}\big),\mathbb
B(\lambda;\varepsilon)\times\mathbb B(\lambda;\varepsilon) \Big)\leq\delta\right)=1\,.
\]
for $\mathbb P^{i,N}_{T_{\varepsilon,\kappa};(x,y)}$ the conditional probability given
$\{(X_{T_{\varepsilon,\kappa}}^{i,N},Y_{T_{\varepsilon,\kappa}}^{i,N})=(x,y)\}$.
Following the same proof arguments as for Proposition \ref{torche}, we obtain

\begin{prop}
\label{tyrese}
For any $\varepsilon\in(0,\varepsilon_c)$ and $\lambda\in\mathbb R^d$, provided $N$ is large enough, we have: for any $\delta>0$ and
any $1\le i\le N$,

\begin{equation*}
\lim_{\sigma\to0}\PP\left(\exp\left[\frac{2}{\sigma^2}\left(H_{\varepsilon}(\lambda)-\delta\right)\right]<\widehat{\beta}^{i}_{\lambda,\varepsilon,N}(\sigma)<\exp\left[\frac{2}{\sigma^2}\left(H_{\varepsilon}(\lambda)+\delta\right)\right]\right)=1\,,
\end{equation*}
and
\[
\lim_{\sigma\rightarrow 0}\mathbb P\left({\rm
dist}\left(\Big(X^{i,N}_{\widehat{\beta}^{i}_{\lambda,\varepsilon,N}(\sigma)},Y^{i,N}_{\widehat{\beta}^{i}_{\lambda,\varepsilon,N}(\sigma)}\Big),\mathbb
B(\lambda;\varepsilon)\times\mathbb B(\lambda;\varepsilon)\right)<\delta \right)=1\,.
\]\end{prop}

From this, we can next derive the analogue of Proposition \ref{label_bizarre8}  which enables to conclude Theorem \ref{thm:main2}.

\begin{prop}
\label{ned} Let $H_\varepsilon$, $\underline H_\varepsilon$, $\mathcal M_\varepsilon$ and $\varepsilon_c$ be as in Proposition
\ref{label_bizarre8}. For any $\varepsilon\in(0,\varepsilon_c)$ and assuming that $N$ is large enough, it holds: for any $\delta>0$,
\begin{equation*}
\lim_{\sigma\to0}\PP\left(\exp\left[\frac{2}{\sigma^2}\left(\underline{H}_{\varepsilon}-\delta\right)\right]<\mathcal{C}^{i}_{\varepsilon,N}(\sigma)<
\exp\left[\frac{2}{\sigma^2}\left(\underline{H}_{\varepsilon}+\delta\right)\right]\right)=1\,.
\end{equation*}

Moreover, the collision-location persists near $\mathcal M_\varepsilon$ with: for any $\delta >0$, $1\le i\le N$,
\begin{equation*}
\lim_{\sigma\to0}\PP\left(\inf_{\lambda_\varepsilon\in\mathcal M_\varepsilon}\max\bigg({\rm
dist}\big(X^{i,N}_{C^{i}_{\varepsilon,N}(\sigma)},\mathbb B(\lambda_\varepsilon;\varepsilon)\big)\,;\,{\rm
dist}\big(Y^{i,N}_{C^{i}_{\varepsilon,N}(\sigma)},\mathbb B(\lambda_\varepsilon;\varepsilon)\big)\bigg)\ge \delta\right)=0\,.
\end{equation*}
\end{prop}

\paragraph{Acknowledgments}
For the first author, the paper was supported by the RSF project No. 24-11-00123. The second author acknowledges the support of the French ANR grant 'METANOLIN' (ANR-19-CE40-0009). 
The first author would like to express his gratitude to the members of the department of Mathematics of the university Jean-Monnet for their warm hospitality during the visits when this work has been developed.

\bibliographystyle{plain}
\bibliography{biblio_collision3}

\appendix 

\section{Proofs of Proposition \ref{prop:WellposedMV} and Lemma \ref{lem:Moment_particle}}
\label{Appendix-A}
\subsection{Proof of Proposition \ref{prop:WellposedMV}}
As a preliminary step, let us recollect  from \cite{HIP} the following classical growth control result.
\begin{lem}\label{lem:GrowthControl} Let $f:[0,\infty)\rightarrow [0,\infty)$ be a Borel function a.e. differentiable. Then, if there exists $M>0$ such that 
	$$
	\{t\ge 0\,:\,f(t)\ge M\}\subset \{t\ge 0\,:\,\frac{\dd f(t)}{\ddt}\le 0\},
	$$
	 we have $\max_{t\ge 0}f(t)\le \max\big\{M;f(0)\big\}$.
\end{lem}
\noindent
\textbf{Existence result:} To construct a solution to \eqref{MV1} (and likewise for \eqref{MV2}), let us introduce the following truncated approximations of $\U$ and $\F$: 
\begin{equation*}
\U_r(x):=\U(x)\mathds{1}_{\{\|x\|\le r\}}+\U\Big(r\frac x{\|x\|}\Big)\mathds{1}_{\{\|x\|>r\}} ,\qquad \F_r(x):=\F(x)\mathds{1}_{\{\|x\|\le r\}}+\F\Big(r\frac x{\|x\|}\Big)\mathds{1}_{\{\|x\|>r\}}, \ \ r>0.
\end{equation*} 
For each fixed $r>0$, $\U_r$ and $\F_r$ are both globally Lipschitz continuous, and $\lim_{r\rightarrow \infty}\U_r\rightarrow \U$ and $\lim_{r\rightarrow \infty}\F_r\rightarrow \F$ pointwise on any compact set of $\mathbb R^d$. Moreover, for $\F_r$, since the truncation preserves the anti-symmetry of $\F$ in \eqref{RotInv}, and, since $\F\Big(r\frac x{\|x\|}\Big)=\frac{x}{\|x\|}G(r)$ whenever $\|x\|>r$, the monotone property \eqref{Hyp-F2} of $\F$ is also preserved globally with: for any integer $p\ge 0$, and any $x,y\in\mathbb R^d$,
\begin{align}
&(\|x\|^px-\|y\|^py)\cdot \F_r(x-y)\nonumber\\
&=(\|x\|^px-\|y\|^py)\cdot \F(x-y)\mathds{1}_{\{\|x-y\|\le r\}}+(\|x\|^px-\|y\|^py)\cdot \F\bigg(r\frac{x-y}{\|x-y\|}\bigg)\mathds{1}_{\{\|x-y\|> r\}}\ge 0,
\end{align}
the final inequality following from the expression $\F\big(r\frac{x-y}{\|x-y\|}\big)=\frac{(x-y)}{\|x-y\|}G(r)$ and \eqref{Hyp-F2}. By extension, the property \eqref{Hyp-F2_bis} still holds true for $\F_r$ and gives
\begin{equation}\label{Hyp-F2-bis}
	2\iint \|x\|^n x\cdot \F_r(x-y)\mu(\ddx)\mu(\ddy)=\iint (\|x\|^nx -\|y\|^ny)\cdot \F_r(x-y)\mu(\ddx)\mu(\ddy)\ge 0.
\end{equation}
For $\U_r$, and taking  $r\ge \Rd'$, the confining property \eqref{Hyp-U3-bis} of $\U$ is immediately preserved for $\Rd'\le \|x\|\le r$ and, for $\|x\|\ge r$, one has
\begin{align}\label{Hyp-U3-truncated}
	-x\cdot \U_r(x)=-x\cdot \U\Big(r\frac x{\|x\|}\Big)=- \frac {\|x\|}r\times\Bigg( r\frac x{\|x\|}\cdot \U\Big(r\frac x{\|x\|}\Big)\Bigg)\le -\frac {\|x\| }r \cd(\Rd')r^2=-r\|x\| \cd(\Rd')\,.
\end{align}
Following these preparatory remarks on $\F_r$ and $\U_r$, consider now, for $r>0$ fixed, the McKean-Vlasov SDE: 
\begin{equation}\label{Truncated-MV}
X^r_t=x_1 -\int_0^t\Big(\U_r\left(X^r_s\right)+\int \F_r\left(X^r_s-x\right)\,\mu^{X^r}_s(\ddx)\Big)\dds+\sigma \int_0^t \Gamma(X^r_s)\,\dd B_s, \quad \mu^{X^r}_t=\text{Law}(X^r_t),\quad t\ge 0\,.
\end{equation}
Since $\U_r$, $\F_r$ and $\Gamma$ are globally Lipschitz continuous, \eqref{Truncated-MV} admits a unique pathwise solution (applying e.g. \cite[Theorem 2.2]{Meleard-96}). Additionally, for any $p\ge 1$, It\^o's formula yields
\begin{align*}
	\mathbb E[\|X^r_t\|^{2p}]&=x_1^{2p}-2p\int_0^t \mathbb E\Big[\|X^r_s\|^{2(p-1)}\big(X^r_s\cdot \U_r(X^r_s)\big)\Big]\,\dds\\
	&\quad-2p\int_0^t \mathbb E\Big[\|X^r_s\|^{2(p-1)}\big(X^r_s\cdot \int \F_r(X^r_s-x)\mu^{X^r}_s(\ddx)\big)\Big]\,\dds\\
	&\quad +\frac{\sigma^2}{2}  \int_0^t\mathbb E\Bigg[\|X^r_s\|^{2(p-2)} \Big(2p\mbox{Trace}(\Gamma\Gamma^*(X^r_s))\|X^r_s\|^{2} +4p(p-1) X^r_s\cdot \Gamma\Gamma^*(X^r_s)X^r_s\Big)\Bigg]\,\dds,
\end{align*} 
or equivalently
\begin{align}\label{ControlMom-1}
	\frac{\dd\mathbb E[\|X^r_t\|^{2p}]}{\ddt}&=-2p\mathbb E\Big[\|X^r_t\|^{2(p-1)}\big(X^r_t\cdot \U_r(X^r_t)\big)\Big]-2p\mathbb E\Big[\|X^r_t\|^{2(p-1)}\big(X^r_t\cdot \int \F_r(X^r_t-x)\mu^{X^r}(t,\ddx)\big)\Big]\nonumber\\
	&\quad +\frac{\sigma^2}{2} \mathbb E\Bigg[\|X^r_t\|^{2(p-2)}\Big( 2p\mbox{Trace}(\Gamma\Gamma^*(X^r_t))\|X^r_t\|^{2} +4p(p-1) X^r_t\cdot \Gamma\Gamma^*(X^r_t)X^r_t\Big)\Bigg]\,.
\end{align} 
According to \eqref{Hyp-F2-bis},
\begin{align*}
\mathbb E\Big[\|X^r_t\|^{2(p-1)}\Big(X^r_t\cdot \int \F_r(X^r_t-x)\mu^{X^r}_t(\ddx)\Big)\Big]=\iint x^{2(p-1)}x\cdot \F_r(x-y)\mu^{X^r}_t(\ddx) \mu^{X^r}_t(\ddy)\ge 0, 	
\end{align*}
and, owing to the boundedness of $\Gamma$ in ${\bf (A)}$-$(vi)$, 
$$
\mathbb E\Bigg[\|X^r_t\|^{2(p-2)}\Big(2p\mbox{Trace}\big(\Gamma\Gamma^*(X^r_t)\big)\|X^r_t\|^{2} +4p(p-1)X^r_t\cdot \Gamma\Gamma^*(X^r_t)X^r_t\Big)\Bigg]\le \big( 2pd+4p(p-1)\big) \big(\Lambda_+\big)^2\mathbb E\Big[\|
X^r_t\|^{2(p-1)}\Big].
$$
For the remaining elements in the r.h.s. of \eqref{ControlMom-1}, let now us assume that $r$ is larger than $\Rd'$. 
Using successively the local Lipschitz continuity and the confining property \eqref{Hyp-U3-bis} of $\U$, along the property \eqref{Hyp-U3-truncated}, it follows that
\begin{align*}
	-X^r_t\cdot \U_r(X^r_t)&=-X^r_t\cdot \U_r(X^r_t)\Big(\mathds{1}_{\{\|X^r_t\|<\Rd'\}}+\mathds{1}_{\{\Rd'< \|X^r_t\|< r\}}+\mathds{1}_{\{\|X^r_t\|\ge r\}}\Big)\\
	&\le \big(\cl(\Rd')\Rd'+\|\U(0)\|\big)\mathds{1}_{\{\|X^r_t\|< \Rd'\}}-\cd(\Rd')\|X^r_t\|^2\mathds{1}_{\{\Rd'\le \|X^r_t\|<r \}} -\cd(\Rd')r\|X^r_t\|\mathds{1}_{\{\|X^r_t\|\ge r \}}\\
	&\le\big(\cl(\Rd')\Rd'+\|\U(0)\|\big)\mathds{1}_{\{\|X^r_t\|<{\Rd'} \}}-\cd(\Rd')\Rd'\|X^r_t\|\mathds{1}_{\{\Rd'\le \|X^r_t\|\}}\\
	&\le\big(\cl(\Rd')\Rd'+\|\U(0)\|\big)-\cd(\Rd')\Rd'\|X^r_t\|\,,
\end{align*}
for $\cd(\Rd')$ is as in ${\bf (A)}$-$(i)$ and $\cl(\Rd')$ is the local Lipschitz constant (as in \eqref{LocalLip-U}):
$$
\cl(\Rd')=\sup_{x\neq y, x,y\in \mathbb{B}(0;\Rd')}\frac{\|U(x)-U(y)\|}{\|x-y\|}.
$$
Plugged into \eqref{ControlMom-1}, the above implies the estimate:
\begin{align*}
	&\frac{\dd\mathbb E[\|X^r_t\|^{2p}]}{\ddt}\\
	&\le \mathbb E\Big[2p\big(\cl(\Rd')\Rd'+\|\U(0)\|-\cd(\Rd')\Rd'\|X^r_t\|\big)\|X^r_t\|^{2(p-1)} +\frac{\sigma^2}{2}\big( 2pd+4p(p-1)\big)(\Lambda_+)^2\|X^r_t\|^{2(p-1)}\Big]\\
	&\quad =:\mathbb E\Big[\Big(a(p)-2p\cd(\Rd')\Rd'\|X^r_t\|\Big)\|X^r_t\|^{2(p-1)}\Big]\,.
\end{align*}
Now, taking {\color{black}$R'\ge \max\{1;\Rd'\}$} as the zero of $z\in[0,\infty)\mapsto a(p)-2p\cd(\Rd')\Rd' z$, we have  
$$
\Big(a(p)-2p\cd(\Rd'){\Rd'}\|X^r_t\|\Big)\le\Big(a(p)-2p\cd(\Rd'){\Rd'}\|X^r_t\|\Big)\mathds{1}_{\{\|X^r_t\|\le {R'}\}}\,,
$$ 
and so
\begin{align*}
	\frac{\dd\mathbb E[\|X^r_t\|^{2p}]}{\ddt}&\le   \mathbb E\Big[\Big(a(p)-2p\cd(\Rd')\Rd'\|X^r_t\|\Big)\|X^r_t\|^{2(p-1)}\mathds{1}_{\{\|X^r_t\|
	\le {R'} \}}\Big]\\
	&\le  \Big(a(p)-2p\cd(\Rd')\Rd'{R'}\Big)\|{R'}\|^{2(p-1)}.
\end{align*}
This estimate induces the control of all moments of $X^r$:
\begin{equation}\label{Truncated_Moment_SS}
\max_{t\in[0,T]}\mathbb E[\|X^r_t\|^{2p}]\le 2^{2p}\Bigg(\|x\|^{2p}+T\Big(a(p)-2p\cd(\Rd')\Rd'{R'}\Big)\|{R'}\|^{2(p-1)}\Bigg).
\end{equation}
 on any finite time interval $[0,T]$, uniformly in the truncation order $r$.  
 In particular, as $\U_r$ and $\F_r$ conserve the polynomial growths of $\U$ and $\F$ with respective order $2m$ and $2n$ (recall $\mathbf{(A)}$-$(i)$ and $\mathbf{(A)}$-$(iii)$), for all $s\le t\le T$:
 \begin{align*}
 \sup_r\mathbb E[\|X^r_t-X^r_s\|^{2p}]&\le C (t-s)^{2p}\sup_r\max_{v\in[0,T]}\mathbb E\Big[\|\U_r(X^r_v)\|^{2p}+\int\|\F_r(X^r_v-x)\|^{2p}\mu^{X^r}_v(\ddy)+\|\Gamma\Gamma^*(X^r_v)\|^{2p}\Big]\\
 &\le C (t-s)^{2p}\sup_r\mathbb  E\Big[1+\|X^r_v\|^{2p \max\{2m:2n\}}\Big]\le C (t-s)^{2p}.
 \end{align*}
 Thus the family of probabilities $\{\mathbb P\circ (X^r_t)_{t\in[0,T]}^{-1}\}_{r\ge 0}$ on the path space $\mathcal C([0,T];\mathbb R^d)$ are tight on $\mathcal P\big(\mathcal C([0,T];\mathbb R^d)\big)$ and, by Prohorov's theorem, admits a converging subsequence  $\{\mathbb P\circ (X^{r_k}_t)_{t\in[0,T]}^{-1}\}_{k\ge 0}$. Applying the Skorokhod representation theorem -- and up to the additional extraction of a clustering subsequence and up to an enlargement of the initial sample space $(\Omega,\mathcal F,\{\mathcal F_t\}_{t\ge 0},\mathbb P)$ --  $(X^r_t)_{t\in[0,T]}$ converges almost surely to some continuous process $(X^\infty_t)_{t\in[0,T]}$. As $$
 \max_{0\le t\le T}\mathbb E\big[\|X^\infty_t\|^{2p}\big]\le \liminf_{r\rightarrow \infty}\max_{0\le t\le T}\mathbb E\big[\|X^r_t\|^{2p}\big]<\infty,
 $$
  by continuity of $\U$, $\F$ and $\Gamma$, $X^\infty$ satisfies 
 \begin{equation*}
 	X^\infty_t=x_1-\int_0^t\Big(\U\left(X^\infty_s\right)+\int \F\left(X^\infty_s-x\right)\,\mu^{X^\infty}_s(\ddx)\Big)\dds+\sigma \int_0^t \Gamma(X^\infty_s)\,\dd B_s, \quad \mu^{X^\infty}_t=\text{Law}(X^\infty_t)\,,\,0\le t\le T.
 \end{equation*}
 The time-horizon $T$ being arbitrary, the existence of a solution to \eqref{MV1} is ensured on the time interval $[0,\infty)$.

\noindent
\textbf{Uniform-in-time moment controls:} The moment control \eqref{Truncated_Moment_SS} is stable at the limit $r\rightarrow \infty$ but  only provides local-in-time moment estimates for $X$. To establish uniform in time estimates for $\mathbb E[\|X_t\|^{2p}]$, we shall replicate the previous arguments having led to \eqref{Truncated_Moment_SS} starting by applying again It\^o's formula to get: for all $p\ge 1$,
\begin{align*}
	\frac{\dd\mathbb E[\|X_t\|^{2p}]}{\ddt}&= -2p\mathbb E\Big[\|X_t\|^{2(p-1)}X_t\cdot\Big( \U(X_t)+ \int \F(X_t-x)\mu^{X}_t(\ddx)\Big)\Big]\\
	&\quad+\frac{\sigma^2}{2}\mathbb E\Big[\|X_t\|^{2(p-2)}\Big(2p\mbox{Trace}(\Gamma\Gamma^*)(X_t)+4p(p-1)X_t\cdot \Gamma\Gamma^*(X_t)X_t\Big)\Big].
\end{align*}
Owing (again) to the property \eqref{Hyp-F2_bis},
$$
\mathbb E\Big[\|X_t\|^{2(p-1)}X_t\cdot\int \F(X_t-x)\mu^{X}_t(\ddx)\Big]=\iint \|x\|^{2(p-1)}xF(x-y)\mu^{X}_t(\ddx)\mu^X_t(\ddy)\ge 0
$$
and, using the boundedness of $\Gamma$, it follows that
\[
\frac{\dd\mathbb E[\|X_t\|^{2p}]}{\ddt}\le -2p\mathbb E\Big[\|X_t\|^{2(p-1)} X_t\cdot\U(X_t)\Big]+\frac{\sigma^2}{2}( 2p+4p(p-1))d \big(\Lambda_+\big)^2\mathbb E\Big[\|X_t\|^{2(p-1)}\Big].
\]
Observe next that, the confining property \eqref{Hyp-U3-bis} yields
\begin{align*}
	-X_t\cdot \U(X_t)&=-X_t\cdot\U(X_t)\Big(\mathds{1}_{\{\|X_t\|\le \Rd'\}}+\mathds{1}_{\{\|X_t\|> \Rd'\}}\Big)\\
	&\le \Big(\cl(\Rd')\Rd'+\|\U(0)\|\Big)\mathds{1}_{\{\|X_t\|\le \Rd'\}}-\cd(\Rd')\|X_t\|^2\mathds{1}_{\{\|X_t\|>\Rd'\}}\\
	&\le 	\Big(\cl(\Rd')\Rd'+\|\U(0)\|+\cd(\Rd')\Rd'^2\Big)\mathds{1}_{\{\|X_t\|\le \Rd'\}}-\cd(\Rd')\|X_t\|^2=: C(\Rd')-\cd(\Rd')\|X_t\|^2.
\end{align*}
Hence, we have
\begin{align}\label{Moment_SS_Iteration}
\frac{\dd\mathbb E[\|X_t\|^{2p}]}{\ddt}&\le 2p\mathbb E\Big[\|X_t\|^{2(p-1)}\big(C(\Rd')-\cd(\Rd')\|X_t\|^{2}\big)\Big]+(\sigma\Lambda_+)^2d(p+2p(p-1)) \mathbb E\Big[\|X_t\|^{2(p-1)}\Big]\nonumber\\
&\le b(p) \mathbb E\big[\|X_t\|^{2(p-1)}\big]-2p\cd(\Rd')\mathbb E\big[\|X_t\|^{2p}\big],
\end{align}
for 
$$
b(p):=2pC(\Rd')+(\sigma\Lambda_+)^2d(p+2p(p-1)).
$$
From \eqref{Moment_SS_Iteration}, and by exploiting Lemma \ref{lem:GrowthControl}, we can now iteratively draw the control of $\sup_{t\ge 0}\mathbb E[\|X_t\|^{2p}]$. For $p=1$, defining $z_1$ the first zero of the mapping $z\in[0,\infty)\mapsto b(1)-2\cd(\Rd')z$, we get $z_1=(\cd(\Rd'))^{-1}\Big(C(\Rd')+\big(\sigma\Lambda_+\big)^2d\Big)$ and so by Lemma \ref{lem:GrowthControl}, we obtain
\[
\sup_{t\ge 0}\mathbb E[\|X_t\|^{2}]\le \max\left\{\|x_1\|^{2};\,z_1\right\}=:\zeta_1.
\]
For $p\ge 2$, assuming  $\sup_{t\ge 0}\mathbb E[\|X_t\|^{2p}]\le \zeta_{p-1}$, define
 $$
 z_p:=\inf\left\{z\ge 0\,:\,b(p)\zeta_{p-1}-2p\cd(\Rd')z\le 0\right\}
 $$
 or equivalently
 $$ z_p=(2p\cd(\Rd'))^{-1}b(p)\zeta_{p-1}=(\cd(\Rd'))^{-1}\Big(C(\Rd')+(\sigma\Lambda_+)^2\big(\frac d{2}+p-1\big)\Big)\zeta_{p-1}\,,
 $$ 
 Lemma \ref{lem:GrowthControl} implies that
 \[
 \sup_{t\ge 0}\mathbb E[\|X_t\|^{2p}]\le \max\{\|x_1\|^{2p};z_p\}=:\zeta_p.
 \]
 Setting $c=3(\cd(\Rd'))^{-1}\max\left\{C(\Rd');(\sigma\Lambda_+)^2\frac d{2}\right\}$ so that
 $$
 (\cd(\Rd'))^{-1}\Big(C(\Rd')+(\sigma\Lambda_+)^2\big(\frac d{2}+p-1\big)\Big)\le c(p-1)
 $$
 enables to exhibit the upper-bounds
 \begin{align*}
 \zeta_p&\le \max\left\{\|x_1\|^{2p};c(p-1)\max\left\{\|x_1\|^{2(p-1)};z_{p-1}\right\}\right\}=\max\left\{\|x_1\|^{2p};c(p-1)\|x_1\|^{2(p-1)};c(p-1)z_{p-1}\right\}\\
 &\le \cdots \le \max\left\{\max_{k=1,\cdots p-1}\left\{c^{k}\|x_1\|^{2(p-k)}\frac{(p-1)!}{(p-1-k)!}\right\};\|x_1\|^{2p}; (p-1)!c^pz_1\right\}.
 \end{align*}
 The estimate \eqref{Moment_SS} eventually follows with $C=c\max\left\{\|x\|^{2};z_1;1\right\}$.

\medskip

\noindent
\textbf{Uniqueness:}  Owing to the synchronization assumption $\textbf{(A)}$-$(iv)$ on the effective force $\W_\mu$, pathwise uniqueness for \eqref{MV1} follows from standard comparison arguments. For the sake of brevity, these arguments are succinctly detailed hereafter: assuming $Z$ is an another solution to \eqref{MV1} driven by the same Brownian motion $B$, so that $Z$ satisfies
\begin{equation}\label{SecondSol}
	Z_t=x_1-\int_0^t\Big(\U\left(Z_s\right)+\int \F\left(Z_s-x\right)\,\mu^{Z}_s(\ddx)\Big)\dds+\sigma \int_0^t \Gamma(Z_s)\,\dd B_s, \quad \mu^{Z}_t=\text{Law}(Z_t)\,,\, t\ge 0,
\end{equation}
and $\max_{0\le t\le T}\mathbb E[\|Z_t\|^{2n}]<\infty$.

Then, we have
\begin{align*}
	&\frac{\dd\mathbb E[\|X_t-Z_t\|^2]}{\ddt}=-\mathbb E\Big[(X_t-Z_t)\cdot \Big(\U(X_t)-\U(Z_t)\Big)\Big]-\mathbb E\Big[(X_t-Z_t)\cdot \Big(\F(X_t-\tilde X_t)-\F(Z_t-\tilde Z_t)\Big)\Big]\\
	&\qquad\qquad\qquad\qquad+\sigma^2\mathbb E\Big[\mbox{Trace}\big( \Gamma(X_s)-\Gamma(Z_s)\big)\big( \Gamma(X_s)-\Gamma(Z_s)\big)^*\Big]\\
	&=-\mathbb E\Big[(X_t-Z_t)\cdot \Big(\W_{\mu^X_t}(X_t)-\W_{\mu^Z_t}(Z_t)\Big)\Big]
	+\sigma^2\mathbb E\Big[\mbox{Trace}\big( \Gamma(X_s)-\Gamma(Z_s)\big)\big( \Gamma(X_s)-\Gamma(Z_s)\big)^*\Big]\\
	&\le -\mathbb E\Big[(X_t-Z_t)\cdot \Big(\W_{\mu^X_t}(Z_t)-\W_{\mu^Z_t}(Z_t)\Big)\Big]
	+\sigma^2\mathbb E\Big[\mbox{Trace}\big( \Gamma(X_s)-\Gamma(Z_s)\big)\big( \Gamma(X_s)-\Gamma(Z_s)\big)^*\Big]\\
	&\le \big(\frac 1{2}+\sigma^2\KD\big)\mathbb E\Big[\|X_t-Z_t\|^2\Big]+\frac 1{2}\mathbb E\Big[\| \W_{\mu^X_t}(Z_t)-\W_{\mu^Z_t}(Z_t)\|^2\Big].
\end{align*}
for $\KD$ the Lipschitz coefficient of $\Gamma$. 
According to \eqref{Hyp-F1}, for any independent copies $\tilde X$ and $\tilde Z$ of $X$ and $Z$ defined on an independent filtered space $(\tilde \Omega,\tilde{\mathcal F},\{\tilde{\mathcal F}_t\}_{t\ge 0},\tilde{\mathbb P})$, we have
\begin{align*}
	\|\W_{\mu^X_t}(Z_t)-\W_{\mu^Z_t}(Z_t)\|&=\|\int \F(Z_t-x)\mu^X_t(\ddx)-\int \F(Z_t-z)\mu^Z_t(\ddz)\|\\
	&\le C\tilde{\mathbb E} \Big[\|\tilde X_t-\tilde Z_t\|\Big(1+\|Z_t\|^{2n}+\|\tilde X_t\|^{2n}+\|\tilde Z_t\|^{2n}\Big)\Big]\\
	&\le C\Big(\tilde{\mathbb E} \Big[\|\tilde X_t-\tilde Z_t\|\Big]\Big)^{1/2}\Big(\tilde{\mathbb E} \Big[1+\|Z_t\|^{2n}+\|\tilde X_t\|^{2n}+\|\tilde Z_t\|^{2n}\Big]\Big)^{1/2}
\end{align*}
Taking the second moment of the above inequality, and owing to the moment control of $X$ and $Z$ (and by extension of $\tilde X$ and $\tilde Z$), it follows that
$$
\mathbb E\left[\|\W_{\mu^X_t}(Z_t)-\W_{\mu^Z_t}(Z_t)\|^2\right]\le C \mathbb E\left[\|\tilde X_t-\tilde Z_t\|^2\right]=C\mathbb E\left[\|X_t-Z_t\|^2\right]
$$ 
Hence, for $0<T<\infty$ arbitrary,
\begin{align*}
	\frac{\dd\mathbb E[\|X_t-Z_t\|^2]}{\ddt}\le \big(\frac 1{2}+\sigma^2\KD+C\big)\mathbb E[\|X_t-Z_t\|^2]
\end{align*}
and, by Gronwall's inequality, this yields $\max_{t\le T}\mathbb E[\|X_t-Z_t\|^2]=0$. Consequently $\{X_t\}_{0\le t\le T}$ and $\{Z_t\}_{0\le t\le T}$ have the same one-time marginal laws reducing \eqref{SecondSol} to the ``linear'' equation:
$$
Z_t=x_1-\int_0^t\Big(\U\left(Z_s\right)+\int \F\left(Z_s-x\right)\,\mu^{X}_s(\ddx)\Big)\dds+\sigma \int_0^t \Gamma(Z_s)\,\dd B_s, \quad \mu^{Z}_t=\text{Law}(Z_t)\,,\, 0\le t\le T.
$$
Owing to the local Lipschitz continuity of $\U$ and of $\F$ (and using, again, \cite[Theorem 2.5, Chapter 5]{KarShr91}), necessarily $X=Z$ a.s..
\subsection{Proof of Lemma \ref{lem:Moment_particle}}
The proof is analogue to the moment estimate in Proposition \ref{prop:WellposedMV}: For all $p\ge 1$, and $t\ge 0$, we have
\begin{align*}
\frac{\dd\mathbb E[\|X^{i,N}_t\|^{2p}]}{\ddt}&= -2p\mathbb E\Big[\|X^{i,N}_t\|^{2(p-1)}X^{i,N}_t\cdot\Big( \U(X^{i,N}_t)+ \frac 1N\sum_{j=1}^N \F(X^{i,N}_t-X^{j,N}_t)\Big)\Big]\\
&\quad+\frac{\sigma^2}{2} 4p(p-1)\mathbb E\Big[\|X^{i,N}_t\|^{2(p-2)}\big(X^{i,N}_t\cdot \Gamma\Gamma^*(X^{i,N}_t)X^{i,N}_t\big)\Big].
\end{align*}
The $X^{i,N}$ being exchangeable, the index $i$ can be arbitrarily chosen between $1$ and $N$. This exchangeability together with \eqref{Hyp-F2} brings
\begin{align}\label{Monotony}
&\frac 1N\sum_{j=1}^N\mathbb E\Big[\|X^{i,N}_t\|^{2(p-1)}X^{i,N}_t\cdot \F(X^{i,N}_t-X^{j,N}_t)\Big]=\frac 1{N^2}\sum_{i,j=1}^N\mathbb E\Big[\|X^{i,N}_t\|^{2(p-1)}X^{i,N}_t\cdot \F(X^{i,N}_t-X^{j,N}_t)\Big]\nonumber\\
&=\frac 1{2N^2}\sum_{i,j=1}^N\mathbb E\Big[\|X^{i,N}_t\|^{2(p-1)}X^{i,N}_t\cdot \F(X^{i,N}_t-X^{j,N}_t)\Big]+\frac 1{2N^2}\sum_{i,j=1}^N\mathbb E\Big[\|X^{j,N}_t\|^{2(p-1)}X^{j,N}_t\cdot \F(X^{j,N}_t-X^{i,N}_t)\Big]\nonumber\\
&=\frac 1{2N^2}\sum_{i,j=1}^N\mathbb E\Bigg[\Big(\|X^{i,N}_t\|^{2(p-1)}X^{i,N}_t-\|X^{j,N}_t\|^{2(p-1)}X^{j,N}_t\Big)\cdot \F(X^{i,N}_t-X^{j,N}_t)\Bigg]\ge 0.
\end{align}
Hence, using further the boundedness of $\Gamma$, we have
\begin{align*}
\frac{\dd\mathbb E[\|X^{i,N}_t\|^{2p}]}{\ddt}\le -2p\mathbb E\Big[\|X^{i,N}_t\|^{2(p-1)} X^{i,N}_t\cdot \U(X^{i,N}_t)\Big]+\frac{\sigma^2}{2}(4p(p-1))\big(\Lambda_+\big)^2d\mathbb E\Big[\|X^{i,N}_t\|^{2(p-1)}\Big].
\end{align*}
Invoking \eqref{Hyp-U3-bis},
\begin{align*}
	-X^{i,N}_t\cdot\U(X^{i,N}_t)&=-X^{i,N}_t\cdot \U(X^{i,N}_t)\Big(\mathds{1}_{\{\|X^{i,N}_t\|\le \Rd'\}}+\mathds{1}_{\{\|X^{i,N}_t\|> \Rd'\}}\Big)\\
	&\le \Big(\cl(\Rd')\Rd'+\|\U(0)\|\Big)\mathds{1}_{\{\|X^{i,N}_t\|\le \Rd'\}}-\cd(\Rd')\|X^{i,N}_t\|^2\mathds{1}_{\{\|X^{i,N}_t\|>\Rd'\}}\\
	&\le 	\Big(\cl(\Rd')\Rd'+\|\U(0)\|+\cd(\Rd')\Rd'^2\Big)\mathds{1}_{\{\|X^{i,N}_t\|\le \Rd'\}}-\cd(\Rd')\|X^{i,N}_t\|^2 
\end{align*}
and so, for $C(\Rd'):=\cl(\Rd')\Rd'+\|\U(0)\|+\cd(\Rd')\Rd'^2$, 
\begin{align*}
	\frac{\dd\mathbb E[\|X^{i,N}_t\|^{2p}]}{\ddt}\le  -2p\cd(\Rd')\mathbb E\Big[\|X^{i,N}_t\|^{2p}\Big]+\Big(C(\Rd')+\frac{\sigma^2}{2}\Big(2p+4p(p-1)\Big)\big(\Lambda_+\big)^2d\Big)\mathbb E\Big[\|X^{i,N}_t\|^{2(p-1)}\Big].
\end{align*}
We recover here the same inequality as in the intermediate estimate \eqref{Moment_SS_Iteration} in the proof of Proposition \ref{prop:WellposedMV}. This finally gives Lemma \ref{lem:Moment_particle}.

 \section{One-dimensional case and related exact first-collision estimates}\label{sec:1DCase}
 As we briefly mentioned at the beginning of the paper,  in the one-dimensional case, the exact first collision-times, respectively for \eqref{MV1}-\eqref{MV2} and \eqref{MFSP1}-\eqref{MFSP2},
 $$
 C(\sigma):=\inf\left\{t\ge 0\,:\,X_t=Y_t\right\}\quad\mbox{ and } \quad C^i_N(\sigma):=\inf\left\{t\ge 0\,:\,X^{i,N}_t=Y^{i,N}_t\right\}\,,
 $$
are well-defined and can be studied along the exact collision-locations:
 $$
 L(\sigma):=X_{C(\sigma)}(=Y_{C(\sigma)})\quad\mbox{ and } \quad L^i_N(\sigma):=X^{i,N}_{C^{i}_{N}(\sigma)}\Big(=Y^{i,N}_{C^{i}_{N}(\sigma)}\Big)\,.
 $$
(Although, here, collision-location naturally reduces to a single-component site instead of a two-components one, for the sake of convenience, we keep the notation $L(\sigma)$.)

In spite of the dimension reduction, establishing the zero-noise limits of $C(\sigma)$ and $L(\sigma)$ carry the same technical issues discussed in Section \ref{sec:mainresults} (with a re-interpretation of the collision-times,  similarly to \eqref{KeyInterpretation}, and coupling argument 
 a coupling argument relating $(X,Y)$ and
	$(X^{i,N},Y^{i,N})$~--~with the {\it linearized} diffusion $(x^\sigma,y^\sigma)$ defined in \eqref{Linearized-MV1}~--~\eqref{Linearized-MV2}. However the proof arguments exhibited in Sections~\ref{sec:LinearCase}, \ref{sec:SelfStabilizingCase} and
	\ref{sec:ParticleCase} simplify significantly. Coupling lemmas are notably less significant and the design of suitable enlargements of collision set introduced
	in Section~\ref{subsec:Linear-collisionA} are not necessary for \emph{a priori} spotting the first collision-location.
	
	\noindent
	Let us focus on the case of the self-stabilizing diffusions and the derivation of the equivalent of Theorem \ref{thm:main1}. Additionally, and without loss of generality, let us assume that the two asymptotically stable points of $\U$ given by $(\mathbf A)$-$(ii)$ are ordering as $\lambda_1\le \lambda_2$. By~$(\mathbf A)$-$(v)$,
	$x_1<x_2$, and, necessarily,  $C(\sigma)$ simply reduces to the first time~$t$ where~$X_t\ge Y_t$. Therefore, we have
	\begin{equation}\label{buffer_collision1d}
		C(\sigma)=\inf_{z\in\mathbb R}C_{z,z}(\sigma),\qquad C_{z_1,z_2}(\sigma):=\inf\left\{t\ge 0\,\,:\,\,X_t\ge z_1,\,Y_t\le z_2\right\}\,,z_1,z_2\in\mathbb R\,.
	\end{equation}
In the reciprocal case,
	$\lambda_1>\lambda_2$, a similar  interpretation holds, the ordering between $X$ and~$Y$ being simply reversed. As it will become obvious in a few lines (with the proof of Theorem \ref{victoire2} below), only the situation  ``$z_1>\lambda_1$ and $z_2<\lambda_2$'' is relevant (namely collision occurs within the interval $[\lambda_1,\lambda_2]$), the alternatives, ``$z_1\leq\lambda_1$'' or ``$z_2\geq\lambda_2$'', having (intuitively and rigorously) no particular interest.
	
	By analogy with Section \ref{sec:LinearCase}, define the domains $\mathcal{D}_{z_1}^1:=[z_1,+\infty)$ and
	$\mathcal{D}_{z_2}^2:=(-\infty,z_2]$, so that $C_{z_1,z_2}(\sigma)$ writes as
	\[
	C_{z_1,z_2}(\sigma)=\inf\left\{t\geq0\,:\,(X_t,Y_t)\notin\mathbb R^2\setminus\big(\mathcal{D}_{z_1}^1\times\mathcal{D}_{z_2}^2\big)=\left(\big(-\infty,z_1\big)\times\big(-\infty,z_2\big]\right)\cup\left(\big[z_1,\infty\big)\times\big(z_2,\infty\big)\right)\right\}\,.
	\]
	For $i=1,2$, the domain $\mathbb R\setminus\mathcal{D}_{z_i}^i$ is stable by $-\W_{\lambda_i}$ and Freidlin and Wentzell's estimate (Theorem~\ref{thm:KramersDZ}) then applies for the exit-time
	\[
	c_{z_1,z_2}(\sigma):=\inf\left\{t\geq0\,:\,(x_t^\sigma,y_t^\sigma)\notin\big(\mathbb R\times\mathbb
	R\big)\setminus\big(\mathcal{D}_{z_1}^1\times\mathcal{D}_{z_2}^2\big)\right\}\,,
	\]
	for
	\begin{equation}\label{Linear1D}
		x_t^\sigma=x_1-\int_0^t \W_{\lambda_1}(x_s^\sigma)\,\dds+\sigma \int_0^t \Gamma(x_s^\sigma)\,\dd B_s,\quad y_t^\sigma=x_1-\int_0^t \W_{\lambda_2}(y_s^\sigma)\,\dds+\sigma \int_0^t \Gamma(y_s^\sigma)\,\dd\widetilde B_s,\qquad t\geq0.
	\end{equation}
	and with the exit-cost related to $c_{z_1,z_2}(\sigma)$:
	\[
	\widetilde{H}_0(z_1,z_2):=\inf_{-\infty<T_1\leq T_2<\infty}\inf_{\Phi}\Big\{I(T_1,T_2,\Phi)\;:\;\phi\in \mathcal A^{(\lambda_1,\lambda_2),(z_1,z_2)}([T_1,T_2])\Big\}. 
	\]
Note that with $\widetilde H_0(\lambda,\lambda)$ we recover $H_0$ defined in \eqref{collisioncost}.
	Following the proof arguments of Proposition \ref{torche}, the $\sigma\downarrow 0$-asymptotic of $c_{z_1,z_2}(\sigma)$ and $(x^{\sigma}_{c_{z_1,z_2}(\sigma)},y^\sigma_{c_{z_1,z_2}(\sigma)})$ transfers to~$C_{z_1,z_2}(\sigma)$ and $(X_{C_{z_1,z_2}(\sigma)},Y_{C_{z_1,z_2}(\sigma)})$,
	yielding to the following result.
	\begin{lem}
		\label{prison} For any $z_1\geq\lambda_1$ and $z_2\leq\lambda_2$ and for any $\delta>0$, it holds:
		\begin{equation}\label{alberto2}
			\lim_{\sigma\rightarrow 0}
			\mathbb P\left(
			\exp\Big[
			\frac{2}{\sigma^2}(\widetilde H_0(z_1,z_2)-\delta)\Big]
			<C_{z_1,z_2}(\sigma)<
			\exp\Big[
			\frac{2}{\sigma^2}(\widetilde H_0(z_1,z_2)+\delta) \Big]
			\right)=1\,,
		\end{equation}
		and
		\[
		\lim_{\sigma\rightarrow 0}\mathbb P\left(\max\bigg(|X_{C_{z_1,z_2}(\sigma)}-z_1|,|Y_{C_{z_1,z_2}(\sigma)}-z_2|\bigg)\le
		\delta\right)=1\,.
		\]
	\end{lem}
	\noindent
	From this preparatory lemma, we can derive the analogue of Theorem \ref{thm:main1}.
	\begin{thm}
		\label{victoire2} Let $\underline H_0$ and $\mathcal M_0$ denote respectively the minimum and minimizer set of $H_0(\lambda)=\widetilde H_0(\lambda,\lambda)$.
		 Then,
		for any~$\delta>0$,
		\begin{equation}
			\label{eq:victoire2}
			\begin{aligned}
				&\lim_{\sigma\to0}\PP\left(\exp\left[\frac{2}{\sigma^2}\left(\underline H_0-\delta\right)\right]<C(\sigma)<\exp\left[\frac{2}{\sigma^2}\left(
				\underline H_0+\delta\right)\right]\right)=1,
			\end{aligned}
		\end{equation}
		and
		\begin{equation}\label{eq:victoirebis2}
			\lim_{\sigma\to0}\PP\left(\mathrm{dist}\big(X_{C(\sigma)},\mathcal M_0\big)\le \delta\right)=1.
		\end{equation}
\end{thm}
	\begin{proof}[Proof of Theorem \ref{victoire2}]
		Let us first establish the upper-tail estimate in \eqref{eq:victoire2}, starting with the inequality, for any $\lambda_0\in\mathcal M_0$,
		\begin{equation*}
			\PP\left(C(\sigma)\geq\exp\Big[\frac{2}{\sigma^2}\left(H_0(\lambda_0)+\delta\right)\Big]\right)\le
			\PP\left(C_{\lambda_0,\lambda_0}(\sigma)\geq\exp\Big[\frac{2}{\sigma^2}\left(H_0(\lambda_0)+\delta\right)\Big]\right)\,,
		\end{equation*}
		with $C_{z_1,z_2}(\sigma)$ being defined as in \eqref{buffer_collision1d} and so $C(\sigma)\le C_{\lambda_0,\lambda_0}(\sigma)$.
		 
\noindent{}Given the ordering $\lambda_1\le \lambda_2$,  $z\mapsto \widetilde{H}_0(z,z)=H_0(z)$ is decreasing on $(-\infty,\lambda_1]$, increasing on
		$[\lambda_2,+\infty)$, and thus the minimizer of the function $H_0$ is necessarily achieved in the interval $(\lambda_1,\lambda_2)$. Moreover, as~$x_1<x_2$, we know that the
		first time the diffusion~$(X,Y)$ reaches a point $(\lambda_0',\lambda_0')$ with $|\lambda_0'-\lambda_0|<\rho$ necessarily occurs before the time $(X,Y)$ enters the
		region~$[\lambda_0+\rho,\infty)\times(-\infty,\lambda_0-\rho]$ for any $\rho>0$ sufficiently small. Therefore,
		\begin{align*}
			\PP\left(C_{\lambda_0',\lambda_0'}(\sigma)\geq\exp\Big[\frac{2}{\sigma^2}\left(H_0(\lambda_0)+\delta\right)\Big]\right)\le
			\PP\left(C_{\lambda_0+\rho,\lambda_0-\rho}(\sigma)\ge\exp\Big[\frac{2}{\sigma^2}\left(H_0(\lambda_0)+\delta\right)\Big]\right)\,.
		\end{align*}
		As $\widetilde{H}_0$ is continuous, $\widetilde{H}_0(\lambda_0+\rho,\lambda_0-\rho)$ converges to $\widetilde{H}_0(\lambda_0,\lambda_0)=H_0(\lambda_0)$ as $\rho\downarrow 0$, and taking $\rho$ sufficiently
		small so that~$\widetilde{H}_0(\lambda_0+\rho,\lambda_0-\rho)\le H_0(\lambda_0)+\frac{\delta}{2}$, we get, by \eqref{alberto2} in Lemma~\ref{prison}, 
		\begin{equation}\label{proof_victoire2_1}
			\lim_{\sigma\to0}\PP\left(C(\sigma)\geq\exp\Big[\frac{2}{\sigma^2}\left(\underline H_0+\delta\right)\Big]\right)\le\lim_{\sigma\to0}\PP\left(C_{\lambda_0+\rho,\lambda_0-\rho}(\sigma)\ge\exp\Big[\frac{2}{\sigma^2}\left(\underline H_0+\delta\right)\Big]\right)=0\,.
		\end{equation}
				For the collision-location estimate \eqref{eq:victoirebis2}, it is enough to check that {\bf Step a}: the collision does not persist outside~$[\lambda_1,\lambda_2]$  and {\bf Step b}: derive, from a compactness argument, that the collision necessarily occurs 
		within~a vicinity of $[\lambda_1,\lambda_2]\cap \mathcal M_0$, i.e. for $\rho$ very small, 
			\begin{equation}\label{Collision1D-bis}
			\lim_{\sigma\rightarrow
				0}\mathbb{P}\Big(\forall \lambda_0\in\mathcal M_0,\,|L_{C(\sigma)}-\lambda_0|>\rho\big)=0.
		\end{equation}
		
		\noindent{}{\bf Step a.} Start with the following inequality:
		\begin{align*}
			\PP\left(X_{C(\sigma)}\notin[\lambda_1,\lambda_2]\right)&\leq\PP\left(X_{C(\sigma)}\notin[\lambda_1,\lambda_2],C(\sigma)\leq\exp\Big[\frac{2}{\sigma^2}\left(\underline H_0+\delta\right)\Big]\right)\\
			&\quad+\PP\left(C(\sigma)\geq\exp\Big[\frac{2}{\sigma^2}\left(\underline H_0+\delta\right)\Big]\right)\,,
		\end{align*}
		for some $\delta>0$. Using \eqref{proof_victoire2_1}, the second upper-bound in the r.h.s. tends to $0$ as $\sigma\downarrow 0$. For the remainder, we have
		\begin{align*}
			&\PP\left(X_{C(\sigma)}\notin[\lambda_1,\lambda_2],C(\sigma)\leq\exp\Big[\frac{2}{\sigma^2}\left(\underline H_0+\delta\right)\Big]\right)\\
			&\leq\PP\left(C_{\lambda_1,\lambda_1}(\sigma)\leq\exp\Big[\frac{2}{\sigma^2}\left(\underline H_0+\delta\right)\Big]\right)+\PP\left(C_{\lambda_2,\lambda_2}(\sigma)\leq\exp\Big[\frac{2}{\sigma^2}\left(\underline H_0+\delta\right)\Big]\right)\,.
		\end{align*}
		Choosing $\delta$ small enough so that
		$\min\left\{\widetilde{H}_0(\lambda_1,\lambda_1);\widetilde{H}_0(\lambda_2,\lambda_2)\right\}>\underline H_0+\delta$, Freidlin-Wentzell estimate in Lemma \ref{prison} immediately yields
			\[
		\lim_{\sigma\rightarrow
			0}\PP\left(C_{\lambda_k,\lambda_k}(\sigma)\leq\exp\Big[\frac{2}{\sigma^2}\left(\underline H_0+\delta\right)\Big]\right)\le \lim_{\sigma\rightarrow
			0}\PP\left(C_{\lambda_k,\lambda_k}(\sigma)<\exp\Big[\frac{2}{\sigma^2}\widetilde  H_0(\lambda_k,\lambda_k)\Big]\right)=0,\qquad k=1,2.
		\]
This further immediately gives $\lim_{\sigma\to0}\PP\left(L_{C(\sigma)}\notin[\lambda_1,\lambda_2]\right)=0$.
		
		\noindent{}{\bf Step b}. Given {\bf Step a} and, for $\lambda^-_0:=\min \mathcal M_0$, $\lambda^+_0;=\max \mathcal M_0$, it remains to show that $L_{C(\sigma)}$ does not persist
		in the region~$[\lambda_1,\lambda_2]\setminus (\lambda^-_0-\rho,\lambda^+_0+\rho)=[\lambda_1,\lambda^-_0-\rho]\cup[\lambda^+_0+\rho,\lambda_2]$. It is
		further sufficient to check this assertion for the one interval part, say $[\lambda^+_0+\rho,\lambda_2]$,
		the reasoning for the remainder being similar. As $[\lambda^+_0+\rho,\lambda_2]$ is compact, we can embed it into the finite union $\cup_{k=1}^L(\eta_k-r,\eta_k+r)$, where $r>0$ will be chosen - sufficiently small - later on. For each sub-interval $(\eta_k-r,\eta_k+r)$, observe that
		\begin{align*}
			&\PP\left(L_{C(\sigma)}\in[\eta_k-r;\eta_k+r]\right)\le\PP\left(C(\sigma)\geq\exp\Big[\frac{2}{\sigma^2}\left(\underline H_0+\delta\right)\Big]\right)\\
			&\quad+\PP\left(L_{C(\sigma)}\in[\eta_k-r;\eta_k+r],C(\sigma)\le\exp\Big[\frac{2}{\sigma^2}\left(\underline H_0+\delta\right)\Big]\right)=:I_1(\sigma)+I_2(\sigma)\,.
		\end{align*}
		According to \eqref{proof_victoire2_1},
		\begin{equation*}
			\lim_{\sigma\rightarrow 0}I_1(\sigma)=\lim_{\sigma\rightarrow
				0}\PP\left(C(\sigma)\geq\exp\Big[\frac{2}{\sigma^2}\left(\underline H_0+\delta\right)\Big]\right)=0\,.
		\end{equation*}
		For $I_2(\sigma)$, we have:
		\begin{align*}
			&\PP\left(L_{C(\sigma)}\in[\eta_k-r;\eta_k+r],C(\sigma)\leq\exp\Big[\frac{2}{\sigma^2}\left(\underline H_0+\delta\right)\Big]\right)\leq\PP\left(C_{\eta_k-r,\eta_k+r}(\sigma)\leq\exp\Big[\frac{2}{\sigma^2}\left(\underline H_0+\delta\right)\Big]\right)\,.
		\end{align*}
		Since $\widetilde{H}_0(\eta_k-r,\eta_k+r)>H_0(\lambda^+_0)=\underline H_0$, by taking $\delta$ sufficiently small,
		\begin{align*}
			&\PP\left(C_{\eta_k-r,\eta_k+r}(\sigma)\leq\exp\Big[\frac{2}{\sigma^2}\left(\underline H_0+\delta\right)\Big]\right)\leq\PP\left(C_{\eta_k-r,\eta_k+r}(\sigma)\leq\exp\Big[\frac{2}{\sigma^2}\left(\widetilde{H}_0(\eta_k-r,\eta_k+r)-\delta\right)\Big]\right)\,.
		\end{align*}
		Applying Lemma \ref{prison}, the above tends to $0$ as $\sigma\downarrow 0$ and it follows that $\lim_{\sigma\rightarrow 0}I_2(\sigma)=0$. Coming back to the embedding of $(\lambda^+_0+\rho,\lambda_2)$, it follows that:
		\begin{equation*}
		\lim_{\sigma\to0}\PP\left(L_{C(\sigma)}\in[\lambda^+_0+\rho,\lambda_2]\right)\le \sum_{k=1}^L	\lim_{\sigma\to0}\PP\left(L_{C(\sigma)}\in[\eta_k-r,\eta_k+r]\right)=0\,.
		\end{equation*}
		As such,
		\begin{equation*}
			\lim_{\sigma\to0}\PP\left(L_{C(\sigma)}\in[\lambda^+_0+\rho,\lambda_2]\right)=0\,,
		\end{equation*}
		which eventually (again by applying the same reasoning to $[\lambda_1,\lambda^-_0-\rho]$) yields \eqref{Collision1D-bis}.	
		
		\noindent
		We eventually complete the proof with the lower-tail estimate in \eqref{eq:victoire2}. For $\rho>0$, given the
		event~$\{\mathrm{dist}(L_{C(\sigma)},\mathcal M_0)\leq\rho\}$, remark that, for any $y$ lying in the interval $(-\rho,\rho)$, the following inclusion holds:
		\begin{equation*}
			\Big\{C(\sigma)\leq\exp\Big[\frac{2}{\sigma^2}\left(H_0(\lambda_0)-\delta\right)\Big]\Big\}\subset\Big\{C_{\lambda_0+y,\lambda_0+y}(\sigma)\leq\exp\Big[\frac{2}{\sigma^2}\left(\underline H_0-\delta\right)\Big]\Big\}\,.
		\end{equation*}
		In particular, for any $\rho>0$, and any $\lambda_0\in\mathcal M_0$,
		\begin{align*}
			&\PP\left(C(\sigma)\leq\exp\Big[\frac{2}{\sigma^2}\left(H_0(\lambda_0)-\delta\right)\Big]\right)
			\leq\PP\left(| X_{C(\sigma)}-\lambda_0|\geq\rho\right)\\
			&\quad
			+\PP\left(\inf_{y\in[-\rho,\rho]}C_{\lambda_0+y,\lambda_0+y}(\sigma)\leq\exp\Big[\frac{2}{\sigma^2}\left(H_0(\lambda_0)-\delta\right)
			\Big]\right)\,.
		\end{align*}
		According to \eqref{Collision1D-bis}, the first upper-bound vanishes as $\sigma\downarrow 0$, and the second upper-bound can be estimated by
		\begin{align*}
			&\PP\left(\inf_{y\in[-\rho,\rho]}C_{\lambda_0+y,\lambda_0+y}(\sigma)\leq\exp\Big[\frac{2}{\sigma^2}\left(H_0(\lambda_0)-\delta\right)\Big]\right) \leq\PP\left(C_{\lambda_0-\rho,\lambda_0+\rho}(\sigma)\leq\exp\Big[\frac{2}{\sigma^2}\left(H_0(\lambda_0)-\delta\right)\Big]\right)\,.
		\end{align*}
		As $\lim_{\rho\rightarrow 0}\widetilde{H}_0(\lambda_0-\rho,\lambda_0+\rho)= H_0(\lambda_0)(=\underline H_0)$, $\rho$ can be chosen sufficiently small
		so that~$\widetilde{H}_0(\lambda_0-\rho,\lambda_0+\rho)\geq H_0(\lambda_0)-\frac{\delta}{2}$. Thus, the limit~\eqref{alberto2} ensures that
		\begin{equation*}
			\lim_{\sigma\to0}\PP\left(C_{\lambda_0-\rho,\lambda_0+\rho}(\sigma)\leq\exp\Big[\frac{2}{\sigma^2}\left(H_0(\lambda_0)-\delta\right)\Big]\right)=0\,,
		\end{equation*}
		and so $\lim_{\sigma\to0}\PP\left(C(\sigma)\leq\exp\Big[\frac{2}{\sigma^2}\left(H_0(\lambda_0)-\delta\right)\Big]\right)=0$\,.
	\end{proof}
Lemma \ref{prison} and	the arguments above can be {\it straightforwardly} adapted from the self-interacting diffusions to the particle systems \eqref{MFSP1}-\eqref{MFSP2} yielding the following analogue of Theorem \ref{victoire2}.
\begin{thm}
	\label{victoire3} For $N$ sufficiently large, it holds:
	for any $\delta>0$, $1\le i\le N$,
	\begin{equation}
		\label{eq:victoire3}
		\lim_{\sigma\to0}\PP\left(\exp\left[\frac{2}{\sigma^2}\left(H_0(\lambda_0)-\delta\right)\right]<
		C^i_N(\sigma)<\exp\left[\frac{2}{\sigma^2}\left(H_0(\lambda_0)+\delta\right)\right]\right)=1\,,
	\end{equation}
	and
	\begin{equation*}
		\lim_{\sigma\to0}\PP\left(\max\big(| X^{i,N}_{C^i_N(\sigma)}-\lambda_0|;| Y^{i,N}_{C^i_N(\sigma)}-\lambda_0|\big)\leq\delta\right)=1\,.
	\end{equation*}
\end{thm}
The proof of the above readily follows the main steps of Theorem \ref{victoire2} : owing to Proposition~\ref{label_bizarre9}, Lemma
\ref{prison} still holds true for
\[
C^{i}_{N,z_1,z_2}(\sigma):=\inf\{t\ge 0\,:\,(X^{i,N}_t,Y^{i,N}_t)\notin (\mathbb R\times\mathbb R)\setminus \mathcal D_{z_1}\times
\mathcal D_{z_2}\}
\]
in place of $C_{z_1,z_2}(\sigma)$. From this, since the proof arguments of Theorem \ref{victoire2} essentially rely on the
regularity of $H_0$, one can replicate each argument replacing $C(\sigma)$ by $C^{i}_{N}(\sigma)$.	

Finally, let us point out that an analogue to Theorem \ref{victoire2} for the linear system \eqref{Linear1D} can also, naturally, be established for $(x^\sigma,y^\sigma)$.
\end{document}